\documentclass[11pt,reqno]{amsart}
\usepackage{xcolor}
\usepackage{amssymb, amsmath, amsthm,amscd,mathabx}
\usepackage[colorlinks,plainpages,backref,urlcolor=blue]{hyperref}
\hypersetup{colorlinks=true, citecolor=magenta}
\usepackage[alphabetic,lite]{amsrefs}
\usepackage{mathtools}
\usepackage{verbatim}
\usepackage{amscd}
\usepackage{tikz}
\usetikzlibrary{cd,calc,arrows}
\usepackage{mathrsfs}
\usepackage{upgreek}
\usepackage{cases}
\usepackage{comment}
\usepackage{colonequals}
\usepackage[shortlabels]{enumitem}
\usepackage[marginpar=2.4cm]{geometry}

\newcommand{\defi}[1]{{\emph{#1}}}

\renewcommand{\a}{\alpha}
\newcommand{\B}{{\bf B}}
\renewcommand{\b}{\beta}

\newcommand{\ds}{\displaystyle}

\newcommand{\bw}{\bigwedge}

\renewcommand{\ll}{\lambda}

\newcommand{\Gr}{{\bf G}}

\newcommand{\onto}{\twoheadrightarrow}
\newcommand{\oo}{\otimes}
\newcommand{\pd}{\partial}
\renewcommand{\P}{{\bf P}}

\newcommand{\Rproj}{{\bf R}}

\newcommand{\W}{{\mc W}}

\newcommand{\Z}{\mathbb{Z}}

\newcommand{\R}{\mathbb{R}}
\newcommand{\C}{\mathbb{C}}

\renewcommand{\SS}{\mathbb{S}}

\newcommand{\A}{\mathcal{A}}
\newcommand{\BB}{\mathcal{B}}
\newcommand{\NN}{\mathcal{N}}
\newcommand{\RR}{\mathcal{R}}
\newcommand{\XX}{\mathcal{X}}

\newcommand{\GL}{\operatorname{GL}}
\renewcommand{\H}{\operatorname{H}}

\newcommand{\rk}{\operatorname{rank}}

\newcommand{\Spec}{\operatorname{Spec}}
\newcommand{\Sym}{\operatorname{Sym}}

\newcommand{\Pic}{\operatorname{Pic}}

\newcommand{\Grass}{\operatorname{Gr}}
\newcommand{\ev}{\operatorname{ev}}
\newcommand{\codim}{\operatorname{codim}}
\newcommand{\coker}{\operatorname{coker}}
\renewcommand{\det}{\operatorname{det}}

\newcommand{\pr}{\operatorname{pr}}
\newcommand{\rank}{\operatorname{rank}}
\newcommand{\im}{\operatorname{im}}

\newcommand{\supp}{\operatorname{supp}}
\newcommand{\spn}{\operatorname{span}}

\newcommand{\bb}[1]{\mathbb{#1}}

\renewcommand{\rm}[1]{\textrm{#1}}
\newcommand{\mc}[1]{\mathcal{#1}}
\newcommand{\mf}[1]{\mathfrak{#1}}
\newcommand{\ol}[1]{\overline{#1}}
\newcommand{\op}[1]{\operatorname{#1}}

\newcommand{\tl}[1]{\tilde{#1}}
\newcommand{\wtl}[1]{\widetilde{#1}}
\newcommand{\ul}[1]{\underline{#1}}

\def\lra{\longrightarrow}

\newcommand{\surj}{\twoheadrightarrow}
\newcommand{\inj}{\hookrightarrow}

\newcommand{\bwedge}{\mbox{\normalsize $\bigwedge$}}
\newcommand{\bbwedge}{\mbox{\small $\bigwedge$}}
\def\dot{\mathchar"013A}  
\newcommand{\hdot}{{\raise1pt\hbox to0.35em{\!\Huge $\dot$}}} 

\newtheorem{theorem}{Theorem}[section]
\newtheorem{lemma}[theorem]{Lemma}
\newtheorem{conjecture}[theorem]{Conjecture}
\newtheorem{proposition}[theorem]{Proposition}
\newtheorem{corollary}[theorem]{Corollary}
\newtheorem*{theorem*}{Theorem}
\newtheorem*{problem*}{Problem}
\newtheorem*{claim*}{Claim}

\theoremstyle{definition}
\newtheorem{definition}[theorem]{Definition}
\newtheorem*{definition*}{Definition}
\newtheorem{example}[theorem]{Example}

\newtheorem{remark}[theorem]{Remark}
\newtheorem*{remark*}{Remark}
\newtheorem*{notation*}{Notation}

\theoremstyle{remark}
\newtheorem{case}{Case}
\newcounter{thecase}
\numberwithin{equation}{section}

\newcommand{\abs}[1]{\lvert#1\rvert}

\begin{document}

\title{The effective Chen Ranks Conjecture}

\author[M. Aprodu]{Marian Aprodu}
\address{Marian Aprodu: 
Faculty of Mathematics and Computer Science, 
\hfill \newline\texttt{}
 \indent University of Bucharest,  Romania, and
\hfill \newline\texttt{}
  \indent Simion Stoilow Institute of Mathematics\hfill \newline\texttt{}
 \indent P.O. Box 1-764,
RO-014700 Bucharest, Romania
}
\email{{\tt marian.aprodu@imar.ro}}

\author[G. Farkas]{Gavril Farkas}
\address{Gavril Farkas: Institut f\"ur Mathematik, Humboldt-Universit\"at zu Berlin \hfill \newline\texttt{}
\indent Unter den Linden 6,
10099 Berlin, Germany}
\email{{\tt farkas@math.hu-berlin.de}}

\author[C. Raicu]{Claudiu Raicu}
\address{Claudiu Raicu: Department of Mathematics,
University of Notre Dame \hfill \newline\texttt{}
\indent 255 Hurley Notre Dame, IN 46556, USA, and \hfill\newline\texttt{}
\indent Simion Stoilow Institute of Mathematics, \hfill\newline\texttt{}
\indent  P.O. Box 1-764, RO-014700 Bucharest, Romania}
\email{{\tt craicu@nd.edu}}

\author[A. Suciu]{Alexander I. Suciu}
\address{Alexander I. Suciu: Department of Mathematics,
Northeastern University \hfill \newline\texttt{}
\indent Boston, MA, 02115, USA,
and
\hfill \newline\texttt{}
  \indent Simion Stoilow Institute of Mathematics\hfill \newline\texttt{}
 \indent P.O. Box 1-764,
RO-014700 Bucharest, Romania}
\email{{\tt a.suciu@northeastern.edu}}

\begin{abstract}
Koszul modules and their associated resonance schemes are objects appearing in a variety of contexts in algebraic geometry, topology, and combinatorics. We present a proof of an effective version of the Chen ranks conjecture describing the Hilbert function of any Koszul module verifying  natural conditions inspired by geometry. We give applications to hyperplane arrangements, describing in a uniform effective manner the Chen ranks of the fundamental group of the complement of every arrangement whose projective resonance is reduced. Finally, we formulate a sharp generic vanishing conjecture for Koszul modules and present a parallel between this statement and the Prym--Green Conjecture on syzygies of general Prym canonical curves.
\end{abstract}

\maketitle

\setlength{\parskip}{6pt}

\section{Introduction}
\label{sect:main}

Originating in topology \cites{PS-imrn, PS-crelle}
in the guise of the infinitesimal Alexander invariant and taking advantage of the influential idea of formality in rational homotopy theory \cite{Sullivan}, Koszul modules turned out to be important algebraic objects on their own, being instrumental in the resolution of major open questions in algebraic geometry, like Green's Conjecture on the syzygies of a general canonical curve in arbitrary characteristic \cite{AFPRW2}. Further applications of Koszul modules to the study of Torelli groups, K\"ahler groups, Stanley-Reisner rings, or to vector bundles on algebraic varieties were presented in \cites{AFPRW, AFRSS, AFRS, AFRW}.

The original \emph{Chen Ranks Conjecture} \cite{Su-conm} predicts the large degree behavior of fundamental homotopical invariants associated to a hyperplane arrangement in terms of the linear combinatorial data of the arrangement. The conjecture has been one of the guiding problems of the field, see  \cites{CS-tams, DPS-duke, PS, SS-tams02, SS-tams, CSc-adv, Su-decomp, SW-mccool}. The main aim of this paper is to present an optimal effective version of a  Chen Ranks Conjecture for arbitrary Koszul modules.

We explain the basic setup. Let $V$ be an $n$-dimensional complex vector space and denote by 
$S\coloneqq \Sym(V)$ the symmetric algebra on $V$. We fix a subspace 
$K\subseteq \bigwedge^2 V$. The \defi{Koszul module}\/ $W(V,K)$ is 
the graded $S$-module defined as the middle homology of the complex
\begin{equation}
\label{eq:def-WVK}
\begin{tikzcd}[column sep=22pt]
K \oo S \ar[rr, "\left. \delta_2\right|_{K \oo S}"] 
&& V\oo S(1) \ar[r, "\delta_1"] & S(2),
\end{tikzcd}
\end{equation}
where $\delta_2\colon \bigwedge^2 V\otimes S \rightarrow V\otimes S(1)$ is the Koszul differential 
$(u\wedge v)\otimes f \stackrel{\delta_2}\mapsto v\otimes (u\cdot f)-u\otimes (v\cdot f)$, for 
$u, v\in V$ and $f\in S$, while $\delta_1$ is the multiplication map. The degree $q$ part 
$W_q(V,K)$ of the Koszul module is then the  vector space 
\[
W_q(V,K)=\mathrm{homology}\Bigl\{\!\!
\begin{tikzcd}[column sep=22pt]
 K\otimes \Sym^q V\ar[r,"\delta_2"] &
 V\otimes \Sym^{q+1} V\ar[r,"\delta_1"] & \Sym^{q+2} V   
\end{tikzcd}
\!\!\Bigr\}.
\]
From the exactness of the Koszul complex we obtain $W_q\bigl(V, \bigwedge^2 V\bigr)=0$; 
at the other end,  the space $W_q(V,0)\cong \ker\bigl\{V\otimes \Sym^{q+1} V\stackrel{\delta_1}
\longrightarrow \Sym^{q+2} V\bigr\}$ may be identified 
with the space $H^0\bigl(\mathbf{P}, \Omega_{\mathbf{P}}(q+2)\bigr)$ of 
twisted $1$-forms on the projective space $\P\coloneqq \mathbf{P}(V^{\vee})$.
It has been established \cite{PS-crelle} that the support of the Koszul module 
$W(V,K)$ is equal to the \emph{resonance variety} $\mathcal{R}(V,K)$ defined as the locus 
\begin{equation}
\label{eq:def-resonance}
\RR(V,K)\coloneqq\Bigl\{a\in V^\vee : \text{ there exists $b\in
V^\vee$ such that $a\wedge b\in K^\perp\setminus \{0\}$} \Bigr\}\cup \{0\}.
\end{equation}

Thus $W_q(V,K)=0$ for $q\gg 0$ if and only if $\mathcal{R}(V,K)=\{0\}$. This 
condition can be rephrased in algebro-geometric terms as  $\Gr \cap \P K^{\perp} =\emptyset$, 
where $\Gr\coloneqq \Grass_2(V^{\vee})\subseteq \P\bigl(\bigwedge^2 V^{\vee}\bigr)$ is the Grassmannian of $2$-dimensional quotients of $V$ and $K^{\perp}= \bigl(\bigwedge^2 V/K\bigr)^{\vee} \subseteq \bwedge^2 V^{\vee}$ is the orthogonal of $K$. The main result of \cite{AFPRW} is a sharp  \emph{effective} characterization of Koszul modules with trivial resonance. Precisely, one has the following equivalence 
\begin{equation}\label{eq:van_resonance}
\RR(V,K)=\{0\}\Longleftrightarrow W_{q}(V,K)=0, \ \mbox{ for all } q\geq n-3.
\end{equation}

For the Koszul module corresponding to $V=\Sym^{n-1} U$ and $K=\Sym^{2n-4} U\subseteq \bwedge^2 V$, 
where $U=\C^2$, it has been proved in \cite{AFPRW2}*{Theorem 1.7} that the vanishing \eqref{eq:van_resonance} 
is precisely the statement of Green's Conjecture \cites{GL, Vo} on the vanishing of the Koszul cohomology 
groups of syzygies of a general canonical curve of genus $2n-3$. 

\subsection{Chen ranks of Koszul modules}
Koszul modules with vanishing resonance form a restrictive class of modules, which limits 
the applicability range of \eqref{eq:van_resonance}. A typical example of geometric nature is when 
$X$ is a quasi-projective variety (e.g. a hyperplane arrangement), $V=H_1(X,\C)$, and $K^{\perp}$ is the kernel of the cup-product map 
$\bigwedge^2 H^1(X,\C)\rightarrow H^2(X,\C)$. In this case the resonance $\RR(X) \coloneqq \RR(V,K)$ 
rarely vanishes, yet it is of great interest to determine the Hilbert function of the corresponding 
Koszul module, which has important implications for the study of the fundamental group $\pi_1(X)$, 
see \cites{AFPRW, DPS-duke, PS-crelle}, or for understanding the support loci for local systems on $X$, 
see \cites{BW, GL91}. The aim of this paper is to solve this problem in a general algebraic setting
for all Koszul modules whose resonance satisfies natural conditions inspired by geometry 
and topology.

Given a subspace $K\subseteq \bigwedge^2 V$, we say that the resonance 
$\RR(V,K)$ is \emph{linear} if it is a union  
\begin{equation}\label{eq:linear resonance}
\RR(V, K)=\overline{V}_1^{\vee}\cup \cdots \cup \overline{V}_k^{\vee},
\end{equation}
of linear subspaces, where each $\overline{V}_t^{\vee}\subseteq V^{\vee}$ is 
a subspace corresponding to a quotient $V\surj \overline{V}_t$.
We say that $\RR(V,K)$ is \emph{isotropic}, if it is linear and 
$\bigwedge^2 \overline{V}_t^{\vee}\subseteq K^{\perp}$, for all $t$. 
Moreover, we say that the resonance is \emph{strongly isotropic}, if 
it is isotropic and furthermore 
\begin{equation}\label{eq:strong_iso}
\bigl(\overline{V}_t^{\vee}\wedge V^{\vee}\bigr) \cap K^{\perp}=
\bwedge^2 \overline{V}_t^{\vee}, \ \mbox{ for } t=1, \ldots, k.
\end{equation}

We refer to \cite{AFRS} for background on these concepts. When $X$ is 
a smooth quasi-projective variety, then $\RR(X)$ is always linear and 
if the mixed Hodge structure on $H^1(X, \C)$ is pure, then $\RR(X)$ is 
isotropic, cf.~\cite{DPS-duke}*{Theorem C}.

Behind the definition of strong isotropicity lies the natural scheme 
structure of a resonance variety. If $I(V,K)$ is the annihilator of 
the Koszul module $W(V,K)$, the projective scheme defined by this 
ideal is the \emph{projective resonance} 
\[
\Rproj(V,K) \coloneqq \op{Proj}\bigl(S/I(V,K)\bigr).
\]
It follows from \cite{AFRS}*{Theorem 1.1} that for an isotropic Koszul module the resonance $\RR(V,K)$ 
is strongly isotropic if and only if $\Rproj(V,K)$ is reduced. The projective resonance $\Rproj(V,K)$ 
is the image via a natural incidence correspondence of the linear section 
\begin{equation}\label{eq:linearsection}
\B(V,K)\coloneqq \Gr\cap \mathbf{P}K^\perp,
\end{equation}
of the Grassmannian $\Gr$. We identify $\B(V,K)$  with the base locus of the linear 
system $|K|$ on $\Gr$, where $K\subseteq \bwedge^2 V=H^0\bigl(\Gr, \mathcal{O}_{\Gr}(1)\bigr)$ 
and refer to Section \ref{sec:defs} for a detailed discussion.

We now state the main result of this paper.

\begin{theorem}
\label{thm:main}
Let $V$ be an $n$-dimensional complex vector space and let $K\subseteq \bigwedge^2 V$ 
be a subspace such that the resonance $\RR(V,K)=\overline{V}^{\vee}_1\cup \cdots \cup \overline{V}^{\vee}_k$ is strongly isotropic. Then 
\[
 \dim W_q(V,K) = \sum_{t=1}^k \dim\, W_q(\overline{V}_t,0),\mbox{ for  all } \  q\geq n-3.
 \]
\end{theorem}

Note that when $\RR(V,K)=\{0\}$,  Theorem \ref{thm:main} specializes to the main 
result of \cites{AFPRW2, AFPRW}, that is, to the statement $W_q(V,K)=0$ for $q\geq n-3$, 
referred to in \eqref{eq:van_resonance}. It was shown in \cite{AFPRW} that the vanishing 
\eqref{eq:van_resonance} is sharp, therefore the bound $n-3$ in Theorem \ref{thm:main} is 
sharp as well. The proof of Theorem \ref{thm:main} is based first on the fact that if $\RR(V,K)$ 
is strongly isotropic, then the base locus $\B(V,K)$ is a disjoint union of sub-Grassmannians 
\[
\B(V,K)=\Gr_1\sqcup \cdots \sqcup \Gr_k,
\]
where $\Gr_t\coloneqq  \Grass_2(\overline{V}_t^{\vee}) \subseteq \Gr$. The condition that $\B(V,K)$ 
be scheme-theoretically a disjoint union of sub-Grassmannians is a geometric manifestation 
of strong isotropicity and suffices for the proof. Under this hypothesis, 
we construct a morphism of graded $S$-modules
\begin{equation}
\label{eq:decomposition}
W(V,K)\longrightarrow \bigoplus_{t=1}^k W(V_t,0)
\end{equation}
which is an isomorphism for degrees $q\ge n-3$. This yields the claimed formula for the 
Hilbert function of the Koszul module $W(V,K)$. It also proves that the algebraic condition of strong isotropicity and the geometric condition that the base locus be a disjoint union of sub-Grassmannians are equivalent, specifically:

\begin{corollary}
    \label{cor:strongly-isotropic-equivalence}
    Let $K\subseteq \bigwedge^2 V$ be as above such that $\RR(V,K)=\overline{V}^{\vee}_1\cup \cdots \cup \overline{V}^{\vee}_k$ is isotropic. Then $\RR(V,K)$ is strongly isotropic if and only if we have the scheme-theoretic equality\[
\B(V,K)=\Gr_1\sqcup \cdots \sqcup \Gr_k.
\]
\end{corollary}

For Corollary \ref{cor:strongly-isotropic-equivalence}, we use \cite{AFRS}*{Corollary 5.2} stating that an isotropic projective resonance with disjoint irreducible components is separable if and only if it is reduced. If $\B(V,K)$ is finite, the converse implication can be proved directly, see Proposition \ref{prop:finite-reduced} .

\begin{corollary}\label{cor:regularity}
Let $K\subseteq \bigwedge^2 V$ be as above, such that $\RR(V,K)$ is strongly isotropic. 
Then the Castelnuovo--Mumford regularity of $W(V,K)$ is at most $n-3$.    
\end{corollary}
Corollary \ref{cor:regularity} is a consequence of the proof of Theorem \ref{thm:main}, in 
particular of the properties of the morphism \eqref{eq:decomposition}. Indeed, the graded 
modules $W(V_t,0)$ are all $0$-regular, which implies that their truncations $W(V_t,0)_{\ge (n-3)}$ 
are $(n-3)$-regular, hence from \eqref{eq:decomposition} also $W(V,K)_{\ge (n-3)}$ is $(n-3)$-regular, 
therefore the regularity of $W(V,K)$ is at most $n-3$.
It would be very interesting to obtain an upper bound for the regularity of arbitrary 
Koszul modules in the absence of any assumption on their support. For such bounds in 
the case of Koszul modules associated to simplicial complexes, we refer to \cite{AFRSS}.

The proof of Theorem \ref{thm:main} is technically elaborate and we will provide an outline 
of it and the end of the Introduction. When the linear section \eqref{eq:linearsection} is \emph{finite}, we compute the 
Hilbert function of the Koszul module in the absence of any hypothesis on $\mathcal{R}(V,K)$.

\begin{theorem}
\label{thm:finite}
Let $(V,K)$ be a pair such that $\B(V,K)$  is finite of length $\ell$. Then 
    \begin{equation}
    \label{eqn:Chen_Finite_Intersection}
        \dim\, W_q(V,K)=\ell\cdot (q+1), \ \mbox{ for  }\ q\ge n-3.
    \end{equation}
\end{theorem}
For transverse intersections,  Theorem \ref{thm:finite} can be made  more precise, see Corollary \ref{thm:transversal}. Applications of Theorem \ref{thm:finite} to the fundamental groups of certain (non-K\"ahler) Calabi-Yau 3-folds are discussed in Example \ref{ex:non-Kähler}.
 
\subsection{Chen ranks of groups}
\label{subsec:chen-ranks}
In order to describe  the lower central series of groups of geometric origin, K.T. Chen \cite{Ch-ann} 
introduced certain invariants, which eventually led to the definition of Koszul modules. For a finitely 
generated group $G$, we denote by 
\[
G=\Gamma_1(G)\supseteq \cdots \supseteq \Gamma_q(G)\supseteq \Gamma_{q+1}(G)\supseteq \cdots
\]
its lower central series. Then \emph{Chen ranks} $\theta_q(G)$ of $G$ are the lower central series 
ranks of the \emph{metabelian quotient}\/ $G/G''$ of $G$, where $G''\coloneqq  \bigl[[G,G], [G,G]\bigr]$. 
Precisely, one defines 
\begin{equation}\label{def:Chenrank}
\theta_q(G)\coloneqq  \rank \Gamma_q\bigl(G/G''\bigr)/\Gamma_{q+1}\bigl(G/G''\bigr).
\end{equation}
The Koszul module of the group $G$ is then obtained by setting
\[
W(G) \coloneqq  W(V,K),
\]
where $V\coloneqq H_1(G,\C)$ and $K^{\perp}\coloneqq \ker\bigl\{\cup_G\colon \bwedge^2 H^1(G,\C) \rightarrow H^2(G,\C)\bigr\}$. Similarly, the resonance variety of $G$ is defined as $\RR(G)\coloneqq \RR(V,K)$. It is  shown in \cite{PS-imrn}, based on earlier work of Massey  \cite{Massey} that $\theta_{q+2}(G)\leq \dim W_q(G)$, with equality if the group $G$ is $1$-formal in the sense of Sullivan \cite{Sullivan}. Note that $W(G)$ being an invariant constructed out of the cup-product map $\cup_G$, it makes no distinction between $G$ and its metabelian quotient $G/G''$. 

For the free group $F_m$, as originally computed in \cite{Ch-ann} one has  
$$\theta_{q+2}(F_m)=\dim W_q\bigl(H_1(F_m, \C), 0\bigr)=
(q+1)\binom{m+q}{2+q}.$$ Applying Theorem \ref{thm:main}, we obtain the following result:

\begin{theorem}\label{thm:chen_groups}
Let $G$ be a finitely generated $1$-formal group and assume its resonance $\RR(G)$ 
is strongly isotropic. If $h_m$ denotes the number of $m$-dimensional components of $\RR(G)$, then 
\[
\theta_{q}(G)=\sum_{m\geq 2} h_m\cdot \theta_q(F_m), \ \mbox{ for } \ q\geq b_1(G)-1.
\]
\end{theorem}
In the absence of the assumption that $\RR(G)$ be strongly isotropic, as shown in \cites{CSc-adv, SW-mccool} in the case some relatives of the braid group like the upper McCool groups, one cannot expect a Chen ranks formula like in Theorem \ref{thm:chen_groups}. In this sense, Theorem \ref{thm:chen_groups} is optimal.

Theorem \ref{thm:chen_groups} can be applied to determine in an effective manner the Chen ranks of prominent groups, like the \emph{pure string motion group} $P\Sigma_n$ of those automorphisms of the free group $F_n$ mapping each generator to one of its conjugates. The homology algebra $H^*(P\Sigma_n, \mathbb C)$ has been described in \cite{JMM}, in particular $b_1(P\Sigma_n)=n(n-1)$. Using Cohen's description \cite{Cohen} of the resonance of $P\Sigma_n$ and that in this case the resonance is strongly isotropic \cite{CSc-adv}*{Theorem B}, by applying Theorem \ref{thm:chen_groups}, we obtain the following formula: 
\begin{equation}\label{eq:string_eq}
\theta_q\bigl(P\Sigma_n\bigr)=(q-1)\binom{n}{2}+(q^2-1)\binom{n}{3}, \ \mbox{ for } \ q\geq n(n-1)-1.
\end{equation}

\subsection{Suciu's Conjecture on Chen ranks of hyperplane arrangements.}
\label{subsec:chen-ranks-conj}
A major source of groups for which the assumptions of Theorem \ref{thm:chen_groups} are 
satisfied is provided by the fundamental groups of hyperplane arrangements. Chronologically, 
this was one of the main motivations for which the theory of Koszul modules has been developed. 

For a hyperplane arrangement $\A$ in $\C^m$, let
$M(\A)\coloneqq \C^{m}\setminus \bigcup_{H\in \A} H$ be the complement of the arrangement. 
We denote by $L(\A)$ the associated intersection lattice (matroid). The cohomology $H^*\bigl(M(\A), \C\bigr)$ 
is determined by the intersection lattice and is isomorphic to the \emph{Orlik--Solomon algebra}\/ 
$A(\A)=E(\A)/I(\A)$, where $E(\A)$ is the exterior algebra over the complex vector space spanned by 
the vectors $\{e_H\}_{H\in \A}$ and $I(\A)$ is the Orlik--Solomon ideal defined in terms of 
dependent subsets  of hyperplanes in $\A$, see \cite{OS}. Determining the 
fundamental group $G(\A)\coloneqq \pi_1\bigl(M(\A)\bigr)$ of the arrangement is one of the central questions in the field. Even though $G(\A)$ is not combinatorially determined 
by $L(\A)$, see \cite{Ryb},  several fundamental invariants of $G(\A)$ are of 
matroidal nature.  Papadima and Suciu \cite{PS-imrn} using the formality of $G(\A)$ showed 
that the Chen ranks $\theta_q(G(\A))$ are determined by the intersection lattice. 

The \emph{Chen ranks Conjecture} \cite{Su-conm} proposes a precise  combinatorial 
formula for $\theta_q\bigl(G(\A)\bigr)$ for $q\gg 0$ in terms of the resonance 
of $\A$.  Our Theorem \ref{thm:main} 
is both an extrapolation of Suciu's Conjecture to the case of arbitrary Koszul modules, as well as 
an optimal effective version of it.

As shown in \cites{DPS-duke}, the resonance variety $\RR(\A)$ is the tangent cone at the identity 
to the Green--Lazarsfeld set \cite{GL91}, from which it follows that each component of $\RR(\A)$ is linear and isotropic. Falk, Libgober, and Yuzvinsky \cites{FY, LY00} showed that these components correspond to certain combinatorial structures called \emph{multinets}\/ on subarrangements $\BB\subseteq \A$. 

Theorem \ref{thm:main} gives an effective answer to  Suciu's Conjecture for all arrangements $\A$ 
for which the resonance $\RR(\A)$ is strongly isotropic. We establish that important classes of 
components of $\RR(\A)$ are strongly isotropic. They include the \emph{local components}\/ corresponding 
to flats $X\in L_2(\A)$ lying on at least $3$ hyperplanes in $\A$, see Proposition \ref{prop:res-local}; 
the \emph{essential components} corresponding to a multinet on $\A$, see Theorem \ref{lem:delX-omega}, 
and for all arrangements having only double and triple points (see Theorem \ref{thm:res-arr-separable}). At the same time, we also clarify a point from the literature regarding the strong isotropicity of resonance components, showing that the non-essential case requires a more detailed analysis than was previously assumed.
We summarize our results:

\begin{theorem}
\label{thm:chen-arrs}
Let $\A$ be an arrangement such that one of the following conditions hold. 
\begin{enumerate}
\item All components of $\mc{R}(\A)$ are either local or essential. 
\item $\A$ has no $2$-flats of size greater than $3$.
\end{enumerate}
If $h_m$ is the number of components of $\mc{R}(\A)$ of 
dimension $m$, then 
	\begin{equation} \label{eq:theta-q-effective}
		\theta_q\bigl(G(\A)\bigr) = (q-1)\cdot \sum_{m\geq 2} h_m\cdot \binom{m+q-2}{q}, \ \ \mbox{ for all }\ 
        q\geq \abs{\A}-1.
	\end{equation}
\end{theorem}

In particular, as we observe in Corollary \ref{cor:chen-ranks-graphic}, this covers the case of all graphic arrangements. The fact that essential components of $\RR(\A)$ are strongly isotropic has been established by Cohen--Schenck in \cite{CSc-adv}*{Theorem 5.1}. However, their proof that the resonance of \emph{every} hyperplane arrangement $\RR(\A)$ is strongly isotropic (which would make our Theorem \ref{thm:chen-arrs} valid without any hypothesis) is incorrect. We refer to Remark \ref{rem:cohen-schenk} for details.

\subsection{Generic vanishing for Koszul modules} An important consequence of the equivalence (\ref{eq:van_resonance}) is the vanishing statement $W_{n-3}(V,K)=0$ for a general $(2n-3)$-dimensional subspace $K\subseteq \bigwedge^2 V$. The condition $W_{n-3}(V,K)\neq 0$ defines an effective divisor in the parameter space 
$\Grass_{2n-3}\bigl(\bigwedge^2 V\bigr)$ of such subspaces, which is then identified in \cite{AFRW}*{Theorem 3.4} with the \emph{Chow form} of the Grassmannian $\Gr=\Grass_2(V^{\vee})\subseteq \P\bigl(\bigwedge^2 V^{\vee}\bigr)$. This identification is then essential for establishing Green's Conjecture for generic curves \cite{AFPRW2}. 

It turns out there is an equally interesting second divisorial case for the generic vanishing of Koszul modules. Precisely, assuming $K\subseteq \bigwedge^2 V$ is an $m$-dimensional subspace, it follows from (\ref{eq:def-WVK}) that the condition $W_q(V,K)=0$ can be rewritten as requiring that the map 
\begin{equation}\label{eq:deglocus1}
\delta_2\colon K\otimes \Sym^q(V)\longrightarrow H^0\bigl(\P, \Omega_{\P}(q+2)\bigr)
\end{equation}
be of maximal rank. The failure of this condition defines a virtual divisor 
on the Grassmannian $\Grass_m\bigl(\bigwedge^2 V\bigr)$ if and only if 
the two vector spaces appearing in \eqref{eq:deglocus1} have the same 
dimension, that is, 
\[
m \binom{n+q-1}{q}=n \binom{n+q}{q+1}- \binom{n+q+1}{q+2},
\]
that is, when $q(m-n+1)=n^2-n-2m$. One immediately concludes that there are 
two cases appearing in each dimension $n=\dim(V)$, precisely when 
\[
(m, q)=(2n-3,n-3) \ \  \mbox{ respectively } \ \ (m,q)=(2n-2,n-4).
\]
The first case having been treated in \cites{AFPRW, AFRW}, we ask now whether 
also in the second case the morphism $\delta_2$ in \eqref{eq:deglocus1}
is generically of maximal rank:
 
\begin{conjecture}\label{conj:Koszul_vanishing}
Let $n\geq 6$. For a general $(2n-2)$-dimensional subspace 
$K\subseteq \bigwedge^2 V$, one has the vanishing 
$W_{n-4}(V,K)=0$.
\end{conjecture}
We explain in Section \ref{sect:Koszul_van} how from Kronecker's classical theory of pencils of skew-symmetric matrices it follows that for $n=5$, one has $W_1(V,K)\neq 0$, for every $8$-dimensional subspace $K\subseteq \bigwedge^2 V$. We can verify the Conjecture \ref{conj:Koszul_vanishing} for all $6\leq n \leq 10$, with the exception $n=9$.

\begin{theorem}\label{thm:surprise_9}
If $V$ is a $n$-dimensional vector space with $6\leq n\leq 10$ but $n\neq 9$, then $W_{n-4}(V,K)\neq 0$, for a general $(2n-2)$-dimensional subspace $K\subseteq \bigwedge^2 V$.   
\end{theorem}

A Macaulay2 calculation for random $16$-dimensional subspaces $K\subseteq \bigwedge^2 V$, 
when $V\cong \C^9$,  strongly suggest that $W_{5}(V,K)\neq 0$, for \emph{every} 
such subspace $K\in \Grass_{16}\bigl(\bigwedge^2 V\bigr)$.

It is very tempting to draw a parallel between the highly surprising failure 
for $n=9$ (and $n=5$)  of Conjecture \ref{conj:Koszul_vanishing} and the equally 
surprising failure of the \emph{Prym-Green Conjecture} for curves of genus $16$ 
(and $8$), see \cites{CEFS, CFVV}. This conjecture, predicting the vanishing 
\[
K_{\frac{g}{2}-3,2}(C, \omega_C\otimes \eta)=0
\]
for a general Prym curve $[C, \eta]$ of genus $g$, with $\eta\in JC[2]$ being a $2$-torsion point, is essential in showing \cite{FL} that the moduli space $\overline{\mathcal{R}}_g$ of stable Prym curves of genus $g$ is of general type for $g\geq 13$ and $g\neq 16$. The Prym--Green conjectures fails on $\overline{\mathcal{R}}_8$ as explained in \cite{CFVV}. A Macaulay2 calculation with general rational $g$-nodal curves strongly indicates \cite{CEFS}*{Proposition 4.4} that $K_{5,2}(C, \omega_C\otimes \eta)\neq 0$, for every $[C,\eta]\in \mathcal{R}_{16}$. In light of the proof of the generic Green Conjecture using Koszul modules in \cite{AFPRW2}, we expect a link between the failure of the Prym-Green Conjecture in genus $16$ and the failure of Conjecture \ref{conj:Koszul_vanishing} for $16$-dimensional subspaces $K\subseteq \bigwedge^2 \C^9$.

\subsection{Outline of the proof of Theorem \ref{thm:main}}\label{subsec:outline}
We end the Introduction with an  outline of the proof of Theorem \ref{thm:main} and we focus on the construction and the properties of the map  
\[
W(V,K)\longrightarrow \bigoplus_{t=1}^k W(\overline{V}_t,0).
\]
If $\gamma \colon \wtl{\Gr}\rightarrow \Gr$ is the blow-up of the Grassmannian 
$\Gr=\Grass_2(V^{\vee})$ along $\B(V,K)$, let $\mc{O}_{\wtl{\Gr}}(H)=\gamma^*\mc{O}_{\Gr}(1)$ be the corresponding pullback of the Pl\"ucker line bundle.

Let $E = E_1 \sqcup \cdots \sqcup E_k$ be the exceptional 
divisor of $\gamma$, with $E_t= \op{\ul{Proj}}\bigl(\Sym(U_t \oo \mc{Q}_t^{\vee})\bigr)$, 
with $U_t = \ker\bigl\{\pi_t\colon V\to \ol{V}_t\bigr\}$ and let $\mc{Q}_t$ be the tautological rank $2$
quotient on $\Gr_t=\Grass_2(\overline{V}_t^{\vee})$. The base locus of the linear system 
$\gamma^*|K|$  being equal to $E$, we have an evaluation  map 
\[
\begin{tikzcd}[column sep=18pt]
K\oo \mc{O}_{\wtl{\Gr}} \ar[r, two heads]& \mc{O}_{\wtl{\Gr}}(H-E),
\end{tikzcd}
\]
which gives rise to an associated exact Koszul complex
$\wtl{\mc{K}}^{\bullet}$ on $\wtl{\Gr}$, where
\begin{equation}
\label{eq:def-tlK}
 \wtl{\mc{K}}^{-i} = 
 \begin{cases}
 \bw^i K \oo \mc{O}_{\wtl{\Gr}}\bigl((1-i)(H-E)\bigr) &\text{ for $i\geq 0$} ,
 \\
0 &\text{ for $i<0$}.
 \end{cases}
\end{equation}
Working with the  blowup $\wtl{\Gr}$ rather than with $\Gr$ is necessary to obtain an \emph{exact} complex of vector bundles, since the Koszul complex on $\Gr$ associated 
to the map \eqref{eqn:ev_K} is not exact.

We modify the complex $\wtl{\mc{K}}^\bullet$ and obtain a complex $\mc{K}^\bullet$, by setting 
$\mc{K}^0 \coloneqq \wtl{\mc{K}}^0(E) = \mc{O}_{\wtl{\Gr}}(H)$, and 
$\mc{K}^i = \wtl{\mc{K}}^i$ for $i\neq 0$. Denoting by   
$\gamma_t\colon E_t \to \Gr_t$ the restriction of $\gamma$ to $E_t$, 
the only non-zero cohomology group of $\mc{K}^{\bullet} \oo_{\mc{O}_{\wtl{\Gr}}}  \gamma^*\Sym^q\mc{Q}$ is
\begin{equation}
\label{eq:H0-K}
\mc{H}^0\bigl(\mc{K}^{\bullet}\oo_{\mc{O}_{\widetilde{\Gr}}}  \gamma^*\Sym^q\mc{Q}\bigr) = 
\gamma^*\Sym^q\mc{Q}\otimes \mc{O}_E(H) \cong 
\bigoplus_{t=1}^k \gamma_t^*\bigl((\Sym^q\mc{Q}_t)(1)\bigr).
\end{equation}

It turns out that the sheaf appearing in \eqref{eq:H0-K} has vanishing
higher cohomology, and 
\begin{equation}
\label{eq:H0-E-gamma}
H^0\bigl(E,\gamma^*\Sym^q\mc{Q}\otimes \mc{O}_E(H)\bigr) \cong 
\bigoplus_{t=1}^k W_q(\overline{V}_t,0)\, .
\end{equation}

The computational details for the identification \eqref{eq:H0-E-gamma} 
involve  Bott  vanishing and the Leray spectral sequences and are presented  
in Section \ref{sec:proof-main-2}.
The link between $W_q(V,K)$ and $\bigoplus_{t=1}^k W_q(\overline{V}_t,0)$ is now provided
by the observation that one can regard 
$\mc{K}^{\bullet}\oo \gamma^*\Sym^q\mc{Q}$ as a resolution
of the sheaf $\gamma^*\Sym^q\mc{Q}\otimes \mc{O}_E(H)$, therefore the corresponding hypercohomology 
is given by
\[
\bb{H}^j(\mc{K}^{\bullet}\oo_{\mc{O}_{\wtl{\Gr}}} \gamma^*\Sym^q\mc{Q}) \cong
H^j\bigl(E,\gamma^*\Sym^q\mc{Q}\otimes \mc{O}_E(H)\bigr), \mbox{ for all } j.
\]

In turn, the hypercohomology groups of the complex 
$\mc{K}^{\bullet}\oo \gamma^*\Sym^q\mc{Q}$ are 
computed by a spectral sequence of graded $S$-modules,
\begin{align}
E_1^{-i,j} = &H^j\bigl(\wtl{\Gr},\mc{K}^{-i}\oo_{\mc{O}_{\wtl{\Gr}}}
\gamma^*\Sym^q\mc{Q}\bigr) \Longrightarrow \\
&\hspace*{1in} \bb{H}^{-i+j}(\mc{K}^{\bullet}
\oo_{\mc{O}_{\wtl{\Gr}}} \gamma^*\Sym^q\mc{Q} 
\cong H^{j-i}\bigl(E,\gamma^*\Sym^q\mc{Q})\otimes \mc{O}_E(H)\bigr), \notag
\end{align}
and one has $E_1^{-i,j} = 0$ for $i<0$ or $j<0$. Writing, as usual,
\begin{equation}
\label{eq:maps-drij}
\begin{tikzcd}[column sep=18pt]
d_r^{i,j}\colon E_r^{i,j} \ar[r]& E_r^{i+r,j-r+1}
\end{tikzcd}
\end{equation}
for the differentials on the $r$-th page of the spectral sequence, we then show in Section \ref{sec:proof-main-2} that   $E_2^{0,0} = \coker(d_1^{-1,0}) = W_q(V,K)$. 

The natural map  $W_q(V,K) \rightarrow \bigoplus_{t = 1}^k W_q(V_t,0)$
 is then obtained as the composition
\begin{equation}
\label{eq:E2-Einf-edge}
\begin{tikzcd}[column sep=20pt]
E_2^{0,0} \ar[r, two heads]& E_{\infty}^{0,0} \ar[r, hook]&
H^0\bigl(E,\gamma^*\Sym^q\mc{Q}\otimes \mc{O}_E(H)\bigr)  ,
\end{tikzcd}
\end{equation}
where the first map is the quotient map (the differentials $d_r^{0,0}$
vanish for all $r$, since their target is $0$, so $E_{r+1}^{0,0}$ is a 
quotient of $E_r^{0,0}$ for all $r$), and the second map is an edge 
homomorphism for the spectral sequence. In order to show that this composition 
is an isomorphism for $q\geq n-3$ (which is precisely the conclusion of 
Theorem~\ref{thm:main}), we verify in  
Section~\ref{sec:proof-main-2} that if $q\geq n-3$, then 
$E_2^{-i,i} = E_2^{-i-1,i} = 0\mbox{ for all }i\neq 0$. 

In fact, we shall prove in Section \ref{sec:proof-main-2} a 
slightly more general vanishing result, namely 
\begin{equation}
 \label{eq:vanishing-E2-ij}
E_2^{-i,j} = 0, \  \mbox{ if } j>0, \ i\leq j+1  \mbox{ and } q\ge n-3.
\end{equation}
This statement is equivalent to the exactness in the 
middle of the following complex
\[
H^j\bigl(\wtl{\Gr},\mc{K}^{-i-1} \oo \gamma^*\Sym^q\mc{Q}\bigr) \longrightarrow H^j\bigl(\wtl{\Gr},\mc{K}^{-i} \oo 
\gamma^*\Sym^q\mc{Q}\bigr) \longrightarrow 
H^j\bigl(\wtl{\Gr},\mc{K}^{-i+1} \oo
\gamma^*\Sym^q\mc{Q}\bigr).
\]

A large part of the effort in the paper is oriented towards understanding the cohomology 
groups $H^j\bigl(\wtl{\Gr}, \gamma^*(\Sym^q\mc{Q})\bigl((1-i)(H-E))\bigr)$. 
Our results can be regarded as a form of Bott vanishing for the blow-up of Grassmannians 
along sub-Grassmannians and they are of interest on their own.

{\small{{\bf{Acknowledgments:}} Aprodu  was partly supported by the project
"Cohomological Hall algebras of smooth surfaces and applications" - CF 44/14.11.2022, call no. 
PNRR-III-C9-2022-I8 funded by
the European Union - NextGenerationEU - through Romania’s National Recovery 
and Resilience Plan.
Farkas was supported by the Berlin mathematics research center MATH+ and by the
ERC Advanced Grant SYZYGY (grant agreement No. 834172).
Raicu was supported by the NSF Grant No.~2302341.
Suciu was supported in part by Simons Foundation Collaboration Grants for
Mathematicians  No.~354156 and  No.~693825 and by the project 
``Singularities and Applications" - CF 132/31.07.2023 funded by
the European Union - NextGenerationEU - through Romania’s National Recovery 
and Resilience Plan.
}}

\section{Koszul modules and their resonance schemes}
\label{sec:defs}

We recall basic definitions on Koszul modules following \cites{AFPRW, AFRS}. We fix  
an $n$-dimensional complex vector space $V$ and set $S\coloneqq \Sym(V)$. 
We also fix a subspace $K\subseteq \bwedge^2 V$.

Let $\widetilde{\delta}_3\colon \bwedge^3V\otimes S(-1)\to (\bwedge^2V/K)\otimes S$ 
be the map induced by the third differential in the Koszul complex 
$\bigl(\bwedge^{\bullet} V \otimes S, \delta\bigr)$. Then the \emph{Koszul module} 
$W(V,K)$ is defined as $\coker(\widetilde{\delta}_3)$, that is, it is the graded 
$S$-module admitting the following presentation
\begin{equation}
    \label{eqn:presentation_W}
    \begin{tikzcd}[column sep=18pt]
    \bwedge^3V\otimes S(-1)\ar[r]& \Big(\bwedge^2V/K\Big)\otimes S 
    \ar[r]&  W(V,K) \ar[r]& 0.
    \end{tikzcd}
\end{equation}
Since $W(V,K)$ is a graded $S$-module, it gives rise to a 
coherent sheaf $\W(V,K)$ on the projective space $\P\coloneqq 
\mathbf{P}(V^{\vee})$. We refer to this  as the \defi{Koszul sheaf} 
of the pair $(V,K)$, which, accordingly, admits the following presentation 
\begin{equation}
    \label{eqn:presentation_W_sheaf}
        \begin{tikzcd}[column sep=18pt]
    \bwedge^3V\otimes \mathcal{O}_\mathbf{P}(-1)\ar[r]&  \left(\bwedge^2V/K\right)\otimes \mathcal{O}_\mathbf{P}\ar[r]&  \W(V,K)\ar[r]&  0.
    \end{tikzcd}
\end{equation}
Writing $\Omega \coloneqq \Omega^1_{\P}$ for the 
sheaf of $1$-differentials on $\P$, recall the Euler sequence
\[
0\longrightarrow \Omega\longrightarrow V\otimes \mathcal{O}_{\P}(-1)\longrightarrow 
\mathcal{O}_{\P}\longrightarrow 0.
\]
By twisting this sequence and taking global sections, we find $W_q(V,0)\cong H^0\bigl(\P,\Omega(q+2)\bigr)$. 
In particular, $\dim W_q(V,0)=(q+1)\binom{n+q}{2+q}$. Using \eqref{eq:def-WVK}, we can also write
\begin{equation}
\label{eq:WVK-quot-WV0}
W_q(V,K) = \coker\Bigl\{ K \oo \Sym^q V \stackrel{\delta_2}\longrightarrow  W_q(V,0)\Bigr\},
\end{equation}
where $\delta_2\bigl(u\wedge v\otimes f)\coloneqq v\otimes (u\cdot f)-u\otimes (v\cdot f)$, 
for $u, v, f\in V$. The  inclusion $K\subseteq \bigwedge^2 V$ together with the surjection 
$\bw^2 V \oo \mc{O}_{\P}(-2) \onto \Omega$ induce a sheaf morphism $K\oo \mc{O}_{\P}(-2) \to \Omega$. 
The Koszul sheaf $\mathcal{W}(V,K)$ defined via the sequence \eqref{eqn:presentation_W_sheaf} 
can then also be realized as
\[
\begin{tikzcd}[column sep=18pt]
\W(V,K)= \coker\bigl\{K\oo \mc{O}_{\P} \ar[r]& \Omega(2)\bigr\}.
\end{tikzcd}
\]

\subsection{Resonance schemes}
\label{subsec:resvars}
It has been proved in \cite{PS-crelle}*{Lemma 2.4}, then used in \cites{AFPRW2, AFPRW, AFRS}, 
that the set-theoretic support of the $S$-module $W(V,K)$ is given by the 
\emph{resonance variety} $\mathcal{R}(V, K)$ defined in \eqref{eq:def-resonance}.

If $I(V,K)$ is the annihilator of the Koszul module $W(V,K)$, following \cites{AFRS}, 
the \defi{affine resonance scheme}\/ is the scheme-theoretic support of $W(V,K)$ 
inside $V^{\vee}$, that is, 
\begin{equation}
\label{eq:def-res-affine-scheme}
\RR(V,K) \coloneqq \op{Spec}\bigl(S/I(V,K)\bigr) .
\end{equation}

The scheme-theoretic support of the Koszul sheaf $\W(V,K)$ is then the 
\defi{projective resonance scheme}, denoted by $\Rproj(V,K) \coloneqq \op{Proj}\bigl(S/I(V,K)\bigr)$, 
see \cite{AFRS}*{\S2}.

The projectivized resonance can be described in terms of the 
geometry of the Grassmannian $\Gr\coloneqq \Grass_2(V^\vee)\subseteq 
\mathbf{P}\big( \bwedge^2V^\vee\big)$ in its Pl\"ucker embedding.  Consider the diagram

\begin{equation}
\begin{tikzcd}[column sep=24pt]
 \mathbf{P}\times \mathbf{G}
& \Xi\ar[r, "\pr_2"] \ar[d, "\pr_1"] \ar[l, hook']
&\mathbf{G} \ar[r, hook]
&\mathbf{P}\big(\!\bwedge^2V^\vee\big) ,\\
 &\mathbf{P} &&
\end{tikzcd}
\end{equation}
where $\Xi=\bigl\{(p,L)\in \mathbf{P} \times \mathbf{G} : p\in L\bigr\}$ 
is the incidence variety. The projection $\pr_2$ 
realizes $\Xi$ as the projectivization of the universal rank-$2$ bundle $\mathcal{Q}$ 
on $\mathbf{G}$. As explained in \cite{AFRS}*{Lemma 2.5}, set-theoretically we have 
\begin{align}
\label{eq:rvk-pr}
\mathbf{R}(V,K)=
\pr_1\bigl(\pr_2^{-1}(\mathbf{G}\cap\mathbf{P}K^\perp)\bigr) \
\mbox{ and } \ 
\mathbf{G}\cap\mathbf{P}K^\perp\subseteq 
\pr_2\bigl(\pr_1^{-1}(\mathbf{R}(V,K))\bigr).
\end{align}

The Koszul module $W(V, K)$ can be computed working directly on 
$\mathbf{G}$, as follows. Since $H^0(\Gr, Q)\cong V$, we consider the evaluation map
$
  V\otimes\mathcal{O}_\mathbf{G}\twoheadrightarrow\mc{Q}  
$ and we introduce the sheaf of graded rings on $\Gr$ given by 
$\mc{S} = \Sym_{\mc{O}_{\Gr}}(\mc{Q}) = \bigoplus_{q\geq 0}\Sym^q\mc{Q}$.
We then have an isomorphism of graded $S$-modules, see also \cites{AFPRW2, AFPRW} 
\[
\begin{tikzcd}[column sep=18pt]
W(V,K)\cong\coker\Bigl\{K\oo H^0(\Gr,\mc{S})\ar[r]& H^0\bigl(\Gr,\mc{S}(1)\bigr)\Bigr\}.
\end{tikzcd}
\]
The inclusion $K\subseteq \bigwedge^2 V = H^0(\mathbf{G},\mc{O}_{\Gr}(1))$ 
induces the evaluation morphism
\begin{equation}
\label{eqn:ev_K}
\begin{tikzcd}[column sep=18pt]
\ev_K\colon K \otimes\mathcal{O}_\mathbf{G}\ar[r]
& \mc{O}_{\Gr}(1) .
\end{tikzcd}
\end{equation}
The scheme-theoretical intersection $\B(V,K)\coloneqq\mathbf{G}\cap \mathbf{P}K^\perp$ 
is then the scheme-theoretic support of the cokernel of this map, that is, 
$\B(V,K)=\supp\bigl(\coker(\ev_K)\bigr)$. Indeed, $K$ generates both the ideal of $\mathbf{P}K^\perp$ and the 
ideal in $S$ of the image of the evaluation map 
\eqref{eqn:ev_K}. This scheme structure may be non-reduced, 
as shown in the following example.

\begin{example}
\label{ex:non-red-bvk}
Let $V^{\vee}$ be the $4$-dimensional space with basis $(e_1,\dots,e_4)$ and 
consider $K^\perp=\spn\{e_1\wedge e_2,e_1\wedge e_3+e_2\wedge e_4\}$. 
We denote by $(v_1,\dots,v_4)$ the dual basis of $V$ and 
by $X_{i,j}\coloneqq v_i\wedge v_j$ the Pl\"ucker coordinates on 
$\mathbf{P}(\bwedge^2V^\vee)$. 
Then $\B(V,K)$ is non-reduced. Indeed, the equations of $\mathbf{P}K^\perp$ are $X_{1,4} = X_{2,3} = X_{3,4} = X_{1,3}-X_{2,4} = 0$ and the equation of  $\Gr$ is  the quadric $X_{1,2}X_{3,4} - X_{1,3}X_{2,4} + X_{1,4}X_{2,3} = 0$. The ideal of the intersection is $\langle X_{1,4}, X_{2,3}, X_{3,4}, X_{1,3}-X_{2,4}, X_{1,3}^2\rangle $ and this scheme is a double point supported at $[e_1\wedge e_2]$. Note also that $\mathbf{R}(V,K)$ is the double line  $\spn\{e_1,e_2\}$.
\end{example}

\begin{example}
\label{ex:embedded-comp}
Let $V^{\vee}=\spn\{e_1, \ldots, e_5\}$ be $5$-dimensional  and consider the subspace 
$K^\perp=\spn\{e_1\wedge e_2,e_1\wedge e_3,e_1\wedge e_4,e_2\wedge e_4,e_2\wedge e_5+e_3\wedge e_4\}$. 
Then $\mathbf{R}(V,K)$ consists of a minimal component $\{v_5=0\}$ and an embedded component, 
$\{v_3=v_5^2=0\}$. Using \cite{PS-crelle}*{Proposition 6.2}, 
the group $$G=\bigl\langle x_1, x_2,x_3,x_4,x_5: [x_1,x_5], [x_2,x_3], [x_3,x_5], [x_4,x_5], [x_2,x_5][x_4,x_3]\bigr\rangle$$ can be shown to have resonance $\RR(G)=\RR(V,K)$.
\end{example}

All these constructions are functorial. If $\ol{V}^{\vee}\subseteq V^{\vee}$ is a subspace, 
let $\pi\colon V\surj \ol{V}$ be the dual map and set $\ol{K}\coloneqq \bw^2 \pi(K)$. 
We obtain a surjective morphism of graded $S$-modules 
$W(V,K) \twoheadrightarrow W(\ol{V},\ol{K})$, where $W(\ol{V},\ol{K})$ is 
seen as an $S$-module by restriction of scalars. By sheafification, we get a surjective 
morphism of  sheaves $\mathcal{W}(V,K) \twoheadrightarrow \iota_*\mathcal{W}(\ol{V},\ol{K})$,
where $\iota\colon \mathbf{P}(\overline{V}^\vee)\to \mathbf{P}(V^\vee)$ 
denotes the inclusion. Also, $\ol{K}^\perp = K^\perp\cap \textstyle{\bigwedge}^2\ol{V}^\vee$ 
and we have the following commuting diagram of projective schemes:
\begin{equation}
   \label{eq:res-base-diagram}
   \begin{tikzcd}[column sep=28pt]
    \mathbf{R}(\ol{V},\ol{K}) \ar[d] &\ol{\Xi}\ar[l, "\pr_1"'] 
    \ar[r, "\pr_2"]\ar[d,hookrightarrow] 
    & \B(\ol{V},\ol{K})\ar[d] \phantom{.}
       \\
    \mathbf{R}(V,K) &\Xi\ar[l, "\pr_1"'] \ar[r, "\pr_2"] & \B(V,K) 
   \end{tikzcd}
\end{equation}

\begin{remark}\label{rem:reduceness2}
 Being reduced is not a functorial property. It might be that $\mathbf{R}(\ol{V},\ol{K})$ is reduced, while $\mathbf{R}(V,K)$ is not reduced along $\P(\ol{V}^\vee)$. 
 This is illustrated in Example \ref{ex:non-red-bvk}. The resonance is supported on the line  $\spn\{e_1,e_2\}$, and the scheme structure is non-reduced.
 If $\ol{V}^\vee=\spn\{e_1,e_2, e_3\}$, then $\ol{K}^\perp = K^\perp\cap \textstyle{\bigwedge}^2\ol{V}^\vee$ is spanned by $e_1\wedge e_2$. It is easy to verify that the corresponding line in $\P(\ol{V}^\vee)$ which coincides with $\mathbf{R}(\ol{V},\ol{K})$ is reduced. In conclusion, if a linear irreducible component of $\mathbf{R}(V,K)$ is contained in some $\mathbf{P}(\ol{V}^\vee)$ and is a reduced component of $\mathbf{R}(\ol{V},\ol{K})$, it does not necessarily imply that it is a reduced component of $\mathbf{R}(V,K)$.
In Section \ref{sect:iso-sep}, building on \cite{AFRS}, we point out that the scheme structure 
of $\B(V,K)$ is reduced if the resonance 
satisfies a supplementary conditions (strong isotropicity).
\end{remark}

\section{Cohomological tools and first applications}

\subsection{Grassmannians and Bott's theorem}
\label{sec:rep-thy}

We recall basic facts about Bott's vanishing theorem on Grassmannians 
which Will be used throughout the paper. We write $\Z^n_{\rm{dom}}$ 
for the set of \defi{dominant weights} in 
$\Z^n$, that is, the set of $n$-tuples $\ll=(\ll_1,\dots,\ll_n)\in\Z^n$ 
with $\ll_1\geq\ll_2\geq\cdots\geq\ll_n$. When each $\ll_i$ is non-negative, 
we identify $\ll$ with a \defi{partition} with (at most) $n$ parts, and 
write $\ll\in\bb{N}^n_{\rm{dom}}$. When $\ll\in\Z^n$ is not dominant,
it must contain \defi{inversions}, i.e., pairs $(i,j)$ with $i<j$ and $\ll_i<\ll_j$.
The \defi{size} of $\ll$ is $\abs{\ll} = \ll_1+\cdots+\ll_n$. If $\ll$ is a partition,
we write $\ll'$ for the \defi{conjugate} partition, where $\ll'_i$ counts the
number of parts $\ll_j$ with $\ll_j\geq i$. We write $(b^a)$ for the sequence
$(b,\dots,b)$, where $b$ is repeated $a$ times.

If $V$ is a complex vector space of dimension $n$, and if
$\ll\in\Z^n_{\rm{dom}}$, we write $\SS_{\ll}V$ for the corresponding irreducible
representation of $\GL(V)$. If $\ll=(d,0,\dots,0)$, then $\SS_{\ll}V = \Sym^d V$,
and if $\ll=(1^n)$, then $\bb{S}_{\ll}V=\bw^n V$. More generally, one can define
$\SS_{\ll}\mc{E}$ for any locally free sheaf $\mc{E}$ of rank $n$ on an algebraic
variety $X$ over $\C$. 

If $N\geq 0$ and $\mc{U}_1$, $\mc{U}_2$ are two complex vector spaces (or more generally,
locally free sheaves on some variety $X$ over $\C$), then \defi{Cauchy's formula} 
\cite{weyman}*{Theorem 2.3.2} gives a decomposition
\begin{equation}
\label{eq:cauchy}
 \bwedge^N\bigl(\mc{U}_1 \oo \mc{U}_2\bigr) = \bigoplus_{{\ll}\vdash  N}
 \SS_{\ll}\mc{U}_1 \oo \SS_{\ll'}\mc{U}_2\, ,
\end{equation}
where the tensor product is considered over $\C$ (or over $\mc{O}_X$).

We write again $\Gr = \op{Gr}_2(V^{\vee})$ 
and consider the tautological exact sequence
\begin{equation}
\label{eq:uv-exact}
\begin{tikzcd}[column sep=18pt]
0 \ar[r]& \mc{U} \ar[r]& V \oo \mc{O}_{\Gr} \ar[r]& \mc{Q} \ar[r]& 0\, ,
\end{tikzcd}
\end{equation}
where $\mc{Q}$ is the tautological rank $2$ quotient of $V$ and $\mc{U}$
is the rank $n-2$ tautological subbundle. As usual, $\mc{O}_\Gr(1)\cong \bw^2\mc{Q}$ is the Pl\"ucker line bundle realizing the embedding $\Gr\hookrightarrow \mathbf{P}\bigl(\bw^2 V^{\vee}\bigr)$. 
We then have
\[
H^i(\Gr,\mc{O}_\Gr(1)) = 
\begin{cases}
\bbwedge^2 V & \text{if $i=0$},
\\
0 & \text{otherwise}. 
\end{cases}
\]

More generally, Bott's theorem describes the cohomology groups 
of sheaves of the form 
$\SS_{\a}\mc{Q}\oo\SS_{\b}\mc{U}$,
where $\a=(\a_1,\a_2)$ and $\b=(\b_1,\dots,\b_{n-2})$ are 
dominant weights. Note that $\SS_{\a}\mc{Q}=(\Sym^{\a_1-\a_2}\mc{Q})(\a_2)$. We let 
$\gamma=(\a|\b)$ denote the concatenation of $\a$ and 
$\b$. Let $\delta=(n-1,\dots,0)$ and set
$\gamma+\delta\coloneqq  (\gamma_1+n-1,\gamma_2+n-2,\dots,\gamma_n)$. 
We write $\rm{sort}(\gamma+\delta)$ for the sequence obtained by arranging
the entries of $\gamma+\delta$ in non-increasing order, and put
\begin{equation}
\label{eq:tildegamma}
\tl{\gamma}=\rm{sort}(\gamma+\delta)-\delta.
\end{equation}

\begin{theorem}[\cite{weyman}*{Corollary~4.1.9}]
\label{thm:bott}
With the above notation, if $\gamma+\delta$ has repeated entries
(or equivalently, if $\gamma_i-i=\gamma_j-j$ for some 
$i\neq j$), then
$H^i\bigl(\Gr,\SS_{\a}\mc{Q}\oo \SS_{\b}\mc{U}\bigr)=0$ 
for all $i$. Otherwise, writing $\ell$ for the number of inversions 
in $\gamma+\delta$, we have that
\[
H^i\bigl(\Gr,\SS_{\a}\mc{Q}\oo
\SS_{\b}\mc{U}\bigr)=
\begin{cases}
\SS_{\tl{\gamma}}V & \text{if $i=\ell$},
\\
0 & \text{otherwise}.
\end{cases}
\]
\end{theorem}

We record two consequences of Bott's theorem that will be useful later.

\begin{lemma}
\label{lem:q-dd-beta-vanishing}
Suppose $0\leq d\leq n-3$ and $q>d$. If $\a=(q-d,-d)$ and if 
$\b\vdash d+1$, then
$H^{d+1}\bigl(\Gr,\SS_{\a}\mc{Q}\oo \SS_{\b}\mc{U}\bigr)=0$. 
\end{lemma}

\begin{proof}
Let $\gamma = (q-d,-d,\b_1,\dots,\b_{n-2})$.  If $d=0$, then,
since $\abs{\b}=d+1=1$, we have $\b_1=1$, therefore
$\gamma_2-2 = \gamma_3-3$.  Hence,
$H^{d+1}(\Gr,\SS_{\a}\mc{Q}\oo \SS_{\b}\mc{U})=0$ by
Theorem~\ref{thm:bott}. We therefore assume that $d>0$ and
$H^{d+1}(\Gr,\SS_{\a}\mc{Q}\oo \SS_{\b}\mc{U})\neq 0$,
and seek a contradiction. By Theorem~\ref{thm:bott}, the integers
$\gamma_i-i$ for $i=1,\dots,n$ must be pairwise distinct. We then have
\[
\gamma_3-3>\cdots>\gamma_{d+2}-(d+2) \geq -(d+2)=\gamma_2-2.
\]
Since the numbers $\gamma_i-i$ are distinct, the last inequality above
must be strict, and therefore $\b_d\geq 1$. If $\b_{d+1}>0$ then the
condition $\abs{\b}=d+1$ implies $\b=(1^{d+1})$, and in particular
$\gamma_{d+3}=\b_{d+1}=1$, so $\gamma_2-2=\gamma_{d+3}-(d+3)$,
which is a contradiction. We must then have $\b_{d+1}=0$, in which
case $\b=(2,1^{d-1})$. Therefore,
\begin{align*}
\gamma_1-1 & = q-d-1>-1=\gamma_3-3>\cdots > \gamma_{d+2}-(d+2)\\
&>\gamma_2-2=-(d+2)>-(d+3)=\gamma_{d+3}-(d+3)>\cdots
\end{align*}
so that $\gamma+\delta$ has precisely $d$ inversions, namely
$(2,3),(2,4),\dots,(2,d+2)$. Applying Theorem~\ref{thm:bott},
we find that $H^{d+1}\bigl(\Gr,\SS_{\a}\mc{Q}\oo \SS_{\b}\mc{U}\bigr)=0$,
again a contradiction.
\end{proof}

\begin{lemma}
\label{lem:Sa-Q-vanishing}
 Let $\a=(\a_1,\a_2)\in\Z^2_{\mathrm{dom}}$. Then
  $H^j(\Gr,\SS_{\a}\mc{Q}) = 0\mbox{ if }j\not\in\{0,n-2,2n-4\}.$
Moreover, we have that
\begin{equation}
\label{eq:Sa-Q-only-0}
 H^j(\Gr,\SS_{\a}\mc{Q}) = 0\mbox{ for }j\neq 0\mbox{ if }\a_2\geq-(n-2),\mbox{ and}
\end{equation}
\begin{equation}
\label{eq:Sa-Q-only-n2}
 H^j(\Gr,\SS_{\a}\mc{Q}) = 0\mbox{ for }j\neq n-2\mbox{ if }\a_1>-n\mbox{ and }\a_2<0.
\end{equation}
\end{lemma}

\begin{proof}
We let $\gamma=(\a_1,\a_2,0^{n-2})$, so that $\gamma+\delta =
(\a_1+n-1,\a_2+n-2,n-3,\dots,0)$. If $\gamma+\delta$ has repeated
entries, then all the cohomology groups of $\SS_{\a}\mc{Q}$ vanish
by Theorem~\ref{thm:bott}. We therefore assume this is not the case. We have 
$(\gamma+\delta)_1 > (\gamma+\delta)_2$ and
$(\gamma+\delta)_3>\cdots>(\gamma+\delta)_n$,
so the only possible inversions in $\gamma+\delta$ are of the form $(i,j)$
with $i\leq 2$ and $j\geq 3$. Since the entries $(\gamma+\delta)_3,
\dots,(\gamma+\delta)_n$ are consecutive integers, if $(i,j)$ is an
inversion for some $j\geq 3$, then it is an inversion for all $j\geq 3$.
This shows that the total number of inversions is
\begin{itemize}[itemsep=2pt]
 \item $0$, if $(\gamma+\delta)_2>(\gamma+\delta)_3$, 
 or equivalently, if $\a_2\geq 0$.
 \item $(n-2)$, if $(\gamma+\delta)_1>(\gamma+\delta)_3$ and
 $(\gamma+\delta)_n>(\gamma+\delta)_2$, or equivalently, 
 if $\a_1\geq -1$ and $\a_2<-(n-2)$.
 \item $(2n-4)$, if $(\gamma+\delta)_n>(\gamma+\delta)_1$, 
 or equivalently, if $\a_1\leq -n$.
\end{itemize}

Based on Theorem~\ref{thm:bott}, we now conclude that the first vanishing holds. Moreover, if $\a_2\geq-(n-2)$ 
then the above cases show that $\gamma+\delta$
cannot have any inversions, proving \eqref{eq:Sa-Q-only-0}. 
If $\a_2<0$, then $\gamma+\delta$ has at least one inversion, 
while if $\a_1>-n$ then it cannot have $2n-4$ inversions,
so \eqref{eq:Sa-Q-only-n2} holds as well.
\end{proof}

\subsection{Proof of Theorem \ref{thm:finite}}
\label{sec:Proof_Finite}

As a first application of Bott's theorem, we  determine the Hilbert function of Koszul modules when 
$\B(V,K)$ is finite of length, say, $\ell$. The proof is along the lines of \cite{AFPRW2}*{Theorem 1.3}, 
and we skip some details which can be found in \cite{AFPRW2}. We consider the negatively indexed Koszul complex 
\begin{equation}
\label{eqn:Koszul_finite}
\mathcal{K}^\bullet:\, 0\longrightarrow\bigwedge^mK\otimes \mathcal{O}_\mathbf{G}(1-m)\longrightarrow
\cdots \longrightarrow \bigwedge^2K\otimes\mathcal{O}_\mathbf{G}(1)\longrightarrow K\otimes \mathcal{O}_\mathbf{G}
\stackrel{\mathrm{ev}_K}\longrightarrow \mathcal{O}_\mathbf{G}(1)\longrightarrow 0
\end{equation}
associated to the evaluation map $\mathrm{ev}_K\colon K\otimes\mathcal{O}_\mathbf{G}\to \mathcal{O}_\mathbf{G}(1)$. Since the cokernel of $\mathrm{ev}_K$ is supported on $\B(V,K)$, the complex $\mathcal{K}^\bullet$ is exact on the complement of $\B(V,K)$. As before,  $\mc{Q}$ is the universal quotient bundle defined by \eqref{eq:uv-exact}. Twisting the 
complex \eqref{eqn:Koszul_finite} by $\Sym^q\mc{Q}$ and taking hypercohomology, we obtain two 
spectral sequences abutting to the same limit
\[
'E_1^{-i,j}=\bigwedge^iK\otimes H^j\left(\mathbf{G}, \mathbb{S}_{(q+1-i,1-i)}\mc{Q}\right)
\]
respectively
\[
''E_2^{-i,j}=H^j\bigl(\mathbf{G}, \mathcal{H}^{-i}(\mathcal{K}^\bullet\otimes \Sym^q\mc{Q})\bigr).
\]

Using Theorem \ref{thm:bott} as in the proof of \cite{AFPRW2}*{Theorem 1.3}, 
we readily obtain
\[
W_q(V,K)={ }'E_\infty^{0,0}, \mbox{ for all }q\ge n-3.
\]

On the other hand, since $\mathcal{H}^0(\mathcal{K}^\bullet\otimes \Sym^qQ)=
\mathcal{O}_{\B(V,K)}(1)\otimes \Sym^qQ$, we obtain
\[
''E_2^{0,0}=H^0\bigl(\B(V,K),\mathcal{O}_{\B(V,K)}(1)\otimes \Sym^q\mc{Q}\bigr).
\]
Furthermore, the support of $\mathcal{H}^{-i}(\mathcal{K}^\bullet\otimes \Sym^q\mc{Q}))$ 
is finite for all $j$, which implies that $''E_2^{-i,j}=0$ for all $i$ and all $j\ne 0$. 
In particular,  $''E_2^{0,0}={ }''E_\infty^{0,0}$, concluding the proof. 
\hfill
$\Box$

\begin{example}\label{ex:non-Kähler}
 Returning to the pair $(V,K)$ from Example \ref{ex:non-red-bvk}, Theorem \ref{thm:finite} implies that $\dim W_q(V,K)=2(q+1)$ for $q\ge 1$. This algebraic data corresponds to the 2-step nilpotent Lie algebra $\mathfrak{h}_4$ studied in \cite{Sal01}*{Theorem 3.3}. With the notation of \cite{Uga07}*{Theorem 8}, this algebra is defined by  $\mathfrak{h}_4\coloneqq V \oplus (K^\perp)^\vee$, with Lie bracket 
 $[(v_1,z),(v_2,z')]= (0,z_1)$, $[(v_1,z),(v_3,z')]= (0,z_2)$, 
 $[(v_2,z),(v_4,z')]= (0,z_2)$, and $[(v_i,z),(v_j,z')]= (0,0)$ in all other cases, for $z,z'\in (K^\perp)^\vee$. For any lattice $G$ in the associated Lie group $H$, the quotient $M = H/G$ is a compact \emph{non-K\"ahler} Calabi-Yau 3-fold with fundamental group $G$, and the corresponding resonance variety $\mathcal{R}(G)$ is isomorphic to $\mathcal{R}(V,K)$. Hence its associated Koszul module is supported on a non-reduced projective line.
\end{example}

\begin{example}
\label{ex:HA_finite}
The fundamental group of the complement of a graphic arrangement has two-dimensional 
resonance. This example will be studied in more detail in Subsection \ref{subsec:graphic}.
\end{example}

An immediate consequence of Theorem \ref{thm:finite} is the following.

\begin{corollary}
\label{thm:transversal}
Assume $\dim K=2n-4$ and that $\B(V,K)$ is finite. 
Then the intersection $\B(V,K)$ is transverse 
if and only if $\Rproj(V,K)$ is reduced, in which case the 
Hilbert function of the Koszul module equals
 \[
 \dim W_q(V,K) = \frac{1}{n-1}\binom{2n-4}{n-2}\cdot(q+1),\mbox{ for }q\geq n-3.
 \]    
\end{corollary}

\begin{proof}
We apply Theorem \ref{thm:finite} and use that the degree of $\Gr$ in its Pl\"ucker embedding  
equals the Catalan number $c_{n-2}=\frac{1}{n-1}\binom{2n-4}{n-2}$.  In particular, the length $\ell$ 
in \eqref{eqn:Chen_Finite_Intersection} equals $c_{n-2}$.
\end{proof}
If we try and adapt the proof of Theorem \ref{thm:finite} to the case where the linear section 
$\B(V,K)$ of the Grassmannian is positive-dimensional, the second spectral sequence becomes 
quickly uncontrollable. To address the positive-dimensional case, we  impose extra assumptions 
and work with a Koszul complex on the blow-up of the Grassmannian along this linear section. 

\subsection{Koszul modules associated to $K3$ surfaces}
Corollary \ref{thm:transversal} has geometric applications provided by Koszul modules associated to vector bundles \cites{AFRS, AFRW}. Suppose $X$ is a polarized $K3$ surface with $\Pic(X)\cong \Z\cdot H$, where $H^2=4r-2$, with $r\geq 2$. Then there exists a \emph{unique} stable vector bundle $E$ on $X$ with Mukai vector $v(E)=(2, H, r)$. For a smooth curve $C\in |H|$, let $E_C\coloneqq E_{|C}$ be the restriction of this vector bundle. Then $\det(E_C)\cong \omega_C$, $h^0(C,E_C)=h^0(X,E)=r+2$ and $E_C$ is stable.  We consider the determinant map 
\[
d\colon \bwedge^2 H^0(C,E_C)\longrightarrow H^0(C,\omega_C)
\]
and the associated Koszul module $W(E_C)\coloneqq W\bigl(H^0(C,E_C)^{\vee}, K\bigr)$, 
where $K^{\perp}=\ker(d)$.

\begin{proposition}
With $X$ as above, for a general curve $C\in|H|$ the projective resonance 
$\Rproj(E_C)\subseteq \P H^0(C, E_C)$ is a reduced union of 
$\frac{1}{r+1}\binom{2r}{r}$ projective lines. Furthermore, one has  
\[
\dim \ W_q(E_C)=\frac{1}{r+1}\binom{2r}{r}\cdot (q+1), \mbox{ for } q\geq r-1.
\]
\end{proposition}
\begin{proof} Denoting 
$\Gr=\Grass_2\bigl(H^0(C, E_C)\bigr)$, the intersection $\P(K^{\perp})\cap \Gr$ corresponds to elements  $[0\neq s_1\wedge s_2]\in \P\bigl(\bigwedge^2 H^0(C, E_C)\bigr)$, where $s_1, s_2\in H^0(C, E_C)$ are such that $d(s_1\wedge s_2)=0$. From \cite{AFRW}*{Lemma 4.8} it follows that this intersection can be identified with the Brill-Noether variety $W^1_{r+1}(C)$. The curve $C$ being Petri general \cite{La}, the projective resonance $\Rproj(E_C)$ consists of $\frac{1}{r+1}{2r\choose r}$ reduced lines. The last statement follows directly from Corollary \ref{thm:transversal}.
\end{proof}

\section{Isotropic and separable components of the resonance}
\label{sect:iso-sep}

We recall from \cite{AFRS} a few definitions describing the structure of resonance varieties 
of Koszul modules. We fix throughout an $n$-dimensional complex vector space $V$. 

\begin{definition}
\label{def:isotropic}
Let $K\subseteq \bw^2 V$ be a subspace. 
A subspace $\ol{V}^\vee\subseteq V^{\vee}$ is called \defi{isotropic} 
(with respect to $K$) if $\bw^2 \ol{V}^\vee\subseteq K^{\perp}$. The resonance $\RR(V, K)$ is 
isotropic if it is linear and each of its irreducible components is isotropic.
\end{definition}

By definition, an isotropic subspace $\ol{V}\subseteq V^{\vee}$ 
is automatically contained in the resonance $\mathcal{R}(V,K)$. As pointed out in \cite{AFRS}, 
isotropicity  can be described by passing to the quotient, as follows.
Let $\ol{V}^{\vee}\subseteq V^{\vee}$ be a linear subspace 
corresponding to a surjective  map $\pi\colon V\to \ol{V}$ 
and let $\ol{K}$ be the image of $K$ in $\bigwedge^2\ol{V}$. 
Then $\ol{V}^{\vee}$ is isotropic if and only if $\ol{K}=0$.

\begin{example}
\label{ex:K-1dim}
 If $K$ is one-dimensional, the quotient map $\bigwedge^2V^\vee\to K^\vee\cong \C$ corresponds 
to a alternating quadratic form on $V^\vee$. Isotropicity 
with respect to $K$ coincides with isotropicity with respect to this 
form, and the base locus $\B(V,K)$ is the isotropic Grassmannian.
\end{example}

\begin{definition}
\label{def:separable}
A subspace $\ol{V}^{\vee}\subseteq V^{\vee}$  is said to be \defi{separable} (with respect to $K$) if 
$\bigl(\ol{V}^\vee \wedge V^{\vee} \bigr) \cap K^{\perp}
\subseteq \ol{V}^\vee \wedge \ol{V}^\vee$. 
The subspace $\ol{V}^\vee$ is \defi{strongly isotropic}
if it is both separable and isotropic, that is, if
$\bigl(\ol{V}^\vee \wedge V^{\vee}\big)\cap K^{\perp} = 
\bw^2 \ol{V}^\vee$. The resonance $\RR(V, K)$ is \emph{separable} (respectively \emph{strongly isotropic}) 
if it is linear and all of its irreducible components are separable (respectively strongly isotropic).
\end{definition}

Separability can be verified in terms of the equations of the subspace $\overline{V}^{\vee}\subseteq V^{\vee}$. 
The following simple fact will be used in Section \ref{sect:arr-gps} when dealing with hyperplane arrangements. 

\begin{lemma}
\label{lem:wedge2-equations}
If $\ol{V}^\vee \subseteq V^\vee$ is defined by equations
$w_1 = \cdots = w_\ell = 0$,  for $w_1,\ldots,w_\ell\in V$, then the equations of 
$\bigwedge^2\ol{V}^\vee\subseteq \bigwedge^2V^\vee$ are
\begin{equation}
\label{eqn:wedge2-equations}
    w_1 \wedge v = \cdots = w_\ell \wedge v = 0, \mbox{ for all }v\in V,
\end{equation}
and $\ol{V}^\vee \wedge V^\vee$ is defined in $\bwedge^2V^\vee$ by
\begin{equation}
   \label{eqn:olVwedgeV-equations}
w_i \wedge w_j = 0, \mbox{ for all }i,j\in\{1,\ldots,\ell\}.
\end{equation}
\end{lemma}
\begin{proof}
 Let $(e_1,\ldots,e_n)$ be the basis of $V^{\vee}$ such that
$(e_1,\ldots,e_{\ol{n}})$ is a basis for $\ol{V}^{\vee}$, and $(v_1,\ldots,v_n)$ 
is the dual basis of $V$. The equations of $\ol{V}^\vee$, $\bigwedge^2\ol{V}^\vee$ 
respectively $\ol{V}^\vee\wedge V^\vee$ are given by $v_{\ol{n}+1} = \cdots = v_n = 0$, 
\[
v_{\ol{n}+1} \wedge v_1 = v_{\ol{n}+1} \wedge v_2 = \cdots = v_{\ol{n}+1} \wedge v_n = 
v_{\ol{n}+2} \wedge v_1 = \cdots = v_{n-1} \wedge v_n = 0,
\] 
respectively,
\begin{equation}
   \label{eqn:olVwedgeV-equations-basis} 
v_{\ol{n}+1} \wedge v_{\ol{n}+2} = \cdots = v_{\ol{n}+1} \wedge v_n = 
\cdots = v_{n-1} \wedge v_n = 0.
\end{equation}
\end{proof}

\begin{remark}
The difficulty in verifying the separability of $\ol{V}^\vee \subseteq V^{\vee}$ 
lies in showing that the equations of $K^\perp$ bring enough contribution to  
\eqref{eqn:olVwedgeV-equations-basis} to recover the extra-equations
\[
v_{\ol{n}+1} \wedge v_1 = \cdots = v_{\ol{n}+1} \wedge v_{\ol{n}} = 
v_{\ol{n}+2} \wedge v_1 = \cdots = v_{n} \wedge v_{\ol{n}} = 0
\]
defining $\bwedge^2\ol{V}^\vee\subseteq \ol{V}^\vee\wedge V^\vee$. We note the necessary 
condition $m\ge \ol{n}(n-\ol{n})$ for separability.
\end{remark}

Suppose $\ol{V}^{\vee}\subseteq V^{\vee}$ is a subspace of dimension $\ol{n}\leq n$ 
and set $U\coloneqq\ker \bigl\{\pi\colon V\to \ol{V}\bigr\}$. As before, fix a basis
$(e_1,\ldots,e_n)$ of $V^{\vee}$ such that
$(e_1,\ldots,e_{\ol{n}})$ is a basis for $\ol{V}^{\vee}$. Letting
$(v_1,\ldots,v_n)$ denote the dual basis of $V$,
we obtain a decomposition
\begin{equation}
\label{eq:decomp-bw2V}
\bwedge^2\, V = L \oplus M \oplus H, 
\end{equation}
where $L\coloneqq \spn\bigl\{ v_s \wedge v_t : s,t\leq \ol{n}\bigr\} \notag\cong \bwedge^2\, \ol{V}$, 
$M\coloneqq \spn\bigl\{ v_s \wedge v_t : s\leq \ol{n} \mbox{ and }  t>\ol{n}\bigr\} \cong \ol{V}\otimes U$, and 
$H\coloneqq \spn\bigl\{ v_s \wedge v_t : s,t>\ol{n}\bigr\} \notag\cong \bwedge^2\, U$. We quote the following 
result from \cite{AFRS}*{\S3.1}:

\begin{lemma}
\label{lem:2nd-basis}
Assume $\ol{V}^{\vee}$ is a separable  component of 
$\RR(V,K)$. 
Then there  exists a basis of $K$ of the form
$\{\a_{s,t}:s\leq \ol{n}, t>\ol{n}\} \cup \{\b_1, \ldots, \b_N\}$,
where $N\geq 0$, such that, for 
$s\leq \ol{n}$, $t>\ol{n}$, and $1\leq j\leq N$, we have
\begin{equation}
\label{eq:abs}
\a_{s,t} = v_s \wedge v_t + h_{s,t},\quad \b_j =\ell_j + h_j,
\end{equation}
for some collection of elements  $\ell_j \in L$ and $h_{s,t}, h_j\in H$.
Furthermore, we may assume that $\ell_{s,t}\in \spn\{\ell_1, \ldots, \ell_N\}$, for each $s\leq \ol{n}$ and $t>\ol{n}$.
\end{lemma}

Let $p_M$ denote the restriction to $K$
of the second  projection $\bw^2 V\to  M$ with respect to the decomposition given by \eqref{eq:decomp-bw2V}.  
It is shown in  \cite{AFRS}*{Corollary 3.12} that an isotropic subspace $\ol{V}^{\vee}\subseteq V^{\vee}$ 
is separable (with respect to $K$) if and only if $p_M\colon K\to M$ is surjective.

\subsection{Strong isotropicity and components of the base locus of $|K|$}
\label{subsec:base-sep}
It is established in \cite{AFRS}*{Theorem 5.1} that an isotropic component $\overline{V}^{\vee}$ of $\RR(V,K)$ is strongly isotropic if and only if $\P\overline{V}^{\vee}$ is a reduced component of the projective resonance $\Rproj(V,K)$. In the sequel, we establish a similar result for the mirror scheme 
$\B(V,K)$. We assume the resonance is 
linear and $\mc{R}(V,K)=\overline{V}_1^\vee\cup\cdots\cup \overline{V}_k^\vee$. Denote by $\Gr_t\coloneqq\Grass_2(\overline{V}_t^\vee)\subseteq \Gr$, for $t=1, \ldots, k$.

In this setup, we prove the following.
\begin{theorem}
    \label{thm:reduced}
    If $\mc{R}(V,K)$ is strongly isotropic, then the base locus $\B(V,K)$ is 
    reduced and  coincides with $\Gr_1 \sqcup \cdots \sqcup \Gr_k$. 
\end{theorem}

\proof
We first establish this equality set-theoretically. Assume $0\ne a\wedge b\in K^\perp$, then $a\in \mc{R}(V,K)=\overline{V}_1^\vee\cup\cdots\cup \overline{V}_k^\vee$, and hence $a\in \overline{V}_t^\vee$ for some $t$. On the other hand, $a\wedge b\in  \overline{V}_t^\vee \wedge V^{\vee}$, which implies, by the separability hypothesis, that $a\wedge b\in \bwedge^2\overline{V}_t^{\vee}$. We obtain $b\in \overline{V}_t^{\vee}$, that is, $[a\wedge b]\in \Gr_t$, as claimed. By \cite{AFRS}*{Corollary 4.6}, the projective resonance $\Rproj(V,K)$ consists of projectively disjoint components, hence also the sub-Grassmannians $\Gr_t$ are mutually disjoint.
Finally, if all $\overline{V}_t$ are isotropic, then by definition  $\bigwedge^2\overline{V}_t^{\vee}\subseteq K^\perp$,  hence also $\Gr_t\subseteq \mathbf{P}K^\perp$ for $t=1, \ldots, k$. 

We verify $\B\coloneqq\B(V,K)$ is reduced. This being a local property, we focus on a component
$\ol{V}^{\vee}=\overline{V}_t^{\vee}$. Set $\overline{\Gr}\coloneqq\Gr_t$ and let
$\mc{I}_\B$ and $\mc{I}$ be the ideal sheaves of $\B$ and $\overline{\Gr}$ in $\Gr$ respectively. 
The  map  $K\otimes \mathcal{O}_{\Gr} \twoheadrightarrow \mc{I}_\B(1)$
restricts to a surjective morphism $K\otimes \mc{O}_{\overline{\Gr}}\twoheadrightarrow 
\mc{I}_\B(1)_{|\overline{\Gr}}$.

\begin{claim*} The composition
\begin{equation}
\label{eqn:restriction-ideal}
\begin{tikzcd}[column sep=18pt]
K\otimes \mc{O}_{\overline{\Gr}}\ar[r, two heads] 
&\mc{I}_\B(1)_{|\overline{\Gr}}\ar[r, hook] 
&\mc{I}(1)_{|\overline{\Gr}}
\end{tikzcd}
\end{equation}
is a surjective morphism.
\end{claim*}
Accepting the claim, it follows that $\mc{I}_{\B}\otimes \mathcal{O}_{\overline{\Gr}}=
\mc{I}\otimes \mathcal{O}_{\overline{\Gr}}=\mc{I}/\mc{I}^2$. This shows that $\B$ 
is reduced along $\overline{\Gr}$. Indeed, if $p\in \overline{\Gr}$ 
is an arbitrary point, we verify that the inclusion $\mc{I}_{\B,p}\subseteq 
\mc{I}_p$ of ideals in $\mc{O}_{\Gr,p}$ is in fact an equality. Let $\mf{m}$ 
be the maximal ideal of $\mc{O}_{\Gr,p}$. Localizing the equality 
$\mc{I}_{\B}\otimes \mathcal{O}_{\overline{\Gr}}=\mc{I}/\mc{I}^2$ at $p$,
we conclude that $\mc{I}_{\B,p} + \mc{I}_p^2 = \mc{I}_p$.

We therefore obtain a chain of inclusions of ideals 
\[
\mc{I}_{\B,p} \subseteq \mc{I}_p =
\mc{I}_{\B,p} + \mc{I}_p^2 \subseteq \mc{I}_{\B,p} + \mf{m}\cdot  \mc{I}_p,
\]
which can be reinterpreted as stating that 
$\mf{m} \cdot (\mc{I}_p / \mc{I}_{\B,p})=\mc{I}_p / \mc{I}_{\B,p}$. 
By Nakayama's lemma, we conclude
that $\mc{I}_{\B,p} = \mc{I}_p$, as desired.

In order to complete the proof, it remains to verify that the morphism 
\eqref{eqn:restriction-ideal} is surjective. To that end,
we  perform a local analysis. Set 
$\ol{n} = \dim(\ol{V}) \geq 2$ and consider bases for $\ol{V}^{\vee}$ 
and $V^{\vee}$ as in Lemma~\ref{lem:2nd-basis}. Consider a point $p\in\ol{\Gr}$ 
and we may assume $p=[e_1\wedge e_2]$. We prove that $p$ is a reduced
point of $\B$. Considering the basis of $V$ as in Lemma~\ref{lem:2nd-basis},
we obtain a basis $\bigl(X_{i,j} = v_i\wedge v_j\bigr)_{1\leq i<j\leq n}$
for $H^0(\Gr,\mc{O}_\Gr(1)) = \bw^2 V$.
We think of $X_{i,j}$ as homogeneous coordinates on $\P$ and let $U_{1,2}\subseteq \P$ be the open affine subset defined by $X_{1,2}\neq 0$. We have
$U_{1,2} = \Spec(R)$, where $R = \C[x_{i,j}]$ and $x_{i,j} =
X_{i,j}/X_{1,2}$ for $1\leq i<j\leq n$. In this chart, the ideal of Pl\"ucker relations
defining $U_{1,2}\cap\Gr \subseteq U$ is generated by
\begin{equation}
\label{eq:x-ij-plucker}
\bigl\{ x_{i,j} - x_{1,i}x_{2,j}+x_{1,j}x_{2,i}: 3\leq i<j\leq n\bigr\}.
\end{equation}

Writing $p_{i,j}$ for the restriction of $x_{i,j}$ to $U_{1,2}\cap\Gr$, we infer that
$U_{1,2}\cap \Gr = \Spec(A)$, where 
$A=\C\bigl[p_{1,j},p_{2,j}:3\leq j\leq n\bigr]$.
Restricting the evaluation map \eqref{eqn:ev_K} to $U_{1,2}\cap\Gr$,
we obtain a map of free $A$-modules, $\pd\colon K\oo A \to A$.  It is easily
seen that this map coincides with the restriction to $K\oo A$ of
the map $\pd\colon \bw^2 V \oo A \to A$ defined by
$(v_i\wedge v_j)\otimes 1 \mapsto p_{i,j}$.  Note that the image of
$\pd\colon K\oo A \to A$  is the defining ideal $I_{\B}$
of the scheme $U_{1,2} \cap \B$.

If we write $I\subseteq A$ for the defining ideal of $U_{1,2}\cap\ol{\Gr}$, 
then $\op{Im}(\pd)=I_{\B}\subseteq I$, since 
$U_{1,2}\cap\ol{\Gr}\subseteq U_{1,2}\cap\B$
is a closed subscheme. Moreover, $I$ has an explicit generating
set of Pl\"ucker coordinates, namely
$I =\langle p_{1,j},p_{2,j}:\ol{n}+1\leq j \leq n\rangle.$

Regarding $\pd$ as a map $K\oo A \to I$,  after tensoring with
$\ol{A} = A/I$, we obtain a map
\[
\begin{tikzcd}[column sep=18pt]
\ol{\pd} \colon K \oo \ol{A} \ar[r]& I/I^2 , 
\end{tikzcd}
\]
and the surjectivity of the composed map \eqref{eqn:restriction-ideal} 
reduces to the surjectivity of $\ol{\pd}$. Using the explicit generating set for $I$, 
we get that the module $I/I^2$ is generated by the classes $\ol{p_{1,j}}, \ol{p_{2,j}}$ of
$p_{1,j},p_{2,j}$ modulo $I^2$, where $j=\ol{n}+1,\dots,n$, thus it suffices to show that 
these classes lie in the image of~$\ol{\pd}$. For $\ol{n}+1\leq i<j\leq n$, we have
\[
\pd(v_i\wedge v_j) = p_{i,j} = p_{1,i}p_{2,j} - p_{1,j} p_{2,i} \in I^2.
\]
Using the notation \eqref{eq:abs} for the elements
$\alpha_{s,t}\in K$, then $\pd(h_{s,t}) \in I^2$,  therefore $\ol{\pd}(\alpha_{i,j}) = \ol{p_{i,j}}$, 
for $i=1,2$ and $j=\ol{n}+1,\dots,n$. This shows $\ol{\pd}$ is surjective.
\endproof

\begin{remark}
The base locus $\mathbf{B}(V,K)$  may be reduced even in the non-separable case. 
For instance, take $V^\vee$ to be a 6-dimensional space with a basis $(e_1,\ldots,e_6)$ and we choose 
$
K^{\perp}=\spn\{ e_1\wedge e_2,e_1\wedge e_3,e_2\wedge e_3,e_1\wedge e_4+e_2\wedge e_5+e_3\wedge e_6\}.
$    
In this case, $\mathbf{B}(V,K)$ is scheme-theoretically the projective plane spanned by $[e_1\wedge e_2]$, $[e_1\wedge e_3]$, and $[e_2\wedge e_3]$. 
\end{remark}

When $\mathbf{B}(V,K)$ is finite, we have the following converse of Theorem \ref{thm:reduced}:

\begin{proposition}
\label{prop:finite-reduced}
With notation as above, we assume the scheme-theoretic intersection 
$\Gr\cap\mathbf{P}K^\perp$ is finite and reduced. Then the 
resonance  $\RR(V,K)$ is strongly isotropic.
\end{proposition}

\begin{proof}
    By the hypothesis, the resonance is linear and every component is of minimal dimension two. In particular, isotropy follows immediately, as already noticed before.
    To conclude, we need to prove that every component is separable. 
    
    Let $\overline{V}^\vee$ be a component of $\RR(V,K)$, generated by $e_1,e_2$. Notice that $e_1\wedge e_2\in K^\perp$, and for every $f\in V$ such that $e_1\wedge f\in K^\perp$ or $e_2\wedge f\in K^\perp$, we automatically have $f\in\overline{V}^\vee$. Assume $\overline{V}^\vee$ is not separable. Then there exist $f_1,f_2\in V^\vee$ such that $e_1\wedge f_1+e_2\wedge f_2\in K^\perp\setminus \bigwedge^2\overline{V}^\vee$. By the observation above, it follows that $f_1,f_2\not\in \overline{V}^\vee$. If $f_1\wedge f_2=0$, then $f_1$ and $f_2$ are collinear and we would have a line contained in $\Gr\cap \mathbf{P}K^\perp$, which contradicts the finiteness of this intersection. Hence, $f_1\wedge f_2\ne 0$, and then, factoring through $\bigwedge^2\overline{V}^\vee$, the element $e_1\wedge f_1+e_2\wedge f_2\in K^\perp$ defines a non-zero vector in $T_{[\overline{V}^\vee]}\Gr\cap T_{[\overline{V}^\vee]}\mathbf{P}K^\perp$.  This contradicts the reducedness of the intersection, 
    and we are done.
\end{proof}

\section{The geometry of sub-Grassmannians}
\label{sect:normal-bundle}

Let $\ol{V}^{\vee}\subsetneqq V^{\vee}$ be a proper linear 
subspace, corresponding to a projection $\pi\colon V\onto \ol{V}$ and set $U\coloneqq \ker(\pi)$. 
We describe the geometry of the Grassmannian  
$\ol{\Gr} \coloneqq \op{Gr}_2(\ol{V}^{\vee})\hookrightarrow \Gr$. As before, set 
$\mc{I}\coloneqq\mc{I}_{\ol{\Gr}/\Gr}$. We have tautological sequences on $\Gr$ 
(respectively $\ol{\Gr}$):
\begin{equation}
\label{eq:tautologica-GGbar}
\begin{tikzcd}[column sep=18pt, row sep=2pt]
0 \ar[r]& \mc{U} \ar[r]& V \oo \mc{O}_{\Gr} \ar[r]& \mc{Q} \ar[r]& 0 ,
\\
0 \ar[r]& \ol{\mc{U}} \ar[r]& \ol{V} \oo \mc{O}_{\ol{\Gr}} \ar[r]
& \ol{\mc{Q}} \ar[r]& 0 ,
 \end{tikzcd}
\end{equation}
where $\mc{Q}$ and $\ol{\mc{Q}}$ are the tautological rank $2$ 
quotients of $V$ (respectively $\ol{V}$). We have 
$\ol{\mc{Q}} \cong \mc{Q}_{|_{\ol{\Gr}}}$, and the tautological subbundles bundles $\mc{U}$ and 
$\ol{\mc{U}}$ are related by the exact sequence
\begin{equation}
\label{eq:UUbar-seq}
\begin{tikzcd}[column sep=18pt]
0 \ar[r]& U \oo \mc{O}_{\ol{\Gr}} \ar[r]&
\mc{U}_{|_{\ol{\Gr}}} \ar[r]& \ol{\mc{U}}\ar[r]& 0 \, .
\end{tikzcd}
\end{equation}
It is well known that $\Omega_{\Gr} = \mc{U} \oo \mc{Q}^{\vee}$, respectively 
$\Omega_{\ol{\Gr}} = \ol{\mc{U}} \oo \ol{\mc{Q}}^{\vee}$.
Therefore, if we tensor \eqref{eq:UUbar-seq} with $\ol{\mc{Q}}^{\vee}\cong
\mc{Q}^{\vee}_{|_{\ol{\Gr}}}$, we obtain the conormal exact sequence
\begin{equation}
\label{eq:conormal-ses}
\begin{tikzcd}[column sep=18pt]
0 \ar[r]& U \oo \ol{\mc{Q}}^{\vee} \ar[r]& (\Omega_{\Gr})_{|_{\ol{\Gr}}} 
\ar[r]& \Omega_{\ol{\Gr}} \ar[r]& 0 .
 \end{tikzcd}
\end{equation}
We find $\mc{I}/\mc{I}^2 \cong U \oo \ol{\mc{Q}}^{\vee}$ and the normal bundle of
$\ol{\Gr}$ in $\Gr$ is given by 
$\mc{N}\cong  U^{\vee} \oo \ol{\mc{Q}}$. Let 
\begin{equation}
\label{eq:def-ProjN}
E \coloneqq \op{\ul{Proj}}\bigl(\Sym_{\mc{O}_{\ol{\Gr}}}(\mc{I}/\mc{I}^2)\bigr) =
\op{\ul{Proj}}\bigl(\Sym_{\mc{O}_{\ol{\Gr}}}(U \oo \ol{\mc{Q}}^{\vee})\bigr) = 
\P(\mc{N})
\end{equation}
be the projectivized normal bundle of $\ol{\Gr}$, and let
\begin{equation}
\label{eq:def-ol-gamma}
\begin{tikzcd}[column sep=16pt]
\ol{\gamma}\colon E \ar[r]& \ol{\Gr}
\end{tikzcd}
\end{equation}
denote the structure map of this bundle. Then $\mc{O}_E(-E)$ is the
tautological quotient of $\ol{\gamma}^*(\mc{I}/\mc{I}^2)$. Setting
$n=\dim V$ and $\ol{n} = \dim \ol{V}$, then  $\rk(\mc{N})=2\cdot(n-\ol{n})$.
We write $N=2\cdot(n-\ol{n})$, hence $E$ is a $\P^{N-1}$-bundle
over $\ol{\Gr}$. Using \cite{hartshorne}*{Exercise~III.8.4}, we obtain
\begin{equation}
\label{eq:gamma*-OE}
\ol{\gamma}_*(\mc{O}_E(-dE)) = \begin{cases}
 \Sym^{d}(U \oo \ol{\mc{Q}}^{\vee}) & \mbox{if }d\geq 0, \\
 0 & \mbox{otherwise.}
\end{cases}
\end{equation}

To compute the higher direct images of $\mc{O}_E(-dE)$, note that $\det(\mc{I}/\mc{I}^2) \cong \mc{O}_{\ol{\Gr}}(\ol{n}-n)$. 
It follows from \cite{hartshorne}*{Exercise~III.8.4} that, for $t\geq 1$ we have
\begin{equation}
\label{eq:Rgamma*-OE}
R^t\ol{\gamma}_*(\mc{O}_E(dE)) = \begin{cases}
 \Sym^{d-N}(U^{\vee} \oo \ol{\mc{Q}}) (n-\ol{n})& 
 \mbox{if }t=N-1\mbox{ and }d\geq N, \\[2pt]
 0 & \mbox{otherwise.} \\
\end{cases}
\end{equation}
Note that above we have used the identification which holds in characteristic 0 that
\[
\bigl(\Sym^{d-N}(\mc{I}/\mc{I}^2)\bigr)^{\vee} = \bigl(\Sym^{d-N}(U \oo
\ol{\mc{Q}}^{\vee})\bigr)^{\vee}\cong\Sym^{d-N}(U^{\vee} \oo \ol{\mc{Q}}).
\]

\subsection{Strong isotropicity and normal bundles of base loci}
\label{subsec:strong-isotropicity}

We assume now that $W(V,K)$ is a Koszul module whose 
resonance is strongly isotropic, hence by 
\cite{AFRS}*{Theorem 5.1} it is projectively
reduced and projectively disjoint. We write
$\mc{R}(V,K) = \overline{V}_1^{\vee} \cup \cdots \cup \overline{V}_k^{\vee}$ for its 
decomposition \eqref{eq:linear resonance}. As before, let $\Gr_t = \op{Gr}_2(V_t^{\vee})\subseteq \Gr$ for
$t=1,\dots,k$ and $\B=\B(V,K)$ defined by \eqref{eq:linearsection}. Restricting the surjection
$K\oo\mc{O}_{\Gr} \onto \mc{I}_{\B}(1)$ to $\B$, we obtain a 
surjection 
\[
\begin{tikzcd}[column sep=20pt]
\op{ev}_{\B}\colon K \oo \mc{O}_{\B} \ar[r]& (\mc{I}_{\B}/\mc{I}_{\B}^2)(1) .
\end{tikzcd}
\]
Restricting this map further to a component 
$\ol{\Gr}=\ol{\Gr}_t$ of $\B$, we obtain a surjective morphism
\begin{equation}
\label{eq:ev-Gr}
\begin{tikzcd}[column sep=20pt]
\op{ev}_{\ol{\Gr}}\colon K \oo \mc{O}_{\ol{\Gr}}  \ar[r]& (\mc{I}/\mc{I}^2)(1)\cong  U\oo\ol{\mc{Q}} ,
 \end{tikzcd}
\end{equation}
 where the
identification $(\mc{I}/\mc{I}^2)(1) \cong U\oo\ol{\mc{Q}}$
follows from the isomorphism 
$\ol{\mc{Q}}^{\vee}(1)\cong \ol{\mc{Q}}$.

The strong isotropicity  of $\mc{R}(V,K)$ does not only imply 
the surjectivity of \eqref{eq:ev-Gr} as a morphism of sheaves, but also at the 
level of global sections.

\begin{lemma}
\label{lem:ev-Gr-surj}
The map 
$H^0(\ol{\Gr},\op{ev}_{\ol{\Gr}})\colon H^0(\ol{\Gr},K\oo \mc{O}_{\ol{\Gr}}) 
\longrightarrow H^0(\ol{\Gr},U\oo\ol{\mc{Q}})$ is surjective.
\end{lemma}
\begin{proof}
From Theorem \ref{thm:bott} we have the identification $H^0(\ol{\Gr},U\oo\ol{\mc{Q}}) \cong U \oo \ol{V}$. Choosing bases as in Lemma~\ref{lem:2nd-basis}, we infer that $U=\spn\{v_{\ol{n}+1},\dots,v_n\}$
and $\ol{V} = \spn\{v_1,\dots,v_{\ol{n}}\}$, so that $U\oo\ol{V}$ equals
the vector space $M$ from the decomposition \eqref{eq:decomp-bw2V}. Moreover, under these 
identifications the map $H^0(\ol{\Gr},\op{ev}_{\ol{\Gr}})$ is given by the 
projection 
\[
p_M\colon K\to M,
\]
which, as already pointed out, is surjective by \cite{AFRS}*{Corollary 3.12}.
\end{proof}

The surjection \eqref{eq:ev-Gr}, together with the natural multiplication
of symmetric powers, induces surjective moprphisms of sheaves  
$K \oo \Sym^d\left((\mc{I}/\mc{I}^2)(1)\right)
\twoheadrightarrow\Sym^{d+1}\left((\mc{I}/\mc{I}^2)(1)\right)$ on $\ol{\Gr}$. 
Tensoring with $(\Sym^q\ol{\mc{Q}})(-d)$
and using the isomorphism
\begin{equation}
\label{eq:Symd-I/I2-P}
\Sym^d\left((\mc{I}/\mc{I}^2)(1)\right)
\cong (\Sym^d(\mc{I}/\mc{I}^2))(d) 
\end{equation}
yields a surjection
\begin{equation}
\label{eq:surj-K-to-PSymQ}
\begin{tikzcd}[column sep=20pt]
 K \oo \Sym^d(\mc{I}/\mc{I}^2) \oo_{\mc{O}_{\ol{\Gr}}} \Sym^q\ol{\mc{Q}} \ar[r, two heads]&
 \Sym^{d+1}(\mc{I}/\mc{I}^2) 
 \oo_{\mc{O}_{\ol{\Gr}}} (\Sym^q\ol{\mc{Q}})(1)\, .
 \end{tikzcd}
\end{equation}

\begin{lemma}
\label{lem:str-iso-1st-surj}
If $0\leq d\leq\ol{n}-3$ and $q>d$, the map 
\eqref{eq:surj-K-to-PSymQ} is surjective on global sections.
\end{lemma}

\begin{proof}
We begin by observing that the map \eqref{eq:ev-Gr} factors 
as the composition
 \begin{equation}
\label{eq:K-factor-UV}
\begin{tikzcd}[column sep=22pt]
K \oo \mc{O}_{\ol{\Gr}} \ar[r]& H^0(\ol{\Gr},U\oo\ol{\mc{Q}}) \oo
\mc{O}_{\ol{\Gr}} \ar[r, "\epsilon"]& U\oo\ol{\mc{Q}} ,
\end{tikzcd}
 \end{equation}
where the first map is induced by the surjection 
$H^0(\ol{\Gr},\op{ev}_{\ol{\Gr}})$ in Lemma~\ref{lem:ev-Gr-surj}, and the 
second one is the evaluation of sections. Recalling that 
$H^0(\ol{\Gr},U\oo\ol{\mc{Q}}) = U\oo\ol{V}$, we can construct the analogue of \eqref{eq:surj-K-to-PSymQ}, 
where $K$ is replaced by $U\oo\ol{V}$.
The map 
\[
K\oo \Sym^d(\mc{I}/\mc{I}^2) \oo_{\mc{O}_{\ol{\Gr}}} \Sym^q\ol{\mc{Q}}
\rightarrow U\oo\ol{V} \oo \Sym^d(\mc{I}/\mc{I}^2) \oo_{\mc{O}_{\ol{\Gr}}} \Sym^q\ol{\mc{Q}}
\]
inducing a surjection  
at the level of global sections, it then suffices to show that the map
\begin{equation}
\label{eq:surj-UV-to-PSymQ}
\begin{tikzcd}[column sep=22pt]
U\oo\ol{V} \oo \Sym^d(\mc{I}/\mc{I}^2) \oo_{\mc{O}_{\ol{\Gr}}} \Sym^q\ol{\mc{Q}}
\ar[r]& \Sym^{d+1}(\mc{I}/\mc{I}^2) 
 \oo_{\mc{O}_{\ol{\Gr}}} (\Sym^q\ol{\mc{Q}})(1)
 \end{tikzcd}
\end{equation}
(constructed in the same way as \eqref{eq:surj-K-to-PSymQ}) induces a surjection
on global sections. We will show that \eqref{eq:surj-UV-to-PSymQ} arises 
naturally from a Koszul-type complex, and prove it induces a surjection 
on global sections by analyzing a hypercohomology spectral sequence. The map 
$\epsilon$ in \eqref{eq:K-factor-UV} is part of the short exact sequence
\begin{equation}
\label{eq:uvv-bar}
\begin{tikzcd}[column sep=20pt]
0 \ar[r]& U \oo \ol{\mc{U}}\ar[r]& U \oo \ol{V} \oo \mc{O}_{\ol{\Gr}}
\ar[r, "\epsilon"]& U \oo \ol{\mc{Q}}\ar[r]&0\, ,
\end{tikzcd}
\end{equation}
obtained from the tautological sequence \eqref{eq:tautologica-GGbar} on 
$\ol{\Gr}$ by tensoring with $U$. We can think of this sequence as a 
resolution of $U \oo \ol{\mc{U}}$ by a $2$-term complex. Taking the 
$(d+1)$-st exterior power of this complex, we obtain a resolution of 
$\bw^{d+1}(U \oo \ol{\mc{U}})$ given by the complex $\mc{B}^{\bullet}$,
\begin{equation}
\label{eq:uv-seq-bbullet}
\begin{tikzcd}[column sep=20pt, row sep=3pt]
\bw^{d+1}(U \oo \ol{V}) \oo \mc{O}_{\ol{\Gr}} \ar[r]
& \bw^d(U \oo \ol{V})\oo (U \oo \ol{\mc{Q}}) \rar[shorten >= 43pt] 
&\hspace*{-40pt}\cdots
\\
\hspace*{50pt}\cdots\rar[shorten <=15pt]
&  U \oo \ol{V} \oo \Sym^d(U \oo \ol{\mc{Q}}) \ar[r]
&\Sym^{d+1}(U \oo \ol{\mc{Q}}) \,,
\end{tikzcd}
\end{equation}
where $\mc{B}^i = \bw^{d+1-i}(U \oo \ol{V}) \oo \Sym^i(U \oo \ol{\mc{Q}})$ for
$i=0,\dots,d+1$. 
We obtain an identification,
\begin{multline}
\label{eq:sh-cohom=hyper-cohom}
H^{d+1}\Big(\ol{\Gr},\bw^{d+1}(U \oo \ol{\mc{U}}) \oo 
(\Sym^q\ol{\mc{Q}})(-d)\Big) = 
\bb{H}^{d+1}\left(\mc{B}^{\bullet}
\oo_{\mc{O}_{\ol{\Gr}}} (\Sym^q\ol{\mc{Q}})(-d)\right).
\end{multline}

\begin{claim*}
The following vanishing statements hold:
\begin{equation}
\label{eq:olGr-vanH}
H^{d+1}\Bigl(\ol{\Gr},\bw^{d+1}(U \oo \ol{\mc{U}}) \oo
(\Sym^q\ol{\mc{Q}})(-d)\Bigr)=0,
\end{equation}
and
\begin{equation}
\label{eq:olGr-vanHH}
H^j\bigl(\ol{\Gr},\mc{B}^i  \oo \Sym^q\ol{\mc{Q}}(-d)\bigr) = 0\mbox{ for }j>0\mbox{ and all }i.
\end{equation}
\end{claim*}

Assuming the claim, we complete the proof of the lemma. 
Using \eqref{eq:olGr-vanHH}, it follows that the 
hypercohomology groups of $\mc{B}^{\bullet} 
\oo_{\mc{O}_{\ol{\Gr}}} (\Sym^q\ol{\mc{Q}})(-d)$ are the 
same as the cohomology groups of the complex
\begin{equation}
\label{eq:bbullet-h0}
B^{\bullet} = H^0\left(\ol{\Gr},\mc{B}^{\bullet}
\oo_{\mc{O}_{\ol{\Gr}}} (\Sym^q\ol{\mc{Q}})(-d)\right).
\end{equation}
Using \eqref{eq:sh-cohom=hyper-cohom} and \eqref{eq:olGr-vanH},
it follows further that $\mc{H}^{d+1}(B^{\bullet}) = 0$; 
that is, the last differential  in $B^{\bullet}$, which we 
denote $\xi\colon B^d \lra B^{d+1}$, is surjective. 

Moreover, using the identification \eqref{eq:Symd-I/I2-P},
we obtain that
\begin{equation}
    \begin{aligned}
\label{eq:mcbd-gr-sym}
\mc{B}^d  \oo_{\mc{O}_{\ol{\Gr}}}
(\Sym^q\ol{\mc{Q}})(-d)& = U\oo\ol{V} \oo \Sym^d(\mc{I}/\mc{I}^2)
\oo_{\mc{O}_{\ol{\Gr}}} \Sym^q\ol{\mc{Q}} ,
\\ 
\mc{B}^{d+1} \oo_{\mc{O}_{\ol{\Gr}}}
(\Sym^q\ol{\mc{Q}})(-d) & = \Sym^{d+1}(\mc{I}/\mc{I}^2) 
\oo_{\mc{O}_{\ol{\Gr}}} (\Sym^q\ol{\mc{Q}})(1) \, ,
\end{aligned}
\end{equation}
and $\xi$ is the map induced on global sections by 
\eqref{eq:surj-UV-to-PSymQ}, which is what we wanted to prove.

We are left to verify \eqref{eq:olGr-vanH} and 
\eqref{eq:olGr-vanHH}. Since 
$\Sym^q\ol{\mc{Q}}(-d)=
 \bb{S}_{(q-d,-d)}\ol{\mc{Q}}$,
 it follows from \eqref{eq:cauchy} that 
$\bw^{d+1}(U \oo \ol{\mc{U}}) 
 \oo (\Sym^q\ol{\mc{Q}})(-d)$ 
decomposes as a direct sum of copies of 
$\bb{S}_{(q-d,-d)}\ol{\mc{Q}} \oo 
\bb{S}_{\b}\ol{\mc{U}}$,
where $\abs{\b}=d+1$. Our hypothesis guarantees that we 
can apply Lemma~\ref{lem:q-dd-beta-vanishing} to conclude 
that $H^{d+1}(\ol{\Gr},\bb{S}_{(q-d,-d)} \ol{\mc{Q}} 
\oo \bb{S}_{\b}\ol{\mc{U}})=0$, which 
implies \eqref{eq:olGr-vanH}. The vanishing in 
\eqref{eq:olGr-vanHH} follows from \eqref{eq:Sa-Q-only-0} by 
noting that  the tensor product
$\mc{B}^i \oo_{\mc{O}_{\ol{\Gr}}}
(\Sym^q\ol{\mc{Q}})(-d)$ decomposes (using for instance
\cite{weyman}*{Theorem~2.3.2(a) and Corollary~2.3.5}) into 
a direct sum of copies of $\SS_{(\a_1,\a_2)}\ol{\mc{Q}}$ 
with $\a_2\geq -d\geq-(\ol{n}-2)$.
\end{proof}

\section{The blow-up of a sub-Grassmannian}
\label{sect:blowup-Gr}

As before, let $\ol{V}^{\vee}\subseteq V^{\vee}$ be an $\overline{n}$-dimensional subspace, and set 
$N=2(n-\overline{n})$. Let
$\pi\colon V\onto \ol{V}$ be the corresponding projection and, again  
$U = \ker(\pi)$. We let $\wtl{\Gr}$ denote the blow-up of 
$\Gr=\Grass_2(V^{\vee})$ along 
$\ol{\Gr}=\op{Gr}_2(\ol{V}^{\vee})$, and write
$E \hookrightarrow\wtl{\Gr}$ for the exceptional divisor 
of the blow-up, that is, the projectivized normal bundle of 
$\ol{\Gr}$ in~$\Gr$, given in \eqref{eq:def-ProjN}. 

Let $\gamma\colon \wtl{\Gr} \longrightarrow \Gr$ be the blow-down morphism, whose restriction to the exceptional 
divisor is the map $\ol{\gamma}$ in \eqref{eq:def-ol-gamma}. 
Writing $\mc{I}$ for the ideal sheaf
of $\ol{\Gr}$ inside $\Gr$, we have that
\begin{equation}
\label{eq:gamma*-dE}
\gamma_*(\mc{O}_{\wtl{\Gr}}(dE)) = \begin{cases}
 \mc{O}_{\Gr} & \mbox{if }d\geq 0, \\[2pt]
 \mc{I}^{-d} & \mbox{if }d<0.
\end{cases}
\end{equation}

For higher direct images we have the following vanishing property, see  \cite{CR15}*{Theorem~1.1}
\begin{equation}
\label{eq:Rgamma*-O=0}
R^i\gamma_*\mc{O}_{\wtl{\Gr}}=0\mbox{ for all } i\geq 1,
\end{equation}
and more generally we have the following.

\begin{lemma}
\label{lem:RN-1*-OdE}
If $i\geq 1$ and $d\in\Z$, then $R^i\gamma_*(\mc{O}_{\wtl{\Gr}}(dE)) = 0$,
unless $i=N-1$ and $d\geq N$, in which case we have a short exact sequence
{\small
\begin{equation*}
\label{eq:ses-RN-1}
\begin{tikzcd}[column sep=13pt]
0 \ar[r]& R^{N-1}\gamma_*\bigl(\mc{O}_{\wtl{\Gr}}((d-1)E)\bigr) \ar[r]&
R^{N-1}\gamma_*\bigl(\mc{O}_{\wtl{\Gr}}(dE)\bigr)\ar[r]&
\bigl(\Sym^{d-N}(U^{\vee}) \oo \ol{\mc{Q}}\bigr) (n-\ol{n}) \ar[r]& 0\, .
\end{tikzcd}
\end{equation*}
}
\end{lemma}

\begin{proof}  When $d=0$, the conclusion follows from \eqref{eq:gamma*-dE}.
We prove the assertion by induction on (the absolute value of) $d$, 
considering the  exact sequence
\begin{equation}
\label{eq:ses-res-E}
\begin{tikzcd}[column sep=18pt]
0 \ar[r]& \mc{O}_{\wtl{\Gr}}((d-1)E)\ar[r]& \mc{O}_{\wtl{\Gr}}(dE) \ar[r]
&  \mc{O}_{E}(dE) \ar[r]&0 ,
\end{tikzcd}
\end{equation}
and the associated long exact sequence obtained by applying $R\gamma_*$.
When $d>0$, we have by induction that 
$R^i\gamma_*(\mc{O}_{\wtl{\Gr}}((d-1)E)) = 0$
if $1\leq i\neq (N-1)$ or if $i=N-1$ and $d-1<N$, and a similar vanishing 
holds for $R^i\ol{\gamma}_*$ applied to $\mc{O}_{E}(dE)$ by 
\eqref{eq:Rgamma*-OE}. This, together with the expression for 
$R^{N-1}\gamma_*(\mc{O}_{E}(dE))$ for $d\geq N$
from \eqref{eq:Rgamma*-OE} completes the inductive step.

When $d\leq 0$, we assume by induction that 
$R^i\gamma_*(\mc{O}_{\wtl{\Gr}}(dE)) = 0$
for $i\geq 1$, and would like to conclude that the same vanishing holds for
$R^i\gamma_*(\mc{O}_{\wtl{\Gr}}((d-1)E))$. When $i>1$ this follows from
the long exact sequence obtained by applying $R\gamma_*$ to \eqref{eq:ses-res-E}
and from \eqref{eq:Rgamma*-OE}, while for $t=1$ it follows if we can prove the
exactness of the sequence
\[
\begin{tikzcd}[column sep=18pt]
0 \ar[r]& \gamma_*(\mc{O}_{\wtl{\Gr}}((d-1)E))\ar[r]& \gamma_*(\mc{O}_{\wtl{\Gr}}(dE))
\ar[r]&\ol{\gamma}_*(\mc{O}_{E}(dE)) \ar[r]& 0 \, .
\end{tikzcd}
\]
In view of \eqref{eq:gamma*-dE} and \eqref{eq:gamma*-OE}, the above sequence
can be rewritten as
\begin{equation}
\label{eq:ses-mcd}
\begin{tikzcd}[column sep=18pt]
0 \ar[r]& \mc{I}^{1-d} \ar[r]& \mc{I}^{-d} \ar[r]&\Sym^{-d}(\mc{I}/\mc{I}^2)\ar[r]& 0 ,
\end{tikzcd}
\end{equation}
which is exact because $\mc{I}$ defines a regular subvariety  of $\Gr$.
\end{proof}

\subsection{The blow-up of the sub-Grassmannian of a strongly isotropic component}
\label{subsec:blow-Grass}

Assume  $\ol{V}^{\vee}\subseteq V^{\vee}$ is a strongly
isotropic component of $\mc{R}(V,K)$, so that $K\subseteq\ker\bigl(\bwedge^2\pi\bigr)$ and the projection $p_M\colon K\to M$ is surjective. 
The base locus of $|K|$ on $\Gr$ contains $\ol{\Gr}$, 
so if we let $\wtl{\Gr}$ denote the blow-up of $\Gr$ along $\ol{\Gr}$,
then the base locus of $(K,\mc{O}_{\ol{\Gr}}(H))$ contains $E$. The evaluation  map
$K\oo\mc{O}_{\wtl{\Gr}} \lra \mc{O}_{\ol{\Gr}}(H)$ factors through $\mc{O}_{\ol{\Gr}}(H-E)$,
and we construct a Koszul complex $\wtl{\mc{K}}$ as expalined in \eqref{eq:def-tlK}. 
Note that this complex may not be exact, since the base locus of 
$(K,\mc{O}_{\ol{\Gr}}(H))$ may be  larger than $\ol{\Gr}$.
Recall that $2\leq\ol{n}=\dim(\ol{V})<n$ and that $N=2(n-\ol{n})$ is the rank of the conormal bundle
$\mc{I}/\mc{I}^2=U \oo \ol{\mc{Q}}^{\vee}$. We define for each $d\geq 0$ the complex
\begin{equation}
\label{eq:olK-d}
\ol{\mc{K}}_d^{\bullet} = R^{N-1}\gamma_*\bigl(\wtl{\mc{K}}^{\bullet}(-dE)\bigr).
\end{equation}
This is a complex supported on~$\ol{\Gr}$, whose terms are given 
(using the projection formula) by
\begin{equation}
\label{eq:def-olK}
 \ol{\mc{K}}_d^{-i} = \bwedge^i K \oo 
 \Bigl(R^{N-1}\gamma_*\bigl(\mc{O}_{\wtl{\Gr}}((i-1-d)E)\bigr)\Bigr)(1-i)
 \end{equation}
for $i\geq 0$ and $\ol{\mc{K}}_d^i = 0$ for $i>0$.

We fix $q\geq n-3$, and note that again by the projection formula and \eqref{eq:olK-d}
we have that
\begin{equation}
\label{eq:def-olK-bis}
\ol{\mc{K}}_d^{\bullet} \oo_{\mc{O}_{\ol{\Gr}}} \Sym^q\ol{\mc{Q}} =
R^{N-1}\gamma_*\left(\wtl{\mc{K}}^{\bullet}(-dE) \oo_{\mc{O}_{\wtl{\Gr}}}
\gamma^*(\Sym^q\mc{Q})\right).
\end{equation}
Applying the functor $H^j(\ol{\Gr},-)$ to the above complex of sheaves,
we obtain a complex of vector spaces
\begin{equation}
\label{eq:def-Fjd}
 F_{j,d}^{\bullet} = H^j(\ol{\Gr},\ol{\mc{K}}_d^{\bullet}
 \oo_{\mc{O}_{\ol{\Gr}}} \Sym^q\ol{\mc{Q}}).
\end{equation}

The goal of this section is to prove the following proposition, the case
$d=0$ of which will play a key role in the proof of Theorem~\ref{thm:main}.

\begin{proposition}
\label{prop:cohom-Fjd}
If $j,d\geq 0$, then $\mc{H}^{-i}(F_{j,d}^{\bullet}) = 0$,  for $i\leq N+j$.
\end{proposition}

It follows from Lemma~\ref{lem:RN-1*-OdE} that
\begin{equation}
\label{eq:vanish-olK-d-i}
\ol{\mc{K}}_d^{-i} = 0\mbox{ for }i\leq N + d.
\end{equation}
Consequently, the conclusion of Proposition~\ref{prop:cohom-Fjd} holds
trivially when $d\geq j$. For instance, if $j=0$ then the condition 
$d\geq j$ is automatic, so it is only interesting to consider the 
case when $j>0$. If $\ol{n}=2$, then $\ol{\Gr}$ is a point and from 
\eqref{eq:def-Fjd} we infer that $F_{j,d}^i = 0$ for all $j>0$, so the 
conclusion of Proposition~\ref{prop:cohom-Fjd} holds.
We may therefore assume that $3\leq\ol{n}<n$.

For each $d\geq 0$, we consider the short exact sequence of 
complexes on $\wtl{\Gr}$
\[
\begin{tikzcd}[column sep=20pt]
0 \ar[r]&\wtl{\mc{K}}^{\bullet}(-(d+1)E)\ar[r]& \wtl{\mc{K}}^{\bullet}(-dE)
\ar[r]& \wtl{\mc{K}}^{\bullet}(-dE)_{|E} \ar[r]& 0 \, .
\end{tikzcd}
\]
Tensoring this sequence with the locally free sheaf $\gamma^*(\Sym^q\mc{Q})$,
applying $R^{N-1}\gamma_*$, and using Lemma~\ref{lem:RN-1*-OdE}, we obtain
a short exact sequence of complexes on~$\ol{\Gr}$,
\begin{equation}
\label{eq:ses-KKC}
\begin{tikzcd}[column sep=20pt]
0 \ar[r]& \ol{\mc{K}}_{d+1}^{\bullet} \oo_{\mc{O}_{\ol{\Gr}}} \Sym^q\ol{\mc{Q}} \ar[r]&
\ol{\mc{K}}_d^{\bullet}  \oo_{\mc{O}_{\ol{\Gr}}} \Sym^q\ol{\mc{Q}} \ar[r]& \mc{C}_d^{\bullet} \ar[r]& 0 \, ,
\end{tikzcd}
\end{equation}
where $\mc{C}_d^{\bullet} = R^{N-1}\ol{\gamma}_*\big(\wtl{\mc{K}}^{\bullet}(-dE)|_E
\oo_{\mc{O}_E} \ol{\gamma}^*(\Sym^q\ol{\mc{Q}})\big)$ is a complex whose terms
are given (using the projection formula and \eqref{eq:Rgamma*-OE}) by
\[
\mc{C}_d^{-i} = \begin{cases}
\bw^i K \oo \Sym^{i-1-d-N}(U^{\vee} \oo \ol{\mc{Q}})
\oo_{\mc{O}_{\ol{\Gr}}} (\Sym^q\ol{\mc{Q}})(1-i+n-\ol{n}) & \mbox{if }i\geq d+N+1; \\[3pt]
0 & \mbox{if }i\leq d+N.
\end{cases}
\]

\begin{lemma}
\label{lem:vanishing-Hj-C}
 If $d\geq 0$, $3\leq\ol{n}<n$, and $H^j(\ol{\Gr},\mc{C}_d^{-i}) \neq 0$,
 then one of the following holds:
 \begin{enumerate}
 \item $j=\ol{n}-2$, or
 \item $j=2\cdot(\ol{n}-2)$ and $i\geq 2n-2$.
 \end{enumerate}
\end{lemma}

\begin{proof}
 Using Cauchy's formula and Pieri's formula, it follows that $\mc{C}_d^{-i}$
 decomposes as a direct sum of copies of $\bb{S}_{\a}\ol{\mc{Q}}$, where
 $\a=(\a_1,\a_2)$ satisfies
 \begin{align}
 \label{eq:ineq-mu}
 \a_1 &\geq 1-i+n-\ol{n}+q, \\
 \a_2 &\leq i-1-d-N + 1-i + n - \ol{n} = -d - n + \ol{n} < 0. \notag
 \end{align}

 Using Lemma~\ref{lem:Sa-Q-vanishing}, it follows that $H^j(\ol{\Gr},\bb{S}_{\a}\ol{\mc{Q}}) = 0$
 for $j\not\in\bigl\{0,\ol{n}-2,2\cdot(\ol{n}-2)\bigr\}$. Moreover, since $\a_2<0$, we have 
 $H^0(\ol{\Gr},\bb{S}_{\a}\ol{\mc{Q}}) = 0$, by Theorem~\ref{thm:bott}. It suffices to show that if $i\leq 2n-3$, then
 $H^{2\cdot(\ol{n}-2)}(\ol{\Gr},\bb{S}_{\a}\ol{\mc{Q}}) = 0$. Using
 \eqref{eq:ineq-mu}, we write $a_1 + \ol{n} \geq 1-i + n + (n-3) = 2n-2 - i > 0$,
so \eqref{eq:Sa-Q-only-n2} applies to yield the desired conclusion.
\end{proof}

\begin{corollary}
\label{cor:vanish-Fjd}
If $3\leq\ol{n}<n$, then 
 \[
 F^{-i}_{j,d} = 0\mbox{ for all }d\geq 0,\ i\in\Z,\mbox{ and }j\not\in
 \{\ol{n}-2,2\cdot(\ol{n}-2)\}.
 \]
Moreover, if $j=2\cdot(\ol{n}-2)$, then $F^{-i}_{2\cdot(\ol{n}-2),d} = 0$, for all $d\geq 0$  and $i\leq 2n-3$.
In particular, Proposition~\ref{prop:cohom-Fjd} holds for all $j\neq \ol{n}-2$.
\end{corollary}

\begin{proof} We prove the vanishing of $F^{-i}_{j,d}$ by descending induction
on $d$. If $d\geq i-N$ then we get that $F^{-i}_{j,d}=0$ using \eqref{eq:def-Fjd}
and \eqref{eq:vanish-olK-d-i}. Suppose now that $d\geq 0$ and $F^{-i}_{j,d+1} = 0$.
Applying the functor $H^j(\ol{\Gr},-)$ to \eqref{eq:ses-KKC}, we obtain an exact sequence
\begin{equation}
\label{eq:0f-i}
\begin{tikzcd}[column sep=20pt]
0=F^{-i}_{j,d+1} \ar[r]&  F^{-i}_{j,d} \ar[r]&  H^j(\ol{\Gr},\mc{C}_d^{-i}) \, .
\end{tikzcd}
\end{equation}
It follows from Lemma~\ref{lem:vanishing-Hj-C} that $F^{-i}_{j,d}=0$ if
$j\not\in\{\ol{n}-2,2\cdot(\ol{n}-2)\}$, or if $j=2\cdot(\ol{n}-2)$ and $i\leq 2n-3$,
thereby completing the inductive step.

To prove the last assertion, it suffices to note that if $j=2\cdot(\ol{n}-2)$ then
$N+j = 2n-4$, so $\mc{H}^{-i}(F_{j,d}^{\bullet}) = F^{-i}_{j,d}=0$ for $i\leq N+j$
(and in fact also for $i=N+j+1$).
\end{proof}

We are left with proving Proposition~\ref{prop:cohom-Fjd} in the case when
$j=\ol{n}-2>0$. Using \eqref{eq:vanish-olK-d-i}, we may therefore assume
that $d\leq\ol{n}-3$. If we truncate the short exact sequence \eqref{eq:ses-KKC}
in cohomological degrees $\bullet\geq-(2n-3)$ and apply the long exact sequence
in sheaf cohomology, then the vanishing from Corollary~\ref{cor:vanish-Fjd} 
yields a short exact sequence of complexes
\begin{equation}
\label{eq:ses-trunc-FFC}
\begin{tikzcd}[column sep=20pt]
0\ar[r]&  F_{\ol{n}-2,d+1}^{\bullet\geq-(2n-3)}
\ar[r]& F_{\ol{n}-2,d}^{\bullet\geq-(2n-3)} \ar[r]&
H^{\ol{n}-2}(\ol{\Gr},\mc{C}_d^{\bullet\geq-(2n-3)}) \ar[r]&0 \,.
\end{tikzcd}
\end{equation}
From the long exact sequence in cohomology associated to this short
exact sequence of complexes, we derive exact sequences
\begin{equation}
\label{eq:ses-hi-f2d}
\begin{tikzcd}[column sep=20pt]
\mc{H}^{-i}(F_{\ol{n}-2,d+1}^{\bullet}) \ar[r]&  \mc{H}^{-i}(F_{\ol{n}-2,d}^{\bullet})
\ar[r]& \mc{H}^{-i}(H^{\ol{n}-2}(\ol{\Gr},\mc{C}_d^{\bullet}))
\end{tikzcd}
\text{ for all $i\leq 2n-4$}.
\end{equation}
Notice that $N+\ol{n}-2 = 2n-\ol{n}-2\leq 2n-4$, so if we apply descending induction
on $d$, then the proof of Proposition~\ref{prop:cohom-Fjd} reduces to verifying that
\begin{equation}
\label{eq:need-vanishing-HHC}
 \mc{H}^{-i}(H^{\ol{n}-2}\bigl(\ol{\Gr},\mc{C}_d^{\bullet})\bigr) = 0, \mbox{ for }i\leq N+\ol{n}-2.
\end{equation}

To that end, we consider the hypercohomology spectral sequence associated
with the complex
\begin{equation}
\label{eq:mcg-bullet}
\mc{G}^{\bullet} = \left(\wtl{\mc{K}}^{\bullet}(-dE)\oo_{\mc{O}_{\wtl{\Gr}}}
\gamma^*\Sym^q\mc{Q}\right)_{|E} = (\wtl{\mc{K}}^{\bullet})_{|E} \oo \mc{O}_E(-dE) \oo \ol{\gamma}^*\Sym^q\ol{\mc{Q}}.
\end{equation}

\begin{lemma}\label{lem:G-exact}
 The complex $\mc{G}^{\bullet}$ is exact.
\end{lemma}

\begin{proof} 
Since the sheaf $\mc{O}_E(-dE) \oo
\ol{\gamma}^*(\Sym^q\ol{\mc{Q}})$ is locally free, 
it suffices to prove that $(\wtl{\mc{K}}^{\bullet})_{|E}$ 
is an exact complex. The formation of the Koszul 
complex being functorial, we can identify
$(\wtl{\mc{K}}^{\bullet})_{|E}$ with the Koszul complex 
on $E$ associated with the map 
\[
\begin{tikzcd}[column sep=18pt]
\psi\colon K\oo\mc{O}_E \ar[r]& \mc{O}_E(H-E),
\end{tikzcd}
\]
and the exactness of $(\wtl{\mc{K}}^{\bullet})_{|E}$
reduces to proving the surjectivity of $\psi$. To establish this 
claim, note that $\psi$ factors as the composition of the map
\begin{equation}
\label{eq:koo-mco}
\begin{tikzcd}[column sep=20pt]
K\oo\mc{O}_E \ar[r] & \ol{\gamma}^*((\mc{I}/\mc{I}^2)(1)) ,
\end{tikzcd}
\end{equation}
which is the pull-back of \eqref{eq:ev-Gr} and is therefore 
surjective, with the map
\[
\begin{tikzcd}[column sep=20pt]
\ol{\gamma}^*((\mc{I}/\mc{I}^2) (1)) =
\ol{\gamma}^*(\mc{O}_{\ol{\Gr}}(1)) \oo_{\mc{O}_E} \ol{\gamma}^*(\mc{I}/\mc{I}^2)
\ar[r]& \ol{\gamma}^*(\mc{O}_{\ol{\Gr}}(1))  \oo \mc{O}_E(-E)
\end{tikzcd}
\]
induced by the surjective quotient map 
$\ol{\gamma}^*(\mc{I}/\mc{I}^2) \to \mc{O}_E(-E)$. This completes the proof.
\end{proof}

The terms of the complex $\mc{G}^{\bullet}$ are described by
\begin{equation}
\label{eq:mcg-koo-ed}
\mc{G}^{-i} = \bwedge^i K \oo \mc{O}_E(-(d+1-i)E) \oo_{\mc{O}_{E}}
\ol{\gamma}^*\left((\Sym^q\ol{\mc{Q}})(1-i)\right),
 \end{equation}
and therefore by \eqref{eq:gamma*-OE} we have
\[
\ol{\gamma}_*\mc{G}^{-i} = \begin{cases}
\bw^i K \oo \Sym^{d+1-i}(U \oo \ol{\mc{Q}}^{\vee}) \oo_{\mc{O}_{\ol{\Gr}}}
(\Sym^q\ol{\mc{Q}})(1-i)
& \mbox{if }0\leq i\leq d+1, \\[2pt]
0 & \mbox{otherwise}.
\end{cases}
\]
Moreover, $R^j\ol{\gamma}_*\mc{G}^{\bullet} = 0$ for $(N-1)\neq j\geq 1$
by \eqref{eq:Rgamma*-OE}, and $R^{N-1}\ol{\gamma}_*\mc{G}^{\bullet} =
\mc{C}_d^{\bullet}$ by definition.

\begin{lemma}
\label{lem:vanishing-Hj-G}
 We have that $H^j(\ol{\Gr},\ol{\gamma}_*\mc{G}^{-i}) = 0$, unless 
 $j=0$ and $i=0,1$. Moreover, the induced~map
 $
 H^0(\ol{\Gr},\ol{\gamma}_*\mc{G}^{-1}) \longrightarrow H^0(\ol{\Gr},\ol{\gamma}_*\mc{G}^0)
 $
 is surjective if $d\leq\ol{n}-3$.
\end{lemma}

\begin{proof}
 Using Cauchy's formula and Pieri's formula as in the proof of 
 Lemma~\ref{lem:vanishing-Hj-C}, it follows that 
 $\ol{\gamma}_*\mc{G}^{-i}$ decomposes as a direct sum of copies of 
 $\bb{S}_{\a}\ol{\mc{Q}}$, where $\a=(\a_1,\a_2)$ satisfies
 \[
 1-i \geq \a_2 \geq 1-i - (d+1-i) = -d\geq -(\ol{n}-3).
 \]
 Since $\a_2 \geq -(\ol{n}-2)$, we have by \eqref{eq:Sa-Q-only-0} that
 $H^j(\ol{\Gr}, \bb{S}_{\a}\ol{\mc{Q}}) = 0$ for $j>0$. 
 Moreover, if $i\geq 2$, then $\a_2<0$, and therefore $H^0(\ol{\Gr}, 
 \bb{S}_{\a}\ol{\mc{Q}}) = 0$, from which the desired vanishing follows.

For the final assertion, we note that since $U \oo \ol{\mc{Q}}^{\vee} \cong 
 \mc{I}/\mc{I}^2$, we have that
\[
 \ol{\gamma}_*\mc{G}^{-1} = K \oo \Sym^d(\mc{I}/\mc{I}^2) \oo \Sym^q\ol{\mc{Q}} \ \ \mbox{ and } 
 \ \  \ol{\gamma}_*\mc{G}^0= \Sym^{d+1}(\mc{I}/\mc{I}^2) \oo (\Sym^q\ol{\mc{Q}})(1),
\]
and the differential is given by \eqref{eq:surj-K-to-PSymQ}. The 
surjectivity of the induced map on global sections then follows from 
Lemma~\ref{lem:str-iso-1st-surj}, since $q\geq n-3>\ol{n}-3$.
\end{proof}

We are now ready to complete the proof of Proposition~\ref{prop:cohom-Fjd}.

\begin{proof}[Proof of Proposition~\ref{prop:cohom-Fjd}]
Recall that we are left with proving the vanishing \eqref{eq:need-vanishing-HHC},
and that we have $q\geq n-3$, $0\leq d\leq\ol{n}-3$, and $3\leq\ol{n}<n$. We analyze the
hypercohomology spectral sequence
\begin{equation}
\label{eq:sp-seq-G}
  E_1^{-i,j} = H^j(E,\mc{G}^{-i})  \Longrightarrow \bb{H}^{-i+j}(\mc{G}^{\bullet}) = 0,
\end{equation}
where the hypercohomology vanishing comes from the exactness of 
$\mc{G}^{\bullet}$ in Lemma~\ref{lem:G-exact}. Using the Leray spectral 
sequence in order to compute
$H^j(E,\mc{G}^{-i})$, it follows from Lemmas~\ref{lem:vanishing-Hj-C}
and~\ref{lem:vanishing-Hj-G} that the only non-zero terms occur when
\begin{itemize}[itemsep=2pt]
 \item $j=0$ and $i=0,1$, in which case the differential $d_1^{-1,0} \colon
 E_1^{-1,0} \lra E_1^{0,0}$ is surjective by Lemma~\ref{lem:vanishing-Hj-G},
 and therefore $E_2^{0,0}=0$.
 \item $j=(N-1)+(\ol{n}-2)$, in which case
$E_1^{\bullet,N+\ol{n}-3} = H^{\ol{n}-2}(\ol{\Gr},\mc{C}_d^{\bullet})$.
 \item $j=(N-1) + 2\cdot(\ol{n}-2)=2n-5$, in which case
\begin{equation}
\label{eq:E1-2n-3-vanishing}
E_1^{-i,2n-5} = H^{2\cdot(\ol{n}-2)}(\ol{\Gr},\mc{C}_d^{-i}) = 0\mbox{ for }i\leq 2n-3.
\end{equation}
\end{itemize}
The vanishing in \eqref{eq:need-vanishing-HHC} is then equivalent to the assertion that
\begin{equation}
\label{eq:E2-in-vanishing}
E_2^{-i,N+\ol{n}-3}=0\mbox{ for }i\leq N+\ol{n}-2.
\end{equation}
Since the spectral sequence converges to $0$, it suffices to check that there are no
non-zero differentials with source or target $E_r^{-i,N+\ol{n}-3}$ for $r\geq 2$. If we
write $d_r^{i,j}:E_r^{i,j} \lra E_r^{i+r,j-r+1}$ for the differentials in the spectral sequence,
then the only potentially non-zero one with source $E_r^{-i,N+\ol{n}-3}$, $r\geq 2$,
may occur for $r=N+\ol{n}-2$, given by
\begin{equation}
\label{eq:dNoln-2-i=0}
\begin{tikzcd}[column sep=20pt]
d_{N+\ol{n}-2}^{-i,N+\ol{n}-3}\colon E_{N+\ol{n}-2}^{-i,N+\ol{n}-3} \ar[r]&
E_{N+\ol{n}-2}^{-i+N+\ol{n}-2,0}\, .
\end{tikzcd}
\end{equation}
Since $E_1^{-i,0} = 0$ for $i<0$, and $E_2^{0,0}=0$ as earlier noted, we conclude
that $E_2^{-t,0} = E_r^{-t,0}=0$ for all $t\leq 0$ and $r\geq 2$. Applying this to
$t=-i+N+\ol{n}-2 \geq 0$, we conclude that the target of the map \eqref{eq:dNoln-2-i=0}
vanishes, so in fact every differential with source $E_r^{-i,N+\ol{n}-3}$, $r\geq 2$, vanishes.
The only potentially non-zero differential with target $E_r^{-i,N+\ol{n}-3}$, where $r\geq 2$,
may occur when $r=\ol{n}-1$. However, in this case the differential is given by
\[
\begin{tikzcd}[column sep=20pt]
d_{\ol{n}-1}^{-i-\ol{n}+1,2n-5}  \colon E_{\ol{n}-1}^{-i-\ol{n}+1,2n-5}
\ar[r]& E_{\ol{n}-1}^{-i,N+\ol{n}-3}
\end{tikzcd}
\]
whose source vanishes by \eqref{eq:E1-2n-3-vanishing}, since
$i+\ol{n}-1 \leq N+2\cdot\ol{n}-3 = 2n-3$.
\end{proof}

\section{Proof of the main Theorem}
\label{sec:proof-main-2}

We now complete the proof of Theorem \ref{thm:main} and we recall the setup explained in the Introduction. 
With the notation as above, we form the following Koszul complex
$\mc{K}^{\bullet}$ on $\wtl{\Gr}$:
\begin{equation}
\label{eq:def-tlK'}
 \mc{K}^{-i} = \begin{cases}
 \bw^i K \oo \mc{O}_{\wtl{\Gr}}((1-i)(H-E))
 & \text{for $i>0$},\\[1pt]
 \mc{O}_{\wtl{\Gr}}(H) & \text{for $i=0$}, \\[1pt]
 0 & \text{otherwise}.
 \end{cases}
\end{equation}

If we write 
$\gamma_t\colon E_t \to \Gr_t$ for the restriction of $\gamma$ to $E_t$, 
then we infer from \eqref{eq:H0-K} that the only non-zero cohomology 
group of $\mc{K}^{\bullet}$ is $\mc{H}^0(\mc{K}^{\bullet}) \cong \bigoplus_{t=1}^k \mc{O}_{\ol{\Gr}_t}(H)$.

For the rest of the proof we fix $q\geq n-3$. 
Since 
$\Sym^q\mc{Q}_{|\Gr_t}=\Sym^q\mc{Q}_t$, 
we conclude from the above that
\begin{equation}
\label{eq:H0-K-oo-Symq}
\mc{H}^j\bigl(\mc{K}^{\bullet}\oo_{\mc{O}_{\wtl{\Gr}}}\gamma^*(\Sym^q\mc{Q})\bigr) =
\begin{cases}
\ds\bigoplus_{t=1}^k \gamma_t^*((\Sym^q\mc{Q}_t)(1))
& \text{if $j=0$}, \\
0 & \text{otherwise}.
\end{cases}
\end{equation}
Using the projection formula \cite{hartshorne}*{Exercise~III.8.3}, together with
the case $d=0$ of \eqref{eq:gamma*-OE} and the Leray spectral sequence
(as in \cite{hartshorne}*{Exercise~III.8.1}), it follows that
\begin{equation}
\label{eq:Hj-Et=Gt}
H^j\bigl(E_t,\gamma_t^*((\Sym^q\mc{Q}_t)(1)\bigr) \cong
H^j\bigl(\Gr_t,(\Sym^q\mc{Q}_t)(1)\bigr), \mbox{ for all }j.
\end{equation}

Moreover, since $\Sym^q\mc{Q}_t(1) =
\bb{S}_{(q+1,1)}\mc{Q}_t$, it follows from Theorem~\ref{thm:bott}
that the above cohomology groups vanish for $j>0$. Finally, it follows
from \cite{AFPRW}*{Lemma~3.4} that
\begin{equation}
\label{eq:Wq=H0-G}
H^0(\Gr_t,(\Sym^q\mc{Q}_t)(1)) =  W_q(V_t,0).
\end{equation}
Putting together \eqref{eq:H0-K-oo-Symq} with \eqref{eq:Hj-Et=Gt} and
\eqref{eq:Wq=H0-G}, we conclude that $\mc{K}^{\bullet} \oo \gamma^*\Sym^q\mc{Q}$
has only one non-zero hypercohomology group, namely
\begin{equation}
\label{eq:bbh0-mck}
\bb{H}^0\big(\mc{K}^{\bullet} \oo_{\mc{O}_{\wtl{\Gr}}} \gamma^*\Sym^q\mc{Q}\big) =
\bigoplus_{t=1}^k W_q(V_t,0).
\end{equation}

We compute this in a different way, using the hypercohomology spectral sequence
\begin{equation}
\label{eq:e1-ij-hjgr}
E_1^{-i,j} = H^j(\wtl{\Gr},\mc{K}^{-i} \oo_{\mc{O}_{\wtl{\Gr}}}
\gamma^*\Sym^q\mc{Q}) \Longrightarrow \bb{H}^{-i+j}(\mc{K}^{\bullet}
\oo_{\mc{O}_{\wtl{\Gr}}} \gamma^*\Sym^q\mc{Q}),
\end{equation}
while noting that unlike in the outline at the beginning of the section,
we do not work with the entire sheaf of algebra $\mc{S}=\Sym(\mc{Q})$,
but instead we are focusing on a single (degree $q$) component.

Using the fact that $\gamma_*\mc{O}_{\wtl{\Gr}} = \mc{O}_{\Gr}$ and
$R^j\gamma_*\mc{O}_{\wtl{\Gr}}=0$ for $j>0$ (as in \eqref{eq:gamma*-dE}
with $d=0$, and \eqref{eq:Rgamma*-O=0}), we can again apply the
projection formula and the Leray spectral sequence to conclude that
the following equalities hold
\begin{align}
\label{eq:E1-i=01}
H^j(\wtl{\Gr},\mc{K}^{-1} \oo_{\mc{O}_{\wtl{\Gr}}}
\gamma^*\Sym^q\mc{Q}) &= H^j(\Gr, K \oo
\Sym^q\mc{Q}),\\
H^j(\wtl{\Gr},\mc{K}^0 \oo_{\mc{O}_{\wtl{\Gr}}} \gamma^*\Sym^q\mc{Q})& =
H^j(\Gr, (\Sym^q\mc{Q})(1))\notag
\end{align}
for all $j$. Using Theorem~\ref{thm:bott} we infer that these groups
vanish for $j>0$, and so
\begin{align}
\label{eq:E1j-i=01}
E_1^{-i,j}&=0 \qquad \text{for $i=0,1$ and $j>0$},\notag
\\
E_1^{-1,0}& = H^0(\Gr,K \oo \Sym^q\mc{Q}) = K \oo \Sym^q V,
\\
E_1^{0,0}&=H^0(\Gr,(\Sym^q\mc{Q})(1)) =
\bb{S}_{(q+1,1)}V.\notag
\end{align}

Recalling that $\bb{S}_{(q+1,1)}V$ is the kernel of the multiplication map
\[
V\oo\Sym^{q+1} V \rightarrow \Sym^{q+2}V,
\]
we obtain the following identification 
\begin{align}
\label{eq:wqvk-e2}
W_q(V,K)& = \coker\Bigl\{H^0(\Gr,K \oo \Sym^q\mc{Q}) \lra H^0(\Gr, \Sym^q\mc{Q}(1))\Bigr\} 
= \coker(d_1^{-1,0}) = E_2^{0,0}, \notag
\end{align}
where we use the notation \eqref{eq:maps-drij} for maps in
spectral sequences. The map 
\begin{equation}
\label{eq:wmap-bis}
\begin{tikzcd}[column sep=20pt]
W_q(V,K) \ar[r]& \bigoplus_{t=1}^k W_q(V_t,0)
\end{tikzcd}
\end{equation}
discussed in ({\ref{subsec:outline}) factors as the composition $E_2^{0,0} \onto E_{\infty}^{0,0} \hookrightarrow
\bb{H}^0(\mc{K}^{\bullet} \oo_{\mc{O}_{\wtl{\Gr}}} \gamma^*\Sym^q\mc{Q})$. 
Moreover, as mentioned in \eqref{subsec:outline}, in order to prove 
Theorem~\ref{thm:main}, it suffices to verify that 
\begin{equation}
\label{eqn:vanishing-SpecSeq}
E_2^{-i,i} = E_2^{-i-1,i} = 0  
\end{equation}
for $i\neq 0$. 
To see that this is indeed enough, we note that the vanishing $E_2^{-i,i} = 0$
forces $E_{\infty}^{-i,i} = 0$ for $i\neq 0$, and therefore the edge homomorphism in
\eqref{eq:E2-Einf-edge} is an isomorphism in degree $q$, i.e.,
\begin{equation}
\label{eq:Einf-00}
E_{\infty}^{0,0} = \bigoplus_{t = 1}^k W_q(V_t,0).
\end{equation}
Moreover, the vanishing $E_2^{-i-1,i} = 0$ for $i\neq 0$ implies that
$E_r^{-r,r-1} = 0$ for $r\geq 2$ and in particular the differential
$d_r^{-r,r-1}\colon E_r^{-r,r-1}\rightarrow E_r^{0,0}
$
is identically $0$ in degree $q$, which implies that $E_{r+1}^{0,0} = E_r^{0,0}$
for all $r\geq 2$, so that the quotient map in \eqref{eq:E2-Einf-edge} 
is in fact an isomorphism.

The vanishing \eqref{eqn:vanishing-SpecSeq} is immediate for
$i<0$ since $E_1^{-i,j}=0$ in this case. We will then prove the 
stronger vanishing statement mentioned, namely,
\begin{equation}
\label{eq:vanishing-E2-ij-bis}
\text{$E_2^{-i,j} = 0$  if $j>0$ and $i\leq j+1$}.
\end{equation}

Since $\dim(\wtl{\Gr})= 2n-4$, it follows  that $E_1^{-i,j} = 0$
for $j>2n-4$. To prove the vanishing \eqref{eq:vanishing-E2-ij}, we 
show for each $1\leq j\leq 2n-4$ that one of the following holds:
\begin{itemize}[itemsep=2pt]
 \item $E_1^{-i,j} = 0$ for all $i\leq j+1$.
 \item $E_1^{\bullet,j}$ is a complex whose homology vanishes in degree $-i$
 for $i\leq j+1$.
\end{itemize}
In order to compute the homology of $E_1^{\bullet,j}$, we will need to 
understand the groups $E_1^{-i,j}$ for $i\leq j+2$.
In particular, we are only interested in the values of $i$ for which 
$i\leq j+2\leq 2n-2$, and $i=2n-2$ will only be relevant when $j=2n-4$. 
To warn the reader of a somewhat awkward case analysis that will occur, 
we note that in our proofs below we will treat the case $j=2n-4$ separately, 
and then reduce to the range $i\leq 2n-3$, where more vanishing of cohomology 
occurs, allowing us to give a more uniform proof of 
\eqref{eq:vanishing-E2-ij} in that case.

We compute the groups $E_1^{-i,j}$ using the Leray spectral sequence.
Note that we have already done this for $i=0,1$ in \eqref{eq:E1j-i=01}
and \eqref{eq:E1-i=01}, so we will assume when needed that $i\geq 2$,
and in particular that $\mc{K}^{-i}=\wtl{\mc{K}}^{-i}$. Using the projection
formula, together with \eqref{eq:gamma*-dE} and \eqref{eq:def-tlK'},
we have that
\begin{equation}
\label{eq:gamma*-K}
 \gamma_*\left(\mc{K}^{-i} \oo_{\mc{O}_{\wtl{\Gr}}} \gamma^*\Sym^q\mc{Q}\right) =
 \bwedge^i K \oo (\Sym^q\mc{Q})(1-i)\quad\mbox{ for }i\geq 2.
\end{equation}

\begin{lemma}
\label{lem:Hu-gamma*-K}
 If $i\geq 2$ then $H^u\big(\Gr,\gamma_*\big(\mc{K}^{-i}
 \oo_{\mc{O}_{\wtl{\Gr}}} \gamma^*\Sym^q\mc{Q}\big)\big)=0$
 if both of the conditions below fail:
 \begin{itemize}[itemsep=2pt]
  \item $u=n-2$ and $i\geq n$; 
  \item $u=2n-4$ and $i\geq 2n-2$.
 \end{itemize}
\end{lemma}

\begin{proof}
 We use \eqref{eq:gamma*-K}, note that $\Sym^q\mc{Q})(1-i) 
 \bb{S}_{(q+1-i,1-i)}\mc{Q}$ and apply Lemma~\ref{lem:Sa-Q-vanishing}. The only potentially
 non-zero cohomology occurs for $u=0,n-2$, or $2n-4$. Since $i\geq 2$ 
 we have $1-i<0$ and therefore $H^0(\Gr,\bb{S}_{(q+1-i,1-i)}\mc{Q})=0$. If
 $H^{n-2}(\Gr,\bb{S}_{(q+1-i,1-i)}\mc{Q})\neq 0$ then it follows
 from \eqref{eq:Sa-Q-only-0} that $1-i < -(n-2)$, or equivalently
 $i\geq n$. If $H^{2n-4}(\Gr,\bb{S}_{(q+1-i,1-i)}\mc{Q})\neq 0$ then,
 since $1-i<0$, it follows from \eqref{eq:Sa-Q-only-n2} that $q+1-i\leq -n$,
 that is $i\geq q+n+1$. Since $q\geq n-3$, this yields $i\geq 2n-2$.
\end{proof}

In order to complete the proof of Theorem~\ref{thm:main}, 
there are two cases to consider.

\begin{case} 
{\it The case when $\mc{R}(V,K)$ is irreducible $(k=1)$.}
\end{case}
\label{case:irred-res}
We assume that $\mc{R}(V,K) = \ol{V}^{\vee}$, and write 
$\ol{n}=\dim\ol{V} \geq 2$ and $N = 2(n-\ol{n})$. If $\ol{n} = n$, 
then the isotropicity condition implies that $K = 0$,
and the conclusion of Theorem~\ref{thm:main} holds trivially.
We will therefore assume that $\ol{n}<n$. For $i\geq 2$ we compute the
terms $E_1^{-i,j}$ using the Leray spectral sequence,
 \begin{align}
 \label{eq:Leray-E-tilde}
  \widetilde{E}_2^{u,v} = &H^u\left(\Gr,R^v\gamma_*(\mc{K}^{-i}
  \oo_{\mc{O}_{\wtl{\Gr}}} \gamma^*\Sym^q\mc{Q})\right)  \Longrightarrow \\
&\hspace*{1in} H^{u+v}\big(\wtl{\Gr},\mc{K}^{-i}
  \oo_{\mc{O}_{\wtl{\Gr}}} \gamma^*\Sym^q\mc{Q}\big) = E_1^{-i,u+v}. \notag
 \end{align}

Using Lemma~\ref{lem:RN-1*-OdE}, we have that $\widetilde{E}_2^{u,v} = 0$
for $v\neq 0,N-1$, and the vanishing of the groups $\widetilde{E}_2^{u,0}$ has
been described in Lemma~\ref{lem:Hu-gamma*-K}. For $v=N-1$, we
have $\widetilde{E}_2^{u,N-1} = F^{-i}_{u,0}$ using the notation \eqref{eq:def-Fjd}
and the fact that $i>0$ (so that $\wtl{\mc{K}}^{-i} = \mc{K}^{-i}$). 
Using Corollary~\ref{cor:vanish-Fjd}, in addition to the previous observations,
we conclude that the only potentially non-zero groups $\widetilde{E}_2^{u,v}$ 
occur in the following four cases, which we now analyze separately.
\begin{itemize}[itemsep=2pt]
 \item $(u,v)=(n-2,0)$ if $i\geq n$. We then have $u+v=n-2<2n-4$.
 \item $(u,v)=(2n-4,0)$ if $i\geq 2n-2$. We then have $u+v = 2n-4$.
 \item $(u,v) = (\ol{n}-2,N-1)$. We then have $u+v = 2n-\ol{n}-3 <2n-4$.
 \item $(u,v) = (2\ol{n}-4,N-1)$ if $i\geq 2n-2$. We then have $u+v = 2n-5 < 2n-4$.
\end{itemize}

It follows that the only non-vanishing groups $E_1^{-i,2n-4}$ can occur when
$i\geq 2n-2$, which proves \eqref{eq:vanishing-E2-ij} for $j=2n-4$. As explained
earlier, we may assume from now on that $j\leq 2n-5$ and $i\leq 2n-3$,
and in particular $\widetilde{E}_2^{u,v}$ may only be non-zero when $(u,v) = (n-2,0)$
or $(u,v)=(\ol{n}-2,N-1)$. It follows that the spectral sequence 
\eqref{eq:Leray-E-tilde} degenerates, and we have $E_1^{-i,j} = 0$ unless
$j=n-2$ and $i\geq n$, or $j=N+\ol{n}-3=2n-\ol{n}-3$.

Suppose first that $\ol{n}<n-1$, so that $n-2 \neq 2n-\ol{n}-3$. 
For $j=n-2$ we infer that $E_1^{-i,n-2} = 0$ when $i\leq n-1$, 
proving \eqref{eq:vanishing-E2-ij}, while for $j = N+\ol{n}-3$, 
we have used the notation \eqref{eq:def-Fjd} that
\begin{equation}
\label{eq:E1-ij-hif}
E_1^{-i,N+\ol{n}-3} = \widetilde{E}_2^{\ol{n}-2,N-1} = F_{\ol{n}-2,0}^{-i}
\end{equation}
for $i\leq 2n-3$, and in particular for $i\leq N+\ol{n}-1$.
It follows from Proposition~\ref{prop:cohom-Fjd} that
$E_2^{-i,j} = \mc{H}^{-i}(F_{\ol{n}-2,0}^{\bullet}) = 0$  for $i\leq N + \ol{n} - 2 = j+1$,
as desired.

Suppose now that $\ol{n}=n-1$, so that $n-2 = 2n-\ol{n}-3$. We set $j=n-2$ and
analyze the complex $E_1^{\bullet,j}$. For $i\leq n-1$ we have as before that
$E_1^{-i,j} = F_{\ol{n}-2,0}^{-i}$, while for $i=n$ we derive from the spectral
sequence \eqref{eq:Leray-E-tilde} a natural short exact sequence
\begin{equation}
\label{eq:E2-n-20}
\begin{tikzcd}[column sep=20pt]
0 \ar[r]& \widetilde{E}_2^{n-2,0}\ar[r]& E_1^{-n,j} \ar[r]& 
\widetilde{E}_2^{n-3,1} = F_{\ol{n}-2,0}^{-n} \ar[r]& 0 .
\end{tikzcd}
\end{equation}
We obtain from this a commutative diagram
\begin{equation}
\label{eq:e1-nj}
\begin{tikzcd}[column sep=20pt]
E_1^{-n,j} \ar[r] \ar[d, two heads] & E_1^{-n+1,j} \ar[r] 
\ar[equal]{d} & \cdots \ar[r] 
& E_1^{0,j} \ar[equal]{d}\phantom{\, .}\\
F_{\ol{n}-2,0}^{-n} \ar[r] &F_{\ol{n}-2,0}^{-n+1} \ar[r] 
& \cdots \ar[r] & F_{\ol{n}-2,0}^0 \,.
\end{tikzcd}
\end{equation}

Recall now from Proposition~\ref{prop:cohom-Fjd} that 
$\mc{H}^{-i}(F_{\ol{n}-2,0}^{\bullet}) = 0$
for $i\leq N + \ol{n} - 2 = n-1$. It follows that the 
same vanishing holds for the cohomology of
$E_1^{\bullet,j}$, that is $E_2^{-i,j} = 0$ for $i\leq n-1 = j+1$, 
concluding our proof in the case when the resonance is irreducible.

\begin{case}
{\it The case when $\mc{R}(V,K)$ has at least two components $(k\geq 2)$.}
\end{case}
\label{case:two-comps}
We let $n_t = \dim(V_t)$ for $1\leq t\leq k$, and note that 
$n_t\leq n-2$ for all $t$. Indeed, since $n_t\geq 2$ for all $t$, 
a component of $\mc{R}(V,K)$ of dimension
$n-1$ or higher would meet every other component non-trivially, which is a contradiction.
It follows that if we let
$N_t = 2(n-n_t)$ denote the codimension of $\Gr_t$ inside $\Gr$, then
\begin{equation}
\label{eq:n-2-ineq}
n-2 < N_t + n_t - 3 = 2n - n_t - 3 < 2n-4.
\end{equation}

For $v>0$, the higher direct images $R^v\gamma_*\left(\mc{K}^{-i} 
\oo_{\mc{O}_{\wtl{\Gr}}} \gamma^*\Sym^q\mc{Q}\right)$ are given 
by sheaves supported on the base locus $\B$. We can isolate the 
contributions of each of the connected components $\Gr_t$ of $\B$,
as follows. We factor the map $\gamma\colon \wtl{\Gr} \lra \Gr$ 
as the composition
\begin{equation}
\label{eq:gr-phi-t}
\begin{tikzcd}[column sep=24pt]
\wtl{\Gr} \ar[r, "\phi_t"]&  \wtl{\Gr}_t \ar[r, "\tl{\gamma}_t"] 
&  \Gr \, ,
\end{tikzcd}
\end{equation}
where $\tl{\gamma}_t$ is the blow-up of $\Gr_t$, and $\phi_t$ denotes the blow-up
of the (preimage via $\tl{\gamma}_t$ of the) remaining components of $\B$.
We let $\wtl{\mc{K}}_t$ denote the Koszul complex on $\wtl{\Gr}_t$ associated
as in Section~\ref{subsec:blow-Grass} with the map
$K\oo\mc{O}_{\wtl{\Gr}_t} \lra \mc{O}_{\wtl{\Gr}_t}(H-E_t)$.
Applying \eqref{eq:gamma*-dE} to the blow-up map $\phi_t$ (whose
exceptional divisor is $E_1\sqcup\cdots\sqcup\hat{E_t}\sqcup\cdots \sqcup E_k$),
we infer that
\begin{equation}
\label{eq:phi-t-kkt}
(\phi_t)_* \wtl{\mc{K}}^{-i} = \wtl{\mc{K}}_t^{-i}\mbox{ for all }i>0 \, .
\end{equation}

Since $\wtl{\mc{K}}^{-i} = \mc{K}^{-i}$ for $i>0$, we conclude using the
projection formula that
\begin{equation}
\label{eq:higher-R*-K}
 R^v\gamma_*(\mc{K}^{-i} \oo_{\mc{O}_{\wtl{\Gr}}} \gamma^*\Sym^q\mc{Q}) =
 \bigoplus_{t=1}^k R^v\tl{\gamma}_{t\,*}(\wtl{\mc{K}}_t^{-i} \oo_{\mc{O}_{\wtl{\Gr}_t}}
 \tl{\gamma}_t^*\Sym^q\mc{Q})
\end{equation}
for all $i,v>0$. In view of Lemma~\ref{lem:RN-1*-OdE}, we obtain that
\begin{equation}
\label{eq:Rv-gammat*=0}
\text{$R^v\tl{\gamma}_{t\,*}(\wtl{\mc{K}}_t^{-i} \oo_{\mc{O}_{\wtl{\Gr}_t}}
\tl{\gamma}_t^*\Sym^q\mc{Q}) = 0$ for $v>0$ and $v\neq N_t-1$}\, .
\end{equation}

If we assume that $i\geq 2$ and compute each of the terms $E_1^{-i,j}$ using
the Leray spectral sequence \eqref{eq:Leray-E-tilde}, then it follows as in the
case when $\mc{R}(V,K)$ was irreducible that the only potentially non-zero
terms $\widetilde{E}_2^{u,v}$ occur when
\begin{itemize}[itemsep=2pt]
 \item $(u,v)=(n-2,0)$ if $i\geq n$, in which case $u+v=n-2<2n-4$.
 \item $(u,v)=(2n-4,0)$ if $i\geq 2n-2$, in which case  $u+v = 2n-4$.
 \item $(u,v) = (n_t-2,N_t-1)$ for some $t$, in which case  $n-2< u+v = 2n-n_t-3 <2n-4$.
 \item $(u,v) = (2n_t-4,N_t-1)$ for some $t$, if $i\geq 2n-2$, in which case 
 $n-2<u+v = 2n-5 < 2n-4$.
\end{itemize}

This shows that if $E_1^{-i,j} \neq 0$ for some $i\leq j+1$, then $n-2<j<2n-4$.
Thus, we need to analyze the complexes $E_1^{\bullet,j}$ for $j$ in this range.
As explained previously, we only need to consider $i\leq j+2$, that is, we may
assume that $i\leq 2n-3$. It follows that $\widetilde{E}_2^{u,v}$ may only 
be non-zero when $(u,v) = (n-2,0)$ or $(n_t-2,N_t-1)$ for some $1\leq t\leq k$, 
which implies that the spectral sequence \eqref{eq:Leray-E-tilde} degenerates.

Indeed, since the differentials in the spectral sequence are given by mas
\begin{equation}
\label{eq:druv}
\begin{tikzcd}[column sep=22pt]
\tl{d}_r^{u,v} \colon \widetilde{E}_r^{u,v} \ar[r]& \widetilde{E}_r^{u+r,v-r+1}
\end{tikzcd}
\end{equation}
for $r\geq 2$, and since $n>n_t$, the only non-zero such maps could occur
in the following two cases, which we treat separately.
\begin{itemize}[itemsep=2pt]
 \item $(u,v) = (n_t-2,N_t-1)$ for some $1\leq t\leq k$, 
 and $(u+r,v-r+1) = (n-2,0)$.
 This implies that $r=n-n_t$ and $r-1 = N_t-1 = 2(n-n_t)-1$,
 which in turn yields $r=n-n_t=0$, which is impossible.
 \item $(u,v) = (n_t-2,N_t-1)$ and $(u+r,v-r+1) = (n_{t'}-2,N_{t'}-1)$, 
 for some $1\leq t\neq t'\leq k$. This implies that $r=n_{t'} - n_t$ 
 and $r-1 = N_t - N_{t'} = 2(n_{t'}-n_t)$. Since $r>0$, it follows that 
 $0< r = (n_{t'} - n_t) < 2(n_{t'}-n_t) = r-1$, which is again impossible.
\end{itemize}
Since the spectral sequence \eqref{eq:Leray-E-tilde} degenerates,
for $n-2<j<2n-4$ and $i\leq 2n-3$, it follows that 
\begin{equation}
\label{eq:e1-ntn}
E_1^{-i,j} = \bigoplus_{N_t+n_t-3=j} H^{n_t-2}
\Big(\Gr_t,R^{N_t-1}\tl{\gamma}_{t\,*}(\wtl{\mc{K}}_t^{-i} \oo_{\mc{O}_{\wtl{\Gr}_t}} 
\tl{\gamma}_t^*\Sym^q\mc{Q})\Big),
\end{equation}
that is, each chain complex $E_1^{\bullet,j}$ is isomorphic, in the range $\bullet\geq-(2n-3)$,
to a direct sum of complexes of type $F^{\bullet}_{\ol{n}-2,0}$, constructed
as in \eqref{eq:def-Fjd} from an irreducible component $\ol{V}^{\vee} = V_t$
of $\mc{R}(V,K)$ with $\ol{n}=n_t$, $N=N_t$, and $N+\ol{n}-3=j$. Since
$j+2\leq 2n-3$, Proposition~\ref{prop:cohom-Fjd} implies that
$E_2^{-i,j} = \mc{H}^{-i}(E_1^{\bullet,j}) = 0$,
for $i\leq N+\ol{n}-2 = j+1$, thus concluding our proof.\hfill\qed

\section{Generic vanishing for Koszul modules}\label{sect:Koszul_van}
We now discuss Conjecture \ref{conj:Koszul_vanishing}, amounting to the statement $W_{n-4}(V, K)=0$, for a general $(2n-2)$-dimensional subspace $K\subseteq \bigwedge^2 V$. From (\ref{eq:def-WVK}), we have that 
	\[
	W_{q}(V,K)=\coker\Bigl\{K\otimes \Sym^{q}V\stackrel{\delta_2}\longrightarrow \ker(\delta_{1,q+1})\Bigr\},
	\]
where $\delta_2$ is the Koszul differential and $\delta_{1,q+1}\colon V\otimes \Sym^{q+1}V\longrightarrow \Sym^{q+2}V$ is the multiplication map. When $q=n-4$ and $m=\mbox{dim}(K)=2n-4$, the source and target of $\delta_2$ have the same dimension. Therefore the degeneracy locus of  $\delta_2$ defines a virtual divisor on  $\Grass_{2n-2}\bigl(\bigwedge^{2}V\bigr)$.

Recalling the presentation (\ref{eqn:presentation_W}) of the Koszul module $W(V,K)$, we see that $W_q(V,K)=0$ if and only $\widetilde{\delta}_3\colon \bigwedge^3 V\otimes \Sym^{q-1}(V)\rightarrow \bigl(\bigwedge^2 V/K\bigr)\otimes \Sym^q V$ is surjective. We dualize and since $K^{\perp}=\bigl(\bigwedge^2 V/K\bigr)^{\vee}$, we conclude that $W_q(V,K)=0$ if and only if the map 
\begin{equation}\label{eq:D_3}
D_3\colon K^{\perp}\otimes \Sym^q(V)^{\vee} \longrightarrow 
\bwedge^3 V^{\vee}\otimes \Sym^{q-1}(V)^{\vee}
\end{equation}
is injective. With respect to a basis $(e_1, \ldots, e_n)$ of $V^{\vee}$, the 
differential $D_3$ is given by 
\[
D_3\bigl((u\wedge v)\otimes f\bigr)=\sum_{i=1}^n (u\wedge v\wedge e_i)\otimes \partial_{e_i}f,
\]
for elements $u, v\in V^{\vee}$ and $f\in \Sym^{q}(V)^{\vee}$.

\subsection{Generic Koszul vanshing fails for $n=5$}
We now explain that Conjecture \ref{conj:Koszul_vanishing} fails for a $5$-dimensional vector space $V$ for every subspace $K\in \Grass_8\bigl(\bigwedge^2 V\bigr)$.

\begin{proposition}\label{prop:gen5}
If $V$ is $5$-dimensional, then $W_1(V,K)\neq 0$, for every $8$-dimensional subspace $K\subseteq \bigwedge^2 V$.    
\end{proposition}    
\begin{proof}
We fix a general subspace $K\in \Grass_8\bigl(\bigwedge^2 V\bigr)$ and note that $K^{\perp}$ is $2$-dimensional. In particular 
$K^{\perp}\subseteq \P\bigl(\bigwedge^2 V^{\vee}\bigr)$ can be regarded as a \emph{pencil} of $5\times 5$ skew-symmetric forms. Using Kronecker's structure theorem for such pencils, see e.g. \cite{Th}*{Theorem 1}, there exists a basis $(e_1, \ldots, e_5)$ of $V^{\vee}$, such that with respect to this basis $K^{\perp}$ is generated by 
	\[
	P=e_{1}\wedge e_{3}+e_{2}\wedge e_{4}, \quad Q=e_{1}\wedge e_{4}+e_{2}\wedge e_{5}.
	\]
To conclude that the map $D_3$ is not injective it suffices to produce a syzygy between $P$ and $Q$ and we immediately notice that $P\wedge e_1+Q\wedge e_2\in \ker(D_3)$. It follows that $D_3$ is not injective for every $K\in \Gr_8\bigl(\bigwedge^2 V\bigr)$.
\end{proof}

\subsection{Generic Koszul vanshing via Macaulay2} By semicontinuity, in order to prove the vanishing of $W_q(V,K)$ for a general subspace $K\in \Grass_{2n-2}\bigl(\bigwedge^2 V\bigr)$, it suffices to do so for a specific choice of $K$. For small values of $n=\dim(V)$, precisely for $6\leq n\leq 10$, this can be done using a computer algebra software. We checked Theorem~\ref{thm:surprise_9} using Macaulay2 \cite{M2}, and we include the relevant code below.

\begin{verbatim}
S = ZZ/32003[x_1..x_n]; 
kosz1 = koszul(1,vars S); 
kosz2 = koszul(2,vars S); 
m = 2*n-2; 
K = random(S^(binomial(n,2)),S^m); 
cc = chainComplex{kosz1,kosz2 * K}; 
W = HH_1(cc); 
hilbertFunction(n-2,W)
\end{verbatim}

In the above, we have chosen to work with coefficients in a finite field to speed up calculations (this suffices in order to verify generic vanishing). The chain complex \texttt{cc} we construct refers to the complex \eqref{eq:def-WVK} defining the Koszul module $W=W(V,K)$. We note that the default grading convention in Macaulay2 differs from ours by $2$, hence the line \texttt{hilbertFunction(n-2,W)} computes $\dim W_{n-4}(V,K)$, as desired.

Running the code above with $n=9$ yields $\dim W_{n-4}(V,K)=1$ for the random examples that the computer generates, which is highly suggestive of the fact that $W_{n-4}(V,K)\neq 0$ for all subspaces $K\in \Grass_{16}\bigl(\bigwedge^2 V\bigr)$. Unlike in the case $n=5$ discussed in Proposition~\ref{prop:gen5}, we do not have a theoretical explanation for this intriguing fact. This is reminiscent of the Prym-Green Conjecture \cite{CEFS} predicting the vanishing $K_{\frac{g}{2}-3,2}(C, \omega_C\otimes \eta)=0$ for a general Prym curve $[C, \eta]\in \mathcal{R}_g$. 
The genera for which this conjecture is known to fail are precisely $g=8$ (in which case \cite{CFVV} provides several geometric explanations for this surprising fact) and $g=16$, see \cite{CEFS}, which is suggested by a random Macaulay2 calculation, though no geometric explanation for this fact has been found. 

We can ask more generally whether for $n=2^{r}+1$ and $m=2^{r+1}$ we always get $W_{n-4}(V,K)\neq 0$, but already the next case $n=17$ and $m=32$ appears to be out of reach for computer experimentation.

\subsection{Generic Koszul vanishing and $K3$ surfaces}
Polarized $K3$ surfaces of odd genus provide an interesting testing ground for 
Conjecture \ref{conj:Koszul_vanishing}. Precisely, let $(X,H)$ be a polarized 
$K3$ surface with $\mbox{Pic}(X)\cong \Z\cdot H$, where $H^2=4r$, with 
$r\geq 1$, and we fix the Mukai vector $v=(2, H,r)$, therefore $v^2=0$. The 
moduli space $M_X(v)=M_X(2,H,r)$ of $H$-stable rank $2$ vector bundles $E$ 
on $X$ with $h^0(X,E)\geq r+2$ and $\det(E)\cong H$ is again a $K3$ 
surface, Fourier--Mukai dual to the surface $X$. We fix $E\in M_X(v)$ and 
consider the determinant map
\[
d\colon \bwedge^2 H^0(X,E)\longrightarrow H^0(X,H).
\]
Note than $h^0(X,H)=2r+2=2h^0(X,E)-2$ and a general such $E$ is globally generated, while $d$ is surjective. The Koszul module $W(E)\coloneqq W\bigl(H^0(X,E)^{\vee}, \ker(d)^{\perp}\bigr)$ associated in \cite{AFRW} to the vector bundle $E$ can be used to test Conjecture \ref{conj:Koszul_vanishing}. Before stating our next result, we recall that 
$M_E:= \ker\bigl\{H^0(X,E)\otimes \mathcal{O_X}\stackrel{\ev}\rightarrow E\bigr\}$ denotes the corresponding kernel bundle.

\begin{proposition}
\label{prop:Kosz_vanK3}
We fix a polarized $K3$ surface $(X,H)$ with $\mathrm{Pic}(X)\cong \Z\cdot H$ 
and $H^2=4r$.  Then if for a general vector bundle $E\in M_H(2, H,r)$ the following holds
\[
H^1\bigl(X, \Sym^r M_E\bigr)=0,
\]
then Conjecture \ref{conj:Koszul_vanishing} holds in dimension $r+2$.
\end{proposition}
\begin{proof}
Setting $V\coloneqq H^0(X,E)^{\vee}$, note that $H^0(X,H)^{\vee}$ can be regarded as a $(2r+2)$-dimensional subspace of $\bigwedge^2 V$. From \cite{AFRW}*{Theorem 4.3} we have the following identification $W_q(E)^{\vee}\cong H^1\bigl(X, \Sym^r M_E\bigr)$, from which the conclusion follows.    
\end{proof}
    
\section{Fundamental groups of hyperplane arrangements}
\label{sect:arr-gps}
\subsection{Generalities on hyperplane arrangements}
\label{subsec:arrs}
We now explain how the main results of this paper can be applied to describe the 
Chen ranks of the fundamental groups of (complements of) hyperplane arrangements. 
We consider an arrangement $\A$ of hyperplanes in $\C^{m}$ and let 
\[
M({\A})\coloneqq \C^{m}\setminus \bigcup_{H\in \A} H
\] 
be the complement of the arrangement. As shown by Arnold and Brieskorn, 
the cohomology algebra $H^{\hdot}\bigl(M(\A),\R\bigr)$ is embedded as a sub-algebra 
of the de Rham algebra $\Omega^{\hdot}_{\rm{dR}}\bigl(M(\A)\bigr)$. It follows that 
$M(\A)$ is a formal space in the sense of Sullivan \cite{Sullivan}, and thus its 
fundamental group $G(\A)\coloneqq \pi_1(M(\A))$ is $1$-formal.
As usual, we denote by $\P(\A)$ the projectivized arrangement of hyperplanes in 
$\P^{\ell-1}$. Since $M(\A)$ is isomorphic to $M\bigl(\P(\A)\bigr)\times \C^*$, 
we can go back and forth between affine and projective hyperplane arrangements.

In \cite{OS}, Orlik and Solomon gave a description of the
cohomology algebra of $M(\A)$ in terms of the 
\emph{intersection lattice}\/ $L(\A)$ of the arrangement. As usual, a \emph{flat} of $\A$ 
is a non-empty intersection of hyperplanes in $\A$ and $L(\A)$ consists of all the flats 
of $\A$. We fix a linear order on $\A$, and let $E=E(\A)$ be the exterior algebra over $\C$ having the degree $1$ generators $\{e_H\}_{ H\in \A}$. 
We write $e_J\coloneqq e_{j_1}\wedge \cdots \wedge e_{j_p}\in E(\A)$, where 
$J=(j_1, \ldots, j_p)$ is a multiindex such that $j_1<\cdots <j_p$. 

We define a differential $\partial \colon E(\A)\to E(\A)$ of degree $-1$, by setting 
$\partial(1)=0$ and $\partial(e_H)=1$, and extending $\partial$ to
a linear operator on $E(\A)$ using the graded Leibniz rule; that is, 
\[
\partial \bigl(e_{j_1} \wedge \ldots \wedge  e_{j_p}\bigr)=\sum_{k=1}^p (-1)^{k-1} e_{j_1}\wedge\cdots 
\wedge\hat{e}_{j_k}\wedge\cdots \wedge e_{j_p}.
\]
For a subarrangement $\BB=\{H_1,\ldots,H_k\}$ of $\A$, we use the notation 
$e_\BB\coloneqq e_{H_1}\wedge\cdots\wedge e_{H_k}$ and for $X\in L_2(\A)$ we put 
$e_X\coloneqq e_{\A_X}$, where $\A_X\coloneqq \{H\in \A : X\subseteq H\}$.

The \emph{Orlik--Solomon ideal}\/ $I=I(\A)$ is the ideal of $E(\A)$ generated 
by the elements $\bigl\{ \partial \big(e_\BB\big) : 
\codim \big( \bigcap_{H\in \BB} H \big)< \abs{\BB}\bigr\}$. 
As shown in \cite{OS}, the graded algebra $H^{\hdot}\bigl(M(\A), \C\bigr)$ 
is isomorphic to the \emph{Orlik--Solomon algebra} $A(\A)\coloneqq E(\A)/I(\A)$. 

The degree $2$ part of the ideal $I(\A)$ decomposes as 
\begin{equation}
\label{eq:brieskorn}
	I^2(\A)\cong \bigoplus_{X \in L_2(\A)} I^2(\A_X).
\end{equation}
For every $X\in L_2(\A)$, define a derivation $\partial_X\colon E(\A)\to E(\A)$ 
of degree $-1$ by 
setting $\partial_X(e_H)=1$ if $H\supset X$ and $\partial_X(e_H)=0$ otherwise, 
and extending to a linear operator on $E(\A)$ using the graded Leibniz rule.   

The resonance varieties of an arrangement complement were 
introduced by Falk in \cite{Fa97}, and subsequently studied in 
\cites{CS-camb, LY00, FY, PY, Yu09, DS-plms}. We refer to 
the exposition in \cite{Su-toul} for details and further references. The resonance 
variety $\mc{R}(\A)\coloneqq \mc{R}\bigl(G(\A)\bigr)$ of the arrangement is linear, 
projectively disjoint, and isotropic.

\subsection{The Koszul module of a hyperplane arrangement}
\label{subsec:koszul-arr}
Given a hyperplane arrangement $\A$, the \emph{Koszul module}\/ 
$W(\A)=W(V,K)$ of $\A$ is the corresponding Koszul module over 
the polynomial ring $S=\C[x_1,\dots, x_n]$ obtained by setting
\begin{equation}
	\label{eq:VK-arrangement}
	V^{\vee} = E^1(\A)=H^1\bigl(M(\A), \C\bigr) \:\text{ and } K^{\perp} = I^2(\A),
\end{equation}
where $K^\perp=I^2(\A)$ is the kernel of the cup product map 
$\bwedge^2 H^1\bigl(M(\A), \C\bigr)\rightarrow H^2\bigl(M(\A), \C\bigr)$. 
With this notation, we have that $\mc{R}(\A)=\mc{R}(V,K)$.

We fix again an ordering $H_1,\dots ,H_n$ of its hyperplanes.
The  space $V^{\vee}=H^1\bigl(M(\A), \C \bigr)$ has basis 
$e_1,\dots, e_n$ corresponding to the hyperplanes of $\A$, while 
$K^{\perp}=I^2(\A)\subseteq \bwedge^2 V^{\vee}$
is the subspace spanned by elements of the form 
\[
\partial (e_{ijk})=(e_i-e_k)\wedge (e_j-e_k),
\] 
for hyperplanes $H_i,H_j,H_k$ that are not in general position. More precisely,
if $X\in L_2(\A)$ is a rank $2$ flat where $r$ hyperplanes meet, there will
be $\binom{r-1}{2}$ such elements contributing to a basis for $K^{\perp}$; 
hence, $\dim K^{\perp}=\sum_{X\in L_2(\A)} \binom{\abs{X}-1}{2}$.

Let $(v_1,\dots, v_n)$ be the basis of $V=H_1\bigl(M(\A),\C\bigr)$ dual to 
$(e_1, \ldots, e_n)$. Then the subspace $K\subseteq \bwedge^2 V$ is 
spanned by all elements of the form
\[
v_{i_q}\wedge \bigg(\sum_{s=1}^r v_{j_s}\bigg)
\]
for all flats $X=H_{i_1}\cap \cdots \cap H_{i_r}\in L_2(\A)$ 
and all indices $1\le q <r$. 
In particular, $m\coloneqq \dim K=\sum_{X\in L_2(\A)} (\abs{X} -1)$. 
For instance, a double point $(ij)$ in $\P(\A)$ contributes a 
basis element $v_i \wedge v_j$ to $K$, whereas a triple point $(ijk)$ 
contributes two basis elements, $v_i\wedge (v_j+v_k)$ and $v_j\wedge (v_i+v_k)$.

Finally, we note the following simple fact. For a subarrangement 
$\BB\subseteq \A$, consider the linear subspace
$V_\BB^\vee \coloneqq \spn \{e_H\}_{H\in \BB} \subseteq V^{\vee}$. 
We then have the decomposition:
\begin{equation}
    \label{eqn:decomposition-subarr-wedge}
    \bwedge^2 V^\vee=\bwedge^2V_\BB^\vee\oplus\left(V_\BB^\vee
    \wedge V_{\A\setminus\BB}^\vee\right)\oplus\bwedge^2V_{\A\setminus\BB}^\vee .
\end{equation}
According to the decomposition \eqref{eqn:decomposition-subarr-wedge}, any 
$\omega\in K^\perp$ can be written as 
\begin{equation}
    \label{eqn:decomposition-subarr-omega}
    \omega = \omega_\BB+\omega_M+\omega_{\A\setminus\BB}.  
\end{equation}
We will repeatedly use this decomposition in relation with separability in the sequel.

\subsection{Multinets and resonance}
\label{subsec:nets}
To describe the resonance varieties $\mc{R}(\A)$ in combinatorial terms, we recall 
an important notion due to Falk and Yuzvinsky \cite{FY}.

\begin{definition}
\label{def:multinet}
A {\em $k$-multinet} on an arrangement $\A$ consists
of a partition of $\A$ into $k$ subsets $\A_1,\ldots,\A_k$,
together with an assignment of multiplicities,
$m\colon \A\to \Z_{>0}$, and a subset $\XX\subseteq L_2(\A)$,
called the \emph{base locus} of $\NN$, such that the following 
are satisfied:
\begin{enumerate}[itemsep=2pt]
	\item \label{m1}
	There is an integer $d$ such that $\sum_{H\in\A_i} m_H=d$,
	for all $i=1, \ldots, k$.
	\item \label{m2}
	For any two hyperplanes $H$ and $H'$ in different classes,
	$H\cap H'\in \XX$.
	\item \label{m3}
	For each $X\in\XX$, the sum
	$n_X\coloneqq \sum_{H\in\A_i\cap \A_X} m_H$ is independent of $i$.
	\item \label{m4}
	For each $i=1, \ldots, k$, the space
	$\big(\bigcup_{H\in \A_i} H\big) \setminus \XX$ is connected.
\end{enumerate}
\end{definition}
We refer to such an object as a $(k,d)$-multinet. A flat $X\in L_2(\A)$ is 
either contained in some block $\A_i$,
or it meets each block, in which case it belongs to $\XX$, 
see \cite{LY00}. The base locus $\XX$ is determined by the partition 
$(\A_1,\dots, \A_k)$; indeed, for each $i\ne j$, we have that 
$\XX = \bigl\{ H \cap H': H\in \A_i,\, H'\in \A_j \bigr\}$.

Work in \cites{PY, Yu09} shows that if $\NN$ is a $k$-multinet with 
$\abs{\mathcal{X}}>1$, then $k=3$ or $4$; moreover, if at least one 
multiplicity $m_H$ is not equal to $1$, then $k=3$. Although several infinite 
families of multinets with $k=3$ are known, only one multinet 
with $k=4$ is known to exist, namely, the $(4,3)$-net on the 
Hessian arrangement of $12$ lines in $\P^2$. 

By a Lefschetz-type argument, since $\mc{R}(\A)$ is determined by $L_{\le 2}(\A)$, 
we may take a projective $2$-dimensional slice of $\A$ and assume without loss 
of generality that $\P(\A)$ is a line arrangement in $\P^2$, so that $L_2(\A)$ consists 
of the multiple points on $\P(\A)$. With this in mind, following \cite{FY}*{Theorem 3.11}, 
a $k$-multinet $\NN$ on $\A$ determines an orbifold fibration
\[
f\colon M(\A) \longrightarrow \Sigma, \qquad f=[C_1:C_2],
\]
where $\Sigma\coloneqq \P^1\setminus \{\text{$k$ points}\}$ and 
$C_i\coloneqq \prod_{H\in \A_i} H^{m_H}\in \bigl|\mathcal{O}_{\P^1}(d)\bigr|$, 
for $i=1, \ldots, k$. It follows from Definition \ref{def:multinet} that 
$\dim \spn \{C_1, \ldots, C_k\} =2$, and the $k$ points in 
$\P^1$ that $\im(f)$ avoids correspond precisely to the values of 
$[C_1], \ldots, [C_k]$ in this pencil.

The induced map in cohomology, $f^{*} \colon H^{\hdot}(\Sigma,\C) \to H^{\hdot}(M(\A),\C)$, 
sends a loop $c_i$ about the $i$-th puncture of $\P^1$ to the element
\[
u_{i}=\sum_{H\in \A_{i}} m_H\cdot e_H\in H^1\bigl(M(\A), \C\bigr)=E^1(\A).
\]
Consequently, the homomorphism $f^{*} \colon H^1(\Sigma,\C) \to H^1(M(\A),\C)$ 
is injective, and thus sends $\mc{R}(\Sigma)$ to $\mc{R}(M(\A))$. Upon identifying 
$\mc{R}(\Sigma)$ with $H^1(\Sigma,\C)=\C^{k-1}$,
one sees that $P_{\NN}\coloneqq f^*(H^1(\Sigma,\C))$ is the $(k-1)$-dimensional 
linear subspace of $V^{\vee}=H^1(M(\A),\C)$ spanned by 
$u_2-u_1,\dots , u_k-u_1$. Since $H^1(\Sigma,\C)$ is isotropic, $P_{\NN}$ is also 
isotropic. As shown in \cite{FY}*{Theorem 2.4}, this subspace is an irreducible component of 
$\mc{R}(\A)$.

\subsection{Multinets on subarrangements}
\label{subsec:multi-subs}
Suppose there is a sub-arrangement
$\BB\subseteq \A$ supporting a multinet $\NN$. In this case,
the inclusion $M(\A) \inj M(\BB)$ induces an
injection $H^1(M(\BB),\C) \inj H^1(M(\A),\C)$,
which restricts to an embedding $\mc{R}(\BB) \inj \mc{R}(\A)$. 
Thus, the resonance component $P_{\NN} \subseteq \mc{R}(\BB)$ 
may be viewed as a linear subspace $P_{\NN} \subseteq \mc{R}(\A)$.
It is shown in \cite{FY} that $P_{\NN}$ is an irreducible component 
of $\mc{R}(\A)$. Furthermore, all (positive-dimensional) 
irreducible components of $\mc{R}(\A)$ are of the form $P_{\NN}$, for some 
multinet $\NN$ on a sub-arrangement $\BB\subseteq \A$, maximal with the 
property that it supports a multinet. The components corresponding 
to multinets of $\A$ are called {\em essential components}. 

\begin{lemma}
\label{lem:proportional}
Let $\NN$ be a $k$-multinet on a sub-arrangement $\BB=\BB_1 \sqcup \cdots \sqcup \BB_k$ 
with corresponding resonance component $P_{\NN}$. If $a=\sum_{H\in \A} a_H\cdot  e_H \in P_{\NN}$, then:
\begin{enumerate}
    \item \label{pp1}
    The vector $a$ is supported on $\BB$, that is, $a_H = 0$ for $H \notin \BB$.
    \item  \label{pp2}
    For any $H \in \BB_j$, the ratio $\frac{a_H}{m_H}$ is independent of $H$.
\end{enumerate}
\end{lemma}
\begin{proof}
Let $a \in P_{\NN}$. Then $a = \sum_{i=2}^k \lambda_{i} (u_i - u_1)$, where 
$u_{i} = \sum_{H\in \BB_{i}} m_H\cdot e_H$. Since each $u_i$ is supported on $\BB$, 
also $a$ is supported on $\BB$, which proves \eqref{pp1}. 
For claim \eqref{pp2}, we rewrite
\[
a = -(\lambda_2+\cdots+\lambda_k)u_1  + \sum_{i=2}^k \lambda_{i} u_i 
\equalscolon \sum_{j=1}^k \mu_j u_j,
\]
where $\mu_1=-(\lambda_2+\cdots+\lambda_k)$ and $\mu_j = \lambda_j$ for $j \ge 2$.
The coefficient $a_H$ for any $H \in \BB_j$ is precisely $a_H = \mu_j \cdot m_H$, 
hence the claimed proportionality holds.
\end{proof}

We set $\partial_H(a)\coloneqq a_H$ and extend $\partial_H\colon E(\A)\rightarrow E(\A)$ 
to an operator of degree $-1$. In terms of the basis $(v_H)_{H\in \A}$ of $V$, dual to the 
basis $(e_H)_{H\in \A}\subset V^\vee$, we have
\[
\partial_{|V^\vee} = v_\A \coloneqq \sum_{H\in\A}v_H, \
\
\partial_{X|V^\vee} = v_X \coloneqq \sum_{H\in\A_X}v_H,
\ \mbox{ and }\ 
\partial_{H|V^\vee} = v_H.
\]

We now describe the equations of $P_{\NN}$ in the form given by \cite{CSc-adv}*{(5.1)}. Recall that  $\XX$ denotes the base locus of  $\NN$.
\begin{proposition}
\label{prop:res-comp-equations}
The component $P_{\NN}$ of the resonance $\mc{R}(\A)$ is given by
\[
P_{\NN} = \Bigl\{ a \in V^{\vee} : 
	\partial(a) =0, \ \text{$\partial_X(a)=0$ for $X\in \XX$}, \ 
	\text{$\partial_H(a)=0$ for $H\notin \BB$}
	\big.\Bigr\}.
\]
\end{proposition}

\begin{proof}
Let $Z$ denote the set on the right-hand side. Since $Z$ is defined by linear equations, 
it is a linear subspace of $V^{\vee}$. We start by showing that $P_{\NN} \subseteq Z$; 
for that, we must verify that each basis element $u_i - u_1$ for $P_{\NN}$ satisfies the three 
defining conditions of $Z$.
\begin{itemize}
    \item 
    By part \eqref{m1} of Definition \ref{def:multinet}, $\partial(u_i) = 
    \sum_{H\in \BB_i} m_H = d$, so $\partial(u_i - u_1) = 0$.
    \item 
    By part \eqref{m3} of Definition \ref{def:multinet}, $\partial_X(u_i) = 
    \sum_{H\in \BB_i \cap \A_X} m_H = n_X$, so $\partial_X(u_i - u_1) = 0$.
    \item 
    Since by definition $u_i$ is supported on $\BB$, we have $\partial_H(u_i - u_1) = 0$ 
    for $H \notin \BB$.
\end{itemize}
Thus, $P_{\NN}\subseteq Z$. Since $\dim P_{\NN}=k-1$, it remains to verify 
that $\dim Z = k-1$. Since the equations $\partial_H(a) = 0$, for $H \notin \BB$, 
restrict the support of $a$ to $\BB$, we may identify 
$Z$ with the kernel of the linear map from $\C^{\abs{\BB}}$ to $\C^{1 + \abs{\XX}}$ given by the 
global sum and the $\abs{\XX}$ local sums over lines through each $X \in \XX$. Equivalently, 
$W = \ker(J) \cap \ker(E)$, where $J$ is the $\abs{\XX} \times \abs{\BB}$ incidence matrix with 
$J_{X,H} = 1$ if $H \supset X$ and $0$ otherwise, and $E$ is the all-ones 
$\abs{\BB} \times \abs{\BB}$ matrix (so $\ker(E)$ is the hyperplane of vectors with 
coordinate sum zero), see \cite{LY00}.
As shown in \cite{FY}*{Theorem 2.5}, this space coincides with $\ker(Q) \cap \ker(E)$, 
where $Q = J^T J - E$. The multinet structure ensures 
that $Q$ decomposes as a direct sum of $k$ indecomposable blocks, $Q_1 \oplus \cdots \oplus Q_k$ 
(one per class $\BB_i$). Moreover, $\dim(\ker(Q_i))=1$ for all $i$, so $\dim(\ker(Q))=k$.  
Intersecting with $\ker(E)$ yields dimension $k-1$, since the kernel vectors 
of the $Q_i$ can be scaled to have equal coordinate sums $d$, and the intersection consists 
of their differences. Thus, $\dim Z = k-1$, and this completes the proof.
\end{proof}

\subsection{Linear equations and separability of essential components}
\label{subsec:eq-sep}
Let $P_{\NN}$ be the component of the resonance $\mc{R}(\A)$ 
arising from a multinet $\NN$ supported on a sub-arrangement $\BB = \BB_1 \sqcup \cdots \sqcup \BB_k$ of
$\A$ as above.  Recall that $P_{\NN}$ is isotropic and hence it is a separable component of $\mc{R}(\A)$ 
if the equality $\bigl(P_{\NN}\wedge V^{\vee}\bigr)\cap K^{\perp}=\bigwedge^2 P_{\NN}$ holds. Thus, it is 
of interest to explicitly determine the equations of these subspaces of $\bigwedge^2 V^{\vee}$ appearing 
on both sides of this (hoped for) equality.

\begin{proposition}
\label{prop:P-wedge-P-equations}
 The subspace $\bwedge^2P_{\NN}\subseteq \bigwedge^2 V^{\vee}$ is described as:
\[
\bwedge^2P_{\NN} = 
\Bigl\{
  \omega \in \bwedge^{2} V^{\vee} :
  \partial(\omega)=0,\;
  \partial_X(\omega)=0 \mbox{ for }  X\in \XX, \;
  \partial_H(\omega)=0\ \mbox{ for } H\in \A \setminus \BB
\Bigr\}.
\]
\end{proposition}

\begin{proof}
Combining Propositions \ref{prop:res-comp-equations} and  \ref{lem:wedge2-equations}, the space 
$\bwedge^2P_{\NN}\subseteq \bwedge^2 V^{\vee}$ is given by
\[
\biggl\{ \omega \in \bigwedge^2 V^{\vee}  : \:
\begin{array}
[c]{l}%
\! (v_\A\wedge v)(\omega)=0,\ \: 
  (v_X\wedge v)(\omega)=0,\ \:
  (v_H\wedge v)(\omega)=0\\[2.5pt]
\hspace*{0.9in} \mbox{ for  } v\in V,\   X\in \XX, \mbox{ and }  H\in \A\setminus\BB 
\end{array}
\biggr\} .
\]

These equations can be wrapped up into the simpler ones stated above, using the identities 
$v_H\wedge v=v\circ \partial_H$, $\partial = \sum_{H\in\A}\partial_H$, and 
$\partial_X=\sum_{H\in\A_X}\partial_H$, which are immediate.
\end{proof}

The next result is established by Cohen and Schenck in \cite{CSc-adv}*{Theorem 5.1}. For the sake of completeness, we include their proof using our setup.

\begin{theorem} \label{lem:delX-omega}
If $\omega \in \bigl(P_{\NN}\wedge V^{\vee}\bigr) \cap K^{\perp}$, then 
$\partial(\omega) = 0$  and $\partial_X(\omega)=0$  for all $X\in \XX$. In particular, 
any essential component $P_{\NN}$ of $\mathcal{R}(\A)$ is separable.
\end{theorem}

\begin{proof} 
Recall that $P_{\NN}=\spn\{ u_2-u_1,\ldots,u_k-u_1\}$, where $u_i = \sum_{H\in\BB_i} 
m_{H}\cdot e_H$. Since $\omega \in P_{\NN} \wedge V^{\vee}$, 
we may write 
$\omega =(u_2-u_1) \wedge g_2+\cdots +(u_k-u_1)\wedge g_k$
for some $g_2,\ldots,g_k \in V^{\vee}$. 
	We compute 	$\partial(\omega) = - \sum_{i=2}^k   \partial(g_i) (u_i-u_1)$.
Since $\omega\in K^{\perp}=I^2(\A)$, we have
$\partial(\omega)=0$. The vectors $u_2-u_1,\ldots,u_k-u_1$ 
are linearly independent, we must also have 
	\begin{equation}
		\label{eq:partial-gi-0}
		\partial(g_i)=0 \:\text{ for $i=2,\dots, k$}.
	\end{equation}

	Now let $X\in \XX$. 
	There is a subset $\{H_{j_1}, \dots , H_{j_k}\}\subseteq \A_X$ 
	with $H_{j_i}\in \BB_{i}$ for all $i=1,\dots, k$. 
	Write $u_i - u_1 = y_{i} + \bar{y}_{i}$, with $y_{i}$ supported on $\A_X$
	and $\bar{y}_{i}$ supported on $\A\setminus \A_X$. Then 
	\begin{equation}
		\label{eq:y-x-i}
		y_{i}= m_{H_{j_i}}\cdot e_{H_{j_i}} - m_{H_{j_1}}\cdot  e_{H_{j_1}}
		+ \sum_{H\in \BB_i\cap \A_X \setminus \{H_{j_i}\}} m_H\cdot e_H 
		- \sum_{H\in \BB_1\cap A_X \setminus \{H_{j_1}\}} m_H\cdot e_H .
	\end{equation}
Since $m_{H_{j_i}}\ne 0$ and $y_{i'}$ for $i'\in \{2, \ldots, k\}\setminus \{i\}$ 
does not contain any vector from the set $\{e_H\}_{H\in \BB_i}$, it follows that 
$y_{2},\dots, y_{k}$ are linearly independent. Via condition \eqref{m3} of Definition \ref{def:multinet}, 
we have $\partial_X(u_i-u_1)=0$. Since by definition $\partial_X (\bar{y}_{i}) = 0$, we infer that 
	\begin{equation}
		\label{eq:ypartial-xyi}
		\partial_X (y_{i}) = 0 .
	\end{equation}
	
Write $g_i=z_{i}+\bar{z}_{i}$, where $\supp(z_{i})\subseteq \A_X$ 
	and $\supp(\bar{z}_{i}) \subseteq \A\setminus \A_X$. Then 
	\[
	\omega = \sum_{i=2}^k (y_{i} + \bar{y}_{i}) \wedge (z_{i} + \bar{z}_{i}).
	\]
	Set $\omega_X\coloneqq \sum_{i=2}^k y_{i}\wedge z_{i}$. By \eqref{eq:brieskorn}, 
	we have $\omega_X \in I^2(\A_X)$; hence,
	\begin{equation}
		\label{eq:xi-x}
		0=\partial (\omega_X) = \sum_{i=2}^k \Bigl(\partial_X (y_{i})\  z_{i} - y_{i}\  \partial_X (z_{i})\Bigr) . 
	\end{equation}
	In view of \eqref{eq:ypartial-xyi}, 
	we get $\sum_{i=2}^k  y_{i} \ \partial_X (z_{i} ) = 0$. 
	Since  $y_{2},\dots, y_{k}$ are linearly independent, we 
	conclude that $\partial_X(z_{i})=0$ for $i=2,\dots, k$. Since by definition
	$\partial_X (\bar{z}_{i}) = 0$, we obtain  
$\partial_X(g_i)=0$, for $i=2, \ldots, k$. 
Finally, since $u_i-u_1\in P_{\NN}$, we also have $\partial_X(u_i -u_1)=0$, for $i=2, \ldots, k$, 
therefore $\partial_X(\omega)=0$, 
and this completes the proof of the first part.

For the second part, when $P_{\NN}$ is an  essential component, then $\BB=\A$ and 
comparing the equations in  the first part of Theorem \ref{lem:delX-omega} with 
those in Proposition \ref{prop:P-wedge-P-equations}, the equality 
$\bigl(P_{\NN}\wedge V^{\vee}\bigr)\cap K^{\perp}=\bigwedge^2 P_{\NN}$ becomes immediate.
\end{proof}

\begin{remark}\label{rem:cohen-schenk}
For \emph{non-essential} components of $\RR(\A)$, when $\A\setminus \BB\neq \emptyset$, 
comparing Theorem \ref{lem:delX-omega} and Proposition \ref{prop:P-wedge-P-equations}, 
we point out the extra equations $\partial_H(\omega)=0$, where $H\in \A\setminus\BB$, 
that ensure the separability of a component $P_{\NN}$. This fact is overlooked in \cite{CSc-adv}. 
The claim at the beginning of the attempted proof in the  revised version \cite{CSc-adv}*{Theorem 5.1}, 
that non-reducedness of a component $P_{\NN}$ of $\RR(\A)$ supported on a subarrangement $\BB$ of $\A$ 
implies the non-reducedness of the essential component $P_{\NN}$ in $\RR(\BB)$ is incorrect, 
see  Remark \ref{rem:reduceness2}. 

Instead, we will verify the extra equations appearing in Proposition \ref{prop:res-comp-equations} 
in several important cases in what follows.
\end{remark}

\subsection{Local components are separable}
\label{subsec:local comps}
The simplest components of  $\mathcal{R}(\A)$ are canonically associated to 
rank $2$ flats of $\A$ lying at the intersection of at least $3$ hyperplanes.

\begin{definition}
\label{ex:local comp}
Let $X\in L_2(\A)$ be a $2$-flat which is the intersection of $k\ge 3$ hyperplanes. 
The \emph{local component} of $\RR(\A)$ 
is the component corresponding to the $k$-net on the sub-arrangement $\A_{X}$, obtained 
by assigning to each hyperplane the multiplicity $1$, placing one hyperplane in each 
class, and setting $\XX=\{X\}$.
\end{definition}

Concretely, we have the following description of a local component of the resonance: 
\begin{equation}\label{eq:res-loc}
P_{X} = \Bigl\{ a \in V^{\vee} : 
	\partial_X(a)=0 \text{ and }
\partial_H(a)=0 \text { for } H\not\supset X
	\big.\Bigr\}.
\end{equation}

\begin{proposition}
\label{prop:res-local}
For every $2$-flat $X\in L_2(\A)$ lying at the intersection of at least 
$3$ hyperplanes, the corresponding local component 
$P_{X}$ of $\RR(\A)$ is separable.
\end{proposition}

\begin{proof}
Let $H_1,\ldots,H_k$ be the set of hyperplanes that 
contain $X$ and let $e_1,\ldots,e_k$ be the corresponding basis elements in $V^\vee$. 
Denote by $\bigl\{X, X_1, \ldots, X_r, X_{r+1}, \ldots, X_m\bigr\}$ the set of 
$2$-flats lying at the intersection of at least $3$ hyperplanes in $\A$, where 
$\{X,X_1,\ldots, X_r\}$ is the subset of the 2-flats contained in one of the 
hyperplanes $H_1, \ldots, H_k$. Note that for any $\alpha\in\{1,\ldots,r\}$
there exists a unique hyperplane $H_i=H_{i(\alpha)}\in \A_X$ containing 
$X_\alpha$. 
With this notation, $\ol{V}^\vee\coloneqq P_X=\spn \{e_1-e_k,\ldots,e_{k-1}-e_k\}$. 
As before,  $V^\vee_{\A_X}=\spn \{e_1, \ldots, e_k\}$.

From \eqref{eq:brieskorn} an element 
$\omega\in K^\perp$ decomposes depending on the support of these $2$-flats as 
 \begin{equation}
 \label{eqn:omega123}
 \omega = \omega_{X}+\sum_{\alpha=1}^r\omega_{X_\alpha}+\sum_{\beta=r+1}^{m} \omega_{X_\beta},
 \end{equation}
 where $\omega_X=\sum_{\substack{X\subset H_i\cap H_j\cap H_\ell}} c_{ij\ell}\cdot \partial(e_{ij\ell}) \in \bigwedge^2\ol{V}^\vee$, 
 and for $\alpha=1, \ldots, r$, we have
 \begin{equation}
 \label{eq:xalpha}
\omega_{X_\alpha}=
 \sum_{\substack{H_{i(\alpha)}\cap H_j\cap H_\ell= X_\alpha\\H_j\cup H_\ell\not\supset X}}
 c_{i(\alpha) j\ell}\cdot \partial(e_{i(\alpha) j\ell})
 + \sum_{\substack{H_i\cap H_j\cap H_\ell=X_\alpha\\H_i\cup H_j\cup H_\ell\not\supset X}}
 c_{ij\ell}\cdot \partial(e_{ij\ell})
 \end{equation}
 while for $\beta\ge r+1$, we have
$\omega_{X_\beta}=\sum_{H_i\cap H_j \cap H_\ell=X_\beta} c_{ij\ell}\cdot \partial(e_{ij\ell})$.
 
According to \eqref{eqn:decomposition-subarr-omega}, we have another decomposition
\begin{equation}
\label{eqn:local-omega-decomp}  
\omega = \omega_{\A_X}+\omega_M+\omega_{\A\setminus\A_X}, 
\end{equation}
with $\omega_{\A_X}\in \bigwedge^2V_{\A_X}^\vee$, $\omega_M\in V_{\A_X}^\vee\wedge V_{\A\setminus\A_X}^\vee$ 
and $\omega_{\A\setminus\A_X}\in\bigwedge^2V_{\A\setminus\A_X}^\vee$. By the definition of $\A_X$, 
we must necessarily have $\omega_X = \omega_{\A_X}$.

Assume now $\omega\in  (\ol{V}\wedge V^\vee)\cap K^\perp$. In this case, since $\ol{V}\wedge V^\vee\subset V_{\A_X}^\vee\wedge V^\vee$, we obtain $\omega_{\A\setminus\A_X}=0$. 
Comparing the decompositions  \eqref{eqn:omega123} and \eqref{eqn:local-omega-decomp}, we find that 
\begin{equation}
  \label{eqn:omegaM-formula1}  
\omega_M = \sum_{\alpha=1}^m\left( 
 \sum_{\substack{H_{i(\alpha)} \cap H_j\cap H_\ell=X_{\alpha}\\  
 H_j\cup H_\ell\not\supset X}} c_{i(\alpha) j\ell}\cdot \bigl(e_{i(\alpha)}\wedge e_j - 
 e_{i(\alpha)}\wedge e_\ell\bigr)\right).
\end{equation}

Since $\omega_X\in I^2(\A_X)=K^\perp$, also $\omega_M\in K^\perp$, and hence 
$\omega_M$ is a linear combination  
\begin{equation}
  \label{eqn:omegaM-formula2} 
\omega_M = \sum a_{pqs}\cdot \partial(e_{pqs}).
\end{equation}

We assume that $\omega_M \ne 0$ and compare \eqref{eqn:omegaM-formula1} and \eqref{eqn:omegaM-formula2}, 
with a focus on the common contributions to basis elements. We remark that in \eqref{eqn:omegaM-formula2}, 
a non-trivial contribution can arise only from elements $a_{i(\alpha)j\ell}\cdot \partial(e_{i(\alpha)j\ell})$ 
with $a_{i(\alpha)j\ell}\ne 0$ and $X_\alpha = H_{i(\alpha)}\cap H_j\cap H_\ell$. In particular, such 
an element will also produce a non-zero term $a_{i(\alpha)j\ell}\cdot e_j\wedge e_\ell$ which has no 
canceling partner in \eqref{eqn:omegaM-formula2}, since $X_\alpha = H_j\cap H_\ell$. This is in 
contradiction with \eqref{eqn:omegaM-formula1}.
\end{proof}

\subsection{Arrangements with double and triple points}
\label{subsec:double-triple}

We now identify an important class of  arrangements for which all the 
resonance components are separable.
\begin{theorem}
\label{thm:res-arr-separable}
Let $P_{\NN}$ be a component of $\mathcal{R}(\A)$ corresponding to a 
multinet $\NN$ on a sub-arrangement $\BB \subsetneq \A$. Assume that for any   
$Y\in L_2(\A)$ that is the intersection of two hyperplanes in $\BB$ we have $\abs{\A_Y}\le 3$. Then 
$P_{\NN}$ is separable. 
\end{theorem}

\begin{proof}
Assume the sub-arrangement $\BB \subsetneq \A$ has parts $\BB_1, \dots, \BB_k$, 
multiplicities $m_H \geq 1$ for $H \in \BB$, and base locus $\XX$. Because of our assumptions, for each $Y\in L_2(\A)$ with $|\A_Y|\ge 3$, we distinguish the following possibilities:
\begin{enumerate}[itemsep=2pt]
    \item $\A_Y\subseteq \BB$ and $|\A_Y|=3$.
    \item $|\A_Y\cap \BB| = 2$ and $|\A_Y|=3$.
    \item $\A_Y\cap \BB$  consists of 
    one element, call it $H_{a(Y)}$.
    \item $\A_Y\subseteq \A\setminus\BB$.
\end{enumerate}

Accordingly, any $\omega \in K^{\perp}$ may be expressed as 
$\omega = S_1+S_2+S_3+S_4$,  where
\begin{align*}
S_1 = \sum_{\substack{Y\in L_2(\A)\\ \A_Y\subseteq \BB}} \lambda_Y \cdot \partial(e_Y), \ \ \ S_2 = \sum_{\substack{Y\in L_2(\A)\\ |\A_Y\cap \BB| = 2}} \lambda_Y \cdot \partial(e_Y),\\
S_3 = \sum_{\substack{Y \in L_2(\A)\\ \A_Y\cap \BB = \left\{H_{a(Y)}\right\}}}  \left(\sum_{\substack{H_b<H_c\\ H_{a(Y)}\cap H_b\cap H_c = Y}}\lambda_{a(Y)bc} \cdot \partial(e_{a(Y)bc})\right),\\ 
S_4 = \sum_{\substack{H_a<H_b<H_c\\ \{H_a, H_b, H_c\}\subseteq\A_Y\setminus \BB}}\lambda_{abc} \cdot \partial(e_{abc}).
\end{align*}
In $S_3$ we regrouped a part of triples that appear in case (3), and in $S_4$ the remaining part of triples from case (3) plus the triples covered by case (4). Note that for a flat $Y\in L_2(\A)$ with $\A_Y=\{H_a,H_b,H_c\}$ and $|\A_Y\cap\BB|\ge 2$, where $H_a<H_b<H_c$,
each of the bivectors $e_a\wedge e_b$, $-e_a\wedge e_c$, and $e_b\wedge e_c$ 
that appear in $S_1$ and $S_2$ appears only once, with coefficient~$\lambda_Y$. 

We fix an element $\omega \in \bigl(P_{\NN}\wedge V^{\vee}\bigr)\cap K^{\perp}$. 
Using the decomposition  \eqref{eqn:decomposition-subarr-wedge},  we may write $\omega = \omega_{\BB} + \omega_{M}$, 
where $\omega_{\BB} \in \bwedge^2 V_{\BB}^{\vee}$ and 
$\omega_{M} \in V_{\BB}^{\vee} \wedge V_{\A\setminus\BB}^{\vee}$. Therefore, from the discussion above, $\omega$ can be written as
\begin{equation}
    \omega = S_1 + S_2 + \sum_{\substack{Y \in L_2(\A)\\ \A_Y\cap \BB = \left\{H_{a(Y)}\right\}}} \left(\sum_{\substack{H_b<H_c\\ H_{a(Y)}\cap H_b\cap H_c = Y}}   \lambda_{a(Y)bc} \cdot (e_{a(Y)}\wedge e_b-e_{a(Y)}\wedge e_c)\right).
\end{equation}

The component $\omega_{M}$ can be written as:
\begin{equation}
\label{eq:omega-sum}
\omega_{M} = \sum_{H_c \in \A \setminus \BB} v_c \wedge e_c, 
\quad \text{with } v_c = \sum_{H_a \in \BB} \omega_{ac} e_a \in P_{\NN} \subseteq V_{\BB}^{\vee}.
\end{equation}

We observe that $S_2=0$, and hence $\omega_\BB = S_1$. 
Indeed, let $Y = H_a\cap H_b\cap H_c\in L_2(\A)$, 
with $H_a, H_b \in \BB$ and $H_c \in \A \setminus \BB$. In this case we have $\A_Y=\{H_a, H_b, H_c\}$.
The contribution of $\lambda_Y\cdot \partial(e_Y)$ to $\omega_{M}$ equals $\lambda_Y(-e_a \wedge e_c + e_b \wedge e_c)$. We get
$\omega_{ac}=-\lambda_Y$ and $\omega_{bc} =
\lambda_Y$, in particular $\omega_{ac} + \omega_{bc} = 0$. Note also that $H_a, H_b$ must be in the same block $\BB_i$, otherwise $H_a \cap H_b = Y \in \XX$ and since $k\geq 3$, we obtain $\abs{\A_Y} \ge 4$, a contradiction.
Since $v_c \in P$, by Lemma \ref{lem:proportional} we have, 
$\frac{\omega_{ac}}{m_{H_a}}=\frac{\omega_{bc}}{m_{H_b}}$. Since $m_{H} \geq 1$ for all $H\in \BB$, this forces $\omega_{ac}=\omega_{bc}=0$, therefore $\lambda_Y=0$.

Since $S_1\in K^\perp$, and $\omega \in K^\perp$ it follows that 
\[
 \sum_{\substack{Y \in L_2(\A)\\ \A_Y\cap \BB = \left\{H_{a(Y)}\right\}}} \left(\sum_{\substack{H_b<H_c\\ H_{a(Y)}\cap H_b\cap H_c = Y}}\lambda_{a(Y)bc} \cdot \big(e_{a(Y)}\wedge e_b-e_{a(Y)}\wedge e_c\big)\right)\in K^\perp.
\]

Arguing as in the proof of Proposition \ref{prop:res-local}, we infer that $\lambda_{a(Y)bc} = 0$ for all $Y\in L_2(\A)$ for which $\A_Y\cap \BB = \bigl\{H_{a(Y)}\bigr\}$ and $H_{a(Y)}\cap H_b\cap H_c = Y$. Hence $\omega_M = 0$ and therefore $\omega = \omega_\BB\in K_
\BB^\perp=I^2(\BB)$. By Theorem \ref{lem:delX-omega}, we know that $P_\NN$ is separable in $V_\BB^\vee$, that is,  $K^\perp_\BB\cap \bigl(P_{\NN}\wedge V^\vee_\BB\bigr) = \bigwedge^2P_{\NN}$. It follows that $\omega = \omega_\BB\in \bigwedge^2 P_{\NN}$, which completes the proof.
\end{proof}

We have a stronger version of Theorem \ref{thm:res-arr-separable} showing that the resonance variety $\RR(B_n)$ of the Coxeter arrangement of type $B_n$ ($n\ge 2$) is separable for every $n$. The proof is slightly more involved and this result will be contained in a forthcoming paper.

An immediate consequence is the following.
\begin{corollary}
\label{cor:res-arr-separable}
If $\A$ has no $2$-flats of size greater than $3$, then 
$\mc{R}(\A)$ is separable. 
\end{corollary}

As an application of the Main Theorem and of the above results, we prove the following 
effective version of the Chen ranks conjecture for a large class of arrangements that was 
announced in the Introduction.

\noindent \emph{Proof of Theorem \ref{thm:chen-arrs}.}
Since $G(\A)$ is a $1$-formal group, it follows  that $\dim W_q(\A) = \theta_{q+2}(G)$, 
for all $q\geq 0$. From the discussion in \S\ref{subsec:nets},  the resonance variety 
$\mc{R}(\A)$ is linear and isotropic. By Proposition \ref{prop:res-local} and 
Theorem \ref{lem:delX-omega} in the first case, and 
Corollary \ref{cor:res-arr-separable} in the second, $\mc{R}(\A)$ is also separable. 
The formula for the Chen ranks  follows from  Theorem \ref{thm:main}.
\hfill $\Box$

\subsection{Chen ranks of graphic arrangements}
\label{subsec:graphic}
A prominent class of arrangements satisfying the hypotheses of Theorem \ref{thm:res-arr-separable} is provided by \emph{graphic arrangements}.

\begin{definition}\label{def:graphic_arr}
Let $\Gamma \coloneqq (\mathsf{V},\mathsf{E})$ be a finite simple graph on vertex set 
$\mathsf{V}=\{1,\dots, m\}$ and edge set $\mathsf{E}$. 
The corresponding graphic arrangement, $\A_{\Gamma}$, consists of all the hyperplanes 
$H_{ij}=\{z_i-z_j=0\}$ in $\C^{m}$, for which $\{i,j\}\in \mathsf{E}$.
\end{definition}

The $2$-flats of $\A_{\Gamma}$  are of two types: 
either $H_{ij}\cap H_{k\ell}$, where $\{i,j\}$ and $\{k,\ell\}$ are disjoint edges of $\Gamma$, or 
$H_{ij}\cap H_{jk}\cap H_{k i}$, where the edges $\{ij\}$, $\{jk\}$, $\{k i\}$ form 
a triangle in $\Gamma$. In particular, any $2$-flat in $\A(\Gamma)$ has size at most $3$.

\begin{example}
If $\Gamma=K_{m}$ is the complete graph on $m$ vertices, then $\A_{K_{m}}$ is the 
\emph{braid arrangement}\/ and the complement $M(\A_{K_m})$ is the configuration 
space of $m$ distinct points on $\C$. Any graphic arrangement 
$\A_{\Gamma}$ is a sub-arrangement of $\A_{K_{m}}$, where $m=\abs{\mathsf{V}}$.
\end{example}

Let $T_r=T_r(\Gamma)$ be the set of $K_r$-subgraphs ($r$-cliques) of $\Gamma$. Set 
$\kappa_r \coloneqq \abs{T_r(\Gamma)}$; thus $\kappa_1=m$ and 
$\kappa_2=\abs{\mathsf{E}}$. Setting $V^{\vee}=E^1(\A_{\Gamma})$ 
and $K^{\perp}=I^2(\A_{\Gamma})$, we then find that $\dim V^{\vee}=\kappa_2$
and $\dim K^{\perp}=\kappa_3$. Using that each graphic arrangement is a subarrangement of a braid arrangement, 
it follows from \cite{CS-camb}*{Proposition 6.9} that multinets on 
$\A_{\Gamma}$ correspond either to triangles or to complete triangles contained 
in $\Gamma$. A triangle $\{i, j, k\}\in T_3(\Gamma)$ induces a local 
$2$-dimensional local component $P_{ijk}=\spn \{e_{ij}-e_{ik}, e_{ij}-e_{jk}\}$ 
of the resonance $\RR(\A_{\Gamma})$. 

A $4$-clique $\{i, j,k, \ell\}\in T_4(\Gamma)$ yields a subarrangement
$\A_{K_4}=\bigl\{H_{ij}, H_{ik}, H_{i\ell}, H_{jk}, H_{j\ell}, H_{k\ell}\bigr\}$ of $\A_{\Gamma}$. This subarrangement has one essential component $P_{ijk\ell}$ corresponding to the unique $(3,2)$-multinet given by the partition $
\A(K_4)=\{H_{ij}, H_{k\ell}\}\sqcup \{H_{ik}, H_{j\ell}\} \sqcup \{H_{i\ell}, H_{jk}\},$ with all multiplicities equal to $1$. The corresponding $2$-dimensional component of $\RR(\A_{\Gamma})$ is then given by 
$P_{ijk\ell}=\spn\bigl\{{e_{ij}-e_{ik}+e_{k\ell}-e_{j\ell}, e_{ij}-e_{jk}+e_{k\ell}-e_{i\ell}\bigr\}}$.
We summarize these facts:
\begin{proposition}
\label{prop:res-graphic}
    The resonance variety $\mathcal{R}(\A_{\Gamma})$ has the following decomposition into 
irreducible components:
\begin{equation*}
	\label{eq:g-res}
	\RR(\A_{\Gamma})= \bigcup_{\{i,j,k\}\in T_3(\Gamma)} P_{ijk} \cup 
	\bigcup_{\{i,j,k,\ell\}\in T_4(\Gamma)} P_{ijk\ell}.
\end{equation*}
\end{proposition}

We  have the following immediate consequence of Theorem \ref{thm:res-arr-separable}:

\begin{corollary}
\label{thm:sep-res-graph}
The resonance variety $\RR(\A_{\Gamma})$ is strongly isotropic.
\end{corollary}

As a consequence of Theorem \ref{thm:chen-arrs} and Proposition \ref{prop:res-graphic}, 
we obtain the following Chen rank formula for graphic 
arrangement groups. Related results have been obtained in \cite{CS-conm}
in the case when $\Gamma$ is a complete graph $K_n$, in \cite{PS-cmh} 
in the case when $\kappa_4=0$, and in \cite{SS-tams}*{Theorem 3.4}, 
where the following formula is stated without a proof. 

\begin{corollary}
\label{cor:chen-ranks-graphic}
The Chen ranks of every graphic arrangement group $G(\A_{\Gamma})$ are 
given by $\theta_q\bigl(G(\A_{\Gamma})\bigr) = (q-1) (\kappa_3 + \kappa_4)$, 
for all $q\ge \kappa_2-1$.
\end{corollary}

\newcommand{\arxiv}[1]
{\texttt{\href{http://arxiv.org/abs/#1}{arXiv:#1}}}
\newcommand{\arx}[1]
{\texttt{\href{http://arxiv.org/abs/#1}{arxiv:}}
	\texttt{\href{http://arxiv.org/abs/#1}{#1}}}
\newcommand{\arxx}[2]
{\texttt{\href{https://arxiv.org/abs/#1.#2}{arxiv:#1.}}
	\texttt{\href{https://arxiv.org/abs/#1.#2}{#2}}}
\newcommand{\doi}[1]
{\texttt{\href{http://dx.doi.org/#1}{doi:#1}}}
\renewcommand{\MR}[1]
{\href{http://www.ams.org/mathscinet-getitem?mr=#1}{MR#1}}
\bibliographystyle{amsplain}

\begin{thebibliography}{10}

\def\cprime{$'$}

\bibitem{AFPRW2}
Marian Aprodu, Gavril Farkas, \c Stefan Papadima, Claudiu Raicu, 
and Jerzy Weyman,
\href{https://doi.org/10.1007/s00222-019-00894-1}%
{\em Koszul modules and Green's Conjecture}, 
Inventiones Math. \textbf{218} (2019), 657--720.
\MR{4022070}

\bibitem{AFPRW}
Marian Aprodu, Gavril Farkas, \c Stefan Papadima, Claudiu Raicu, and Jerzy Weyman,
\href{https://doi.org/10.1215/00127094-2022-0010}%
{\em Topological invariants of groups and {K}oszul modules}, Duke Math. J.
\textbf{171} (2022), no.~19, 2013--2046.
\MR{4484204}

\bibitem{AFRSS}
Marian Aprodu, Gavril Farkas, Claudiu Raicu,  Alessio Sammartano, 
and Alexander I. Suciu, 
\href{https://dx.doi.org/10.1007/s10801-024-01313-2}%
{\em Higher resonance schemes and Koszul modules of 
simplicial complexes}, J. Algebraic Combinatorics \textbf{59} 
(2024),  787--805.
\MR{4759659}

\bibitem{AFRS}
Marian Aprodu, Gavril Farkas, Claudiu Raicu, and Alexander~I. Suciu,
\href{https://doi.org/10.1515/crelle-2024-0051}%
{\em Reduced resonance schemes and Chen ranks}, 
J. Reine Angew. Math. \textbf{814} (2024), 205--240.
\MR{4793344}

\bibitem{AFRW}
Marian Aprodu, Gavril Farkas, Claudiu Raicu, and Jerzy Weyman,
\href{https://doi.org/10.1007/s00029-023-00912-4}%
{\em Koszul modules with vanishing resonance in algebraic geometry},
Selecta Math. \textbf{30} (2024), Paper No.~24, 33 pp. 
\MR{4705250}

\bibitem{BW} Nero Budur and Botong Wang,
\href{https://dx.doi.org/10.4310/PAMQ.2020.v16.n4.a3}%
{\em Cohomology jump loci of quasi-compact {K}\"{a}hler manifolds},  
Pure Appl. Math. Q. \textbf{16} (2020), 981--999.
\MR{4180236}

\bibitem{CR15} Andre Chatzistamatiou and Kay R\"{u}lling,
\href{http://dx.doi.org/10.1112/S0010437X15007435}%
{\em Vanishing of the higher direct images of the structure sheaf},
Compos. Math. \textbf{151} (2015), 2131--2144.
\MR{3427575}

\bibitem{CFVV} Elisabetta Colombo, Gavril Farkas, Claire Voisin and Alessandro Verra,
\href{https://doi.org/10.46298/epiga.2017.volume1.2602}
{\em Syzygies of {P}rym and paracanonical curves of genus 8},
\'{E}pijournal G\'{e}om. Alg\'{e}brique \textbf{1} (2017), Art. 23,
\MR{3743110}

\bibitem{Ch-ann} Kuo-Tsai Chen,
\href{http://dx.doi.org/10.2307/1969316}%
{\emph{Integration in free groups}}, Annals of Math.  \textbf{54}
(1951), 147--162.
\MR{0042414}

\bibitem{CEFS} Alessandro Chiodo, David Eisenbud, Gavril Farkas and Frank-Olaf Schreyer,
\href{https://doi.org/10.1007/s00222-012-0441-0}%
{\emph{Syzygies of torsion bunldes and the geometry of the level $\ell$ modular variety over $\overline{\mathcal{M}}_g$}}, Inventiones Math. \textbf{194} (2013), 73--118.
\MR{3103256}

\bibitem{Cohen} Daniel C. Cohen,
\href{https://doi.org/10.1090/S0002-9939-09-09858-X}%
{\emph Resonance of basis-conjugating automorphism groups}, 
Proc. Amer. Math. Soc. \textbf{137} (2009), 2835–-2841. 
\MR{MR2506439}

\bibitem{CSc-adv} Daniel C. Cohen and Henry K. Schenck,
\href{https://dx.doi.org/10.1016/j.aim.2015.07.023}%
{\em Chen ranks and resonance},
Advances Math. \textbf{285} (2015), 1--27.
\MR{3406494}
Revised version at \arxiv{1312.3652v3}.

\bibitem{CS-conm}  Daniel~C. Cohen and Alexander~I. Suciu,
\href{https://dx.doi.org/10.1090/conm/181/02029}%
{\em The Chen groups of the pure braid group}, in:
The \v{C}ech centennial:  A conference on homotopy theory,
Contemp. Math., vol.~181, Amer. Math. Soc., Providence, RI,
1995, pp.~45--64.
\MR{1320987}

\bibitem{CS-tams} Daniel~C. Cohen and Alexander~I. Suciu,
\href{https://dx.doi.org/10.1090/S0002-9947-99-02206-0}%
{\em Alexander invariants of complex hyperplane arrangements},
Transactions Amer. Math. Soc. \textbf{351} (1999), 4043--4067.
\MR{1475679}

\bibitem{CS-camb}  Daniel~C. Cohen and Alexander~I. Suciu,
\href{https://dx.doi.org/10.1017/S0305004199003576}%
{\em Characteristic varieties of arrangements},
Math. Proc. Cambridge Phil. Soc. \textbf{127} (1999), 33--53.
\MR{1692519}

\bibitem{DS-plms}  Graham Denham and Alexander~I. Suciu,
\href{http://dx.doi.org/10.1112/plms/pdt058}%
{\em Multinets, parallel connections, and {M}ilnor fibrations
of arrangements}, Proc. London Math. Soc.
\textbf{108} (2014), 1435--1470.
\MR{3218315}

\bibitem{DPS-duke} Alexandru Dimca, \c{S}tefan Papadima, and Alexander I.~Suciu,
\href{https://dx.doi.org/10.1215/00127094-2009-030}%
{\em Topology and geometry of cohomology jump loci},
Duke Math. Journal \textbf{148} (2009), 405--457.
\MR{2527322}

\bibitem{Fa97} Michael Falk,
\href{https://dx.doi.org/10.1007/BF02558471}%
{\em Arrangements and cohomology},
Ann. Combinatorics \textbf{1} (1997), 135--157.
\MR{1629681}

\bibitem{FY} Michael Falk and Sergey Yuzvinsky,
\href{https://dx.doi.org/10.1112/S0010437X07002722}%
{\em Multinets, resonance varieties, and pencils of plane curves},
Compositio Math. \textbf{143} (2007), 1069--1088.
\MR{2339840}

\bibitem{FL} Gavril Farkas and Katharina Ludwig,
\href{https://doi.org/10.4171/JEMS/214}
{\em The Kodaira dimension of the moduli space of Prym varieties},
J. European Math. Soc. \textbf{12} (2010), 755--795.
\MR{2639318}

\bibitem{M2} Daniel Grayson and Michael Stillman,
{\em Macaulay 2, a software system for research in algebraic geometry}. Available at {\url{http://www.macaulay2.com/}}

\bibitem{GL} Mark Green and Robert Lazarsfeld, 
\href{https://dx.doi.org/10.1007/BF01388754}%
{\em On the projective normality of complete linear series on an 
algebraic curve}, Inventiones Math. \textbf{83} (1986),  73--90.
\MR{0813583}

\bibitem{GL91} Mark Green and Robert Lazarsfeld,
\href{https://doi.org/10.1007/BF01388711}%
\emph{Higher obstructions to deforming cohomology groups of line bundles}, 
J. Amer. Math. Soc. \textbf{4} (1991), 87--103.
\MR{MR1076513}

\bibitem{hartshorne} Robin Hartshorne,
\href{https://dx.doi.org/10.1007/978-1-4757-3849-0}%
{\em Algebraic geometry}, Graduate Texts in Mathematics, vol. 52,
Springer-Verlag, New York-Heidelberg, 1977.
\MR{0463157}

\bibitem{JMM} Craig Jensen, Jon McCammond and John Meier,
\href{https://doi.org/10.2140/gt.2006.10.759}
{\em The integral cohomology of the group of loops}, Geometry \& Topology \textbf{10} (2006), 759--784.
\MR{MR2240905}

\bibitem{La} Rob Lazarsfeld,
\href{http://projecteuclid.org/euclid.jdg/1214440116}
{\em Brill-Noether-Petri without degenerations}, J. Differential Geom. \textbf{23} (1986), 299--307.
\MR{852158}

\bibitem{LY00} Anatoly Libgober and Sergey Yuzvinsky,
\href{https://dx.doi.org/10.1023/A:1001826010964}%
{\em Cohomology of {O}rlik--{S}olomon algebras and local
systems}, Compositio Math. \textbf{21} (2000), 337--361.
\MR{1761630}

\bibitem{Massey} William S.~Massey,
\href{https://dx.doi.org/10.1215/S0012-7094-80-04724-9}%
{\em Completion of link modules}, Duke Math. J.
\textbf{47} (1980),  399--420.
\MR{0575904}

\bibitem{OS} Peter Orlik and Louis Solomon,
\href{https://dx.doi.org/10.1007/BF01392549}%
{\em Combinatorics and topology of complements of
hyperplanes}, Invent. Math. \textbf{56} (1980), 167--189.
\MR{0558866}

\bibitem{PS-imrn} Stefan Papadima and Alexander~I. Suciu,
\href{https://dx.doi.org/10.1155/S1073792804132017}%
{\emph{Chen {L}ie algebras}}, Intenational Math. Res. Notices 
\textbf{2004} (2004),  1057--1086.
\MR{2037049}

\bibitem{PS-cmh} Stefan Papadima and Alexander~I. Suciu,
\href{https://dx.doi.org/10.4171/CMH/77}%
{\em When does the associated graded Lie algebra of an arrangement group decompose?}, Comment. Math. Helv. \textbf{81} (2006), 859--875.
\MR{2271225}

\bibitem{PS-crelle} Stefan Papadima and Alexander~I. Suciu,
\href{https://dx.doi.org/10.1515/crelle-2013-0073}%
{\emph{Vanishing resonance and representations of {L}ie algebras}},
J. Reine Angew. Math. \textbf{706} (2015), 83--101.
\MR{3393364}

\bibitem{PS} Irena Peeva and Mike Stillman,
\href{https://doi.org/10.1216/JCA-2009-1-1-159}
{\emph{Open problems on syzygies and {H}ilbert functions}},
J. Commut. Algebra \textbf{1} (2009), 159--195.
\MR{2462384}

\bibitem{PY} Jorge V.~Pereira and Sergey Yuzvinsky,
\href{https://dx.doi.org/10.1016/j.aim.2008.05.014}%
{\em Completely reducible hypersurfaces in a pencil},
Advances  Math.  \textbf{219} (2008), 672--688.
\MR{2435653}

\bibitem{Ryb} G.~Rybnikov, 
\href{https://doi.org/10.1007/s10688-011-0015-8}%
{\em On the fundamental group of the complement of a complex
hyperplane arrangement},  Funct. Anal. Appl.  
\textbf{45} (2011), 137--148. 
\MR{2848779}

\bibitem{Sal01} Simon Salamon,
\href{https://doi.org/10.1016/S0022-4049(00)00033-5}%
{\em Complex structures on nilpotent Lie algebras},
J. Pure Appl. Algebra \textbf{157} (2001), 311--333.

\bibitem{SS-tams02} Henry K.~Schenck and Alexander~I.~Suciu,
\href{http://dx.doi.org/10.1090/S0002-9947-02-03021-0}%
{\em Lower central series and free resolutions of hyperplane 
arrangements}, Transactions Amer. Math. Soc. \textbf{354} (2002), 
 3409--3433. 
\MR{1911506}

\bibitem{SS-tams} Henry K.~Schenck and Alexander~I.~Suciu,
\href{https://dx.doi.org/10.1090/S0002-9947-05-03853-5}%
{\em Resonance, linear syzygies, {C}hen groups, and the
{B}ernstein--{G}elfand--{G}elfand correspondence},
Transactions Amer. Math. Soc. \textbf{358} (2006),  2269--2289.
\MR{2197444}

\bibitem{Su-conm} Alexander~I. Suciu,
\href{https://dx.doi.org/10.1090/conm/276/04510}
{\emph{Fundamental groups of line arrangements: enumerative aspects}},
in: {\em Advances in algebraic geometry motivated by physics} 
({L}owell, {MA}, 2000), 43--79, Contemp. Math., vol. 276, 
Amer. Math. Soc., Providence, RI, 2001.
\MR{1837109}

\bibitem{Su-toul} Alexander~I. Suciu,
\href{https://dx.doi.org/10.5802/afst.1412}%
{\em Hyperplane arrangements and {M}ilnor fibrations},
Ann. Fac. Sci. Toulouse Math. \textbf{23} (2014), 417--481.
\MR{3205599}

\bibitem{Su-decomp}  Aleander~I.~Suciu, 
{\em On the topology and combinatorics of decomposable arrangements}, 
to appear in Contemporary Mathematics, available at 
\arxiv{2404.04784}

\bibitem{SW-mccool} Alexander~I. Suciu and He~Wang,
\href{https://doi.org/10.1016/j.aam.2019.07.004}%
{\em Chen ranks and resonance varieties of the upper McCool groups}, 
Adv. in Appl. Math. \textbf{110} (2019), 197--234. 
\MR{3983125}

\bibitem{Sullivan} Dennis Sullivan, 
\href{https://dx.doi.org/10.1007/BF02684341}%
{\emph{Infinitesimal computations in topology}}, Inst. Hautes \'Etudes Sci.
Publ. Math. (1977),  269--331.
\MR{0646078}

\bibitem{Th} Robert Thompson,
\href{https://doi.org/10.1016/0024-3795(91)90238-R}
{\emph{Pencils of complex and real symmetric and skew  matrices}}, Linear Alg. Appl. \textbf{147} (1991), 323--371.
\MR{1088668}

\bibitem{Uga07} Luis Ugarte,
\href{https://doi.org/10.1007/s00031-005-1134-1}%
{\em Hermitian structures on six--dimensional nilmanifolds},
Transformation Groups, \textbf{12}  (2007), 175--202.

\bibitem{Vo} Claire Voisin,
\href{https://dx.doi.org/10.1007/s100970200042}%
{\em Green's generic syzygy conjecture
for curves of even genus lying on a $K3$ surface}, 
J. European Math. Soc. \textbf{4} (2002), 363--404.
\MR{1941089}

\bibitem{weyman}  Jerzy~Weyman,
\href{https://dx.doi.org/10.1017/CBO9780511546556}%
{\em Cohomology of vector bundles and syzygies},
Cambridge Tracts in Mathematics, vol. 149,
Cambridge University Press, Cambridge, 2003.
\MR{1988690}

\bibitem{Yu09} Sergey Yuzvinsky,
\href{https://dx.doi.org/10.1090/S0002-9939-08-09753-0}%
{\em A new bound on the number of special fibers in a
pencil of curves}, Proc. Amer. Math. Soc. \textbf{137} 
(2009), 1641--1648.
\MR{2470822}

\end{thebibliography}

\end{document}